\documentclass[11pt]{article}

\usepackage[utf8]{inputenc}
\usepackage[T1]{fontenc}
\usepackage[french]{babel}
\usepackage{lmodern}
\usepackage[a4paper]{geometry}
\usepackage{amsmath,amssymb,amsthm}
\usepackage[shortlabels]{enumitem}
\usepackage[all]{xy}

\theoremstyle{plain}
\newtheorem{theo}{Théorème}[section]
\newtheorem{lem}[theo]{Lemme}
\newtheorem{cor}[theo]{Corollaire}
\newtheorem{prop}[theo]{Proposition}
\theoremstyle{definition}
\newtheorem{defi}[theo]{Définition}

\newtheorem{rem}[theo]{Remarque}

\theoremstyle{plain}

\theoremstyle{definition}
\newtheorem*{defi_sans_num}{Définition}
\newtheorem*{rem_sans_num}{Remarque}

\theoremstyle{plain}
\newtheorem{theorem}{Théorème}
\newtheorem{corollary}[theorem]{Corollaire}

\title{Correspondance thêta locale $\ell$-modulaire I : groupe métaplectique, représentation de Weil et $\Theta$-lift}
\author{Justin Trias}
\date{}

\begin{document}

\maketitle

\begin{abstract} Soit $F$ un corps qui est, soit local non archimédien, soit fini, de caractéristique résiduelle $p$ mais toujours de caractéristique différente de $2$. Soit $W$ un espace symplectique de dimension finie sur $F$. On suppose que $R$ est un corps de caractéristique $\ell \neq p$ de sorte qu'il existe un caractère lisse additif non trivial $\psi : F \to R^\times$. On prouve que le théorème de Stone-von Neumann pour le groupe d'Heisenberg $H(W)$ est encore valable pour les représentations à coefficients dans $R$. On construit ainsi une représentation projective du groupe $\textup{Sp}(W)$ qui se relève en une représentation lisse à coefficients dans $R$ d'une extension centrale de $\textup{Sp}(W)$ par $R^\times$ : c'est la représentation de Weil modulaire du groupe métaplectique. On donne ensuite pour toute paire duale $(H_1,H_2)$ dans $\textup{Sp}(W)$ les cas où leurs relevés au groupe métaplectique sont scindés. On calcule enfin le plus grand quotient isotypique de la représentation de Weil modulaire pour définir le $\Theta$-lift et on donne quelques pistes d'études qui s'ouvrent avec les outils construits ici, à savoir l'extension des scalaires et la réduction modulo $\ell$. \end{abstract}

\renewcommand{\abstractname}{Abstract}
\begin{abstract} Let $F$ be a field which is, either local non archimedean, or finite, of residual charcateristic $p$ but of characteristic different from $2$. Let $W$ be a symplectic space of finite dimension over $F$. Suppose $R$ is a field of characteristic $\ell \neq p$ so that there exists a non trivial smooth additive character $\psi : F \to R^\times$. Then the Stone-von Neumann theorem of the Heisenberg group $H(W)$ is still valid for representations with coefficients in $R$. It leads to a projective representation of the group $\textup{Sp}(W)$ which lifts to a genuine smooth representation of a central extension of $\textup{Sp}(W)$ by $R^\times$: this is the modular Weil representation of the metaplectic group. For any dual pair $(H_1,H_2)$, their lifts to the metaplectic group may splitor not according to the different cases at stake. Eventually, computing the biggest isotypic quotient of the modular Weil representation allows to define the $\Theta$-lift. Some new lines of investigation are thus available with these new tools such as studying scalar extension and reduction modulo $\ell$. \end{abstract}

\noindent \textbf{Remerciements :} J'aimerais remercier Alberto M\'inguez et Shaun Stevens pour leur soutien constant et leurs conseils durant la rédaction de cet article, ainsi que Nadir Matringe et Wee Teck Gan d'avoir rapporté mon travail de thèse qui sert de base à ce premier papier. J'ai eu la chance de discuter avec Colette Moeglin, Marie-France Vignéras et Jean-Loup Waldspurger, qui ont éclairé ma compréhension du livre sur la correspondance thêta et au-delà. J'ai également pu compter sur des échanges enrichissants et stimulants avec Anne-Marie Aubert, Gianmarco Chinello, Jean-François Dat, Guy Henniart et Vincent Sécherre.

\section*{Introduction}

La correspondance thêta sur un corps local établit une bijection $\theta$ entre des classes d'équivalences de représentations lisses irréductibles complexes de deux groupes provenant d'une paire duale dans un groupe symplectique. L'interprétation de cette bijection donne des informations arithmétiques profondes sur les représentations de chacun des deux groupes précédents. Par exemple, on peut relier cette bijection aux propriétés analytiques de fonctions L ou bien aux valeurs de facteurs $\varepsilon$ ; ou encore l'interpréter dans certains cas en termes de la correspondance de Langlands. Par ailleurs, M.-F. Vignéras a été la première à étudier les représentations lisses des groupes $p$-adiques à coefficients dans un corps de caractéristique $\ell$ avec $\ell \neq p$. On appelle ces représentations \og $\ell$-modulaires \fg{} pour insister sur la distinction entre $p$, souvent sous-entendu, et $\ell$, qui est différent. Considérer de telles représentations permet de s'intéresser à des questions globales de congruences entre formes automorphes. La correspondance thêta globale est construite à l'aide de sa version locale, \textit{i.e.} celle sur des corps locaux, et entraîne des résultats profonds sur la théorie des formes automorphes. On justifie la motivation du présent travail par la perspective -- lointaine, mais stimulante -- suivante : développer une correspondance thêta locale $\ell$-modulaire pourrait impliquer des résultats inédits de congruence entre formes automorphes tels qu'initiés par M.-F. Vignéras.

Il existe à ce jour deux manières élémentaires \cite{weil,mvw}, qui sont réputées équivalentes, de développer la théorie qui mène à la définition de cette bijection $\theta$. Elles aboutissent toutes deux à définir une représentation remarquable du groupe métaplectique : la représentation de Weil. Pour les besoins de cette introduction, on donne le nom de \og théorie de la représentation de Weil \fg{} à l'ensemble de ces méthodes. La première façon de procéder trouve sa source dans les travaux de Weil \cite{weil} pour les représentations complexes et a été généralisée dans \cite{ct} pour des représentations à coefficients plus généraux. La seconde s'appuie de manière cruciale sur le théorème de Stone-von Neumann et présente l'avantage de fonctionner aussi bien pour les corps finis que pour les corps locaux non archimédiens. C'est cette dernière construction que l'on propose de généraliser à des représentations à coefficients plus généraux dans les Sections \ref{representations_modulaires_section} et \ref{representation_de_weil_mod_section}. On prouve de plus  dans l'Annexe \ref{lien_avec_ct_section} l'équivalence, réputée vraie, entre ces deux méthodes. Il existe également une variante géométrique de cette théorie \cite{shin} qui mène à définir une représentation de Weil, mais la nature radicalement différente des outils qui y sont employés rend toute comparaison particulièrement difficile.

Avant de donner plus de détails, voici les principales contributions qu'apporte ce papier et que l'on va présenter : le théorème de Stone-von Neumann modulaire, qui est une représentation du groupe d'Heisenberg à coefficients dans un corps de caractéristique $\ell$ possédant suffisamment de racines de l'unité ; la construction de la représentation de Weil modulaire du groupe métaplectique et ses modèles explicites ; ses conséquences en direction d'une correspondance thêta $\ell$-modulaire, avec notamment l'équivalence entre correspondance thêta à coefficients dans un corps de base et à coefficients dans une clôture algébrique, mais également des exemples de compatibilité de la correspondance thêta à la réduction modulo $\ell$ sous certaines hypothèses de banalité.

\subsection{Théorie de la représentation de Weil modulaire}

Soit $F$ un corps qui est, soit fini de caractéristique $p$, soit local non archimédien de caractéristique résiduelle $p$ et de caractéristique différente de $2$. Soit $(W,\langle , \rangle)$ un espace symplectique de dimension finie sur $F$ dont on note $\textup{Sp}(W)$ le groupe symplectique associé. Soit $R$ un corps de caractéristique $\ell$. On suppose qu'il existe un caractère lisse non trivial $\psi : F \to R^\times$. Cette condition se traduit en termes de racines de l'unité et la Remarque \ref{remarque_car_R_non_p_si_psi_existe} impose en particulier que $\ell$ est différent de $p$.

\subsubsection*{Théorème de Stone-von Neumann}

Le groupe d'Heisenberg $H$ est le groupe localement profini formé de l'ensemble $W \times F$ muni de la topologie produit et de la loi de groupe :
$$(w,t) \cdot (w',t') = \bigg(w+w',t+t'+\frac{\langle w,w' \rangle }{2}\bigg).$$
On identifie $F$ avec le centre de $H$ via l'isomorphisme de groupes topologiques $t \mapsto (0,t)$, de sorte que $\psi$ soit un caractère du centre de $H$. La catégorie des représentations lisses de $H$ à coefficients dans $R$ est notée $\textup{Rep}_R(H)$.

Quand $R$ est le corps des nombres complexes, le théorème suivant est connu sous le nom de théorème de Stone-von Neumann \cite[Chap. 2, Th. I.2]{mvw}. On le généralise dans la Section \ref{representations_modulaires_section} pour les représentations modulaires \textit{i.e.} à coefficients dans $R$ :

\begin{theorem}[Stone-von Neumann modulaire] À isomorphisme près, il existe une unique représentation irréductible $(\rho_\psi,S)$ dans $\textup{Rep}_R(H)$ dont le caractère central est $\psi$. \end{theorem}

On appelle représentation métaplectique associée à $\psi$
 l'unique classe d'équivalence qui provient du théorème. Un modèle de la représentation métaplectique est un représentant dans $\textup{Rep}_R(H)$ de cette classe d'équivalence. On décrit dans la Section \ref{modeles_rep_metaplectique_section} des modèles explicites $(\rho_\psi,S_A)$ quand $A$ est un sous-groupe auto-dual de $W$.

\subsubsection*{Groupe métaplectique}

Soit $(\rho_\psi,S)$ un modèle de la représentation métaplectique associée à $\psi$. On développe au début de la Section \ref{representation_de_weil_mod_section} les considérations suivantes. Pour tout $g \in \textup{Sp}(W)$, la représentation $(\rho_\psi^g,S)$ définie par $\rho_\psi^g(w,t) = \rho_\psi(g^{-1}w,t)$ pour $(w,t) \in H$ est un modèle de la représentation métaplectique associée à $\psi$. Elle est donc isomorphe à $(\rho_\psi,S)$. Comme $\textup{Hom}_H(\rho_\psi^g,\rho_\psi) \simeq R$, l'action de $\textup{Sp}(W)$ sur $(\rho_\psi,S)$ définit une représentation projective $\sigma_S : \textup{Sp}(W) \to \textup{PGL}_R(S)$, qui se relève en une représentation $\omega_{\psi,S}$ d'une extension centrale de $\textup{Sp}(W)$ par $R^\times$. Celle-ci est obtenue comme la première projection du produit fibré associé au morphisme quotient $\textsc{red}$ et à $\sigma_S$ au-dessus de $\textup{PGL}_R(S)$ :
$$\xymatrix{
\widetilde{\textup{Sp}}_{\psi,S}^R(W) \ar[d]^{p_S} \ar[r]^{\omega_{\psi,S}} & \textup{GL}_R(S) \ar[d]^{\textsc{red}} \\
\textup{Sp}(W) \ar[r]^{\sigma_S} & \textup{PGL}_R(S) 
}$$
où $\widetilde{\textup{Sp}}_{\psi,S}^R(W)$ est l'extension centrale de $\textup{Sp}(W)$ par $R^\times$ en question et $p_S$ est la seconde projection. On prouve dans la Section \ref{gp_met_et_rep_de_Weil_subsection} tout un ensemble de faits concernant ces extensions centrales. Les voici pêle-mêle.

Il existe une collection $\{ \Phi_{S,S'} : \widetilde{\textup{Sp}}_{\psi,S}^R(W) \to \widetilde{\textup{Sp}}_{\psi,S'}^R(W) \}_{S,S'}$ d'isomorphismes d'extensions centrales qui identifient canoniquement l'ensemble des extensions centrales précédentes. Excepté dans le cas $F = \mathbb{F}_3$ et $\textup{dim}_F W = 2$, cette collection est l'unique collection d'isomorphismes d'extensions centrales possibles. Son groupe dérivé $\widehat{\textup{Sp}}_{\psi,S}^R(W)$ est : le groupe dérivé de $\textup{Sp}(W)$ si $F$ est fini ou $\ell = 2$ ; l'unique extension centrale non triviale de $\textup{Sp}(W)$ par $\{ \pm 1 \}$ sinon.

En munissant $S$ de la topologie discrète et $\textup{GL}_R(S)$ de la topologie compacte-ouverte, le diagramme précédent est celui du produit fibré dans la catégorie des groupes topologiques associé aux morphismes de groupes continus $\sigma_S$ et $\textsc{red}$. Ainsi, le morphisme de groupes topologiques $p_S$ munit :
\begin{itemize}[label=$\circ$]
\item le groupe $\widetilde{\textup{Sp}}_{\psi,S}^R(W)$ d'une structure de revêtement topologique de base $\textup{Sp}(W)$ et de fibre (discrète) $R^\times$ ;
\item le sous-groupe ouvert $\widehat{\textup{Sp}}_{\psi,S}^R(W)$ de $\widetilde{\textup{Sp}}_{\psi,S}^R(W)$ d'une structure de revêtement topologique de base $\textup{Sp}(W)$ et de fibre $\{ \pm 1\}$ \end{itemize}
En particulier, les groupes $\widetilde{\textup{Sp}}_{\psi,S}^R(W)$ et $\widehat{\textup{Sp}}_{\psi,S}^R(W)$ sont localement profinis et la représentation $\omega_{\psi,S}$ est lisse. Cette représentation est même admissible.

\begin{defi_sans_num} La classe d'isomorphisme des extensions centrales topologiques $\widetilde{\textup{Sp}}_{\psi,S}^R(W)$, où $S$ parcourt les modèles de la représentation métaplectique associée à $\psi$, est appelée groupe métaplectique associé à $\psi$. Un modèle du groupe métaplectique est simplement un représentant de cette classe d'isomorphisme. \end{defi_sans_num}

\begin{rem_sans_num} L'Annexe \ref{lien_avec_ct_section} donne une traduction explicite entre cette construction du groupe métaplectique et celle de \cite{weil,ct}. \end{rem_sans_num}

\subsubsection*{Représentation de Weil} Soit $\textup{Mp}(W)$ un modèle du groupe métaplectique associé à $\psi$. Pour tout modèle $(\rho_\psi,S)$ de la représentation métaplectique associée à $\psi$, il existe par définition un isomorphisme $\varphi_S$ d'extensions centrales topologiques :
$$\varphi_S : \textup{Mp}(W) \to \widetilde{\textup{Sp}}_{\psi,S}^R(W).$$
Qui plus est, un tel isomorphisme est unique sauf si $F = \mathbb{F}_3$ et $\textup{dim}_F W =2$. On choisit un système d'isomorphismes $\varphi_S$ pour tout modèle et on définit :

\begin{defi_sans_num} La représentation de Weil modulaire associée à $\psi$ est la classe d'isomorphisme des représentations $(\omega_{\psi,S} \circ \varphi_S,S)$ du groupe $\textup{Mp}(W)$. Un modèle de cette représentation de Weil modulaire est un représentant de cette classe d'isomorphisme. \end{defi_sans_num}

Quelques commentaires sur cette représentation de Weil modulaire : il subsiste une ambiguïté quand $F=\mathbb{F}_3$ et $\textup{dim}_F(W)=2$ puisque cette classe d'isomorphisme dépend \textit{a priori} du choix du système d'isomorphisme $\{ \varphi_S \}$ ; le cas que l'on vient de mentionner est désigné comme exceptionnel et ne se produit évidemment pas quand $F$ est local non archimédien. Peu importe, la représentation de Weil modulaire est lisse admissible ; et
sauf dans le cas exceptionnel, il suffit de connaître l'action du groupe dérivé $\widehat{\textup{Mp}}(W)$ de $\textup{Mp}(W)$ pour  étudier la représentation de Weil modulaire.

\subsection{Vers une correspondance thêta locale modulaire}

Soit $F$ un corps local non archimédien de caractéristique résiduelle $p$ et de caractéristique différente de $2$. Soit $W$ un espace symplectique de dimension finie sur $F$.

\subsubsection*{Correspondance thêta locale classique}

Soit $\psi : F \to \mathbb{C}^\times$ un caractère lisse non trivial. Le groupe métaplectique classique $\textup{Mp}(W)$ est l'extension centrale de $\textup{Sp}(W)$ par $\mathbb{C}^\times$ décrite plus haut. Pour tout sous-groupe $H$ de $\textup{Sp}(W)$, on note $\widetilde{H}$ l'image réciproque de $H$ dans $\textup{Mp}(W)$. Avec ces conventions $\widetilde{\textup{Sp}}(W) = \textup{Mp}(W)$. En notant $i$ l'inclusion naturelle de $\mathbb{C}^\times$ dans $\textup{Mp}(W)$, on définit la catégorie $\textup{Rep}_\mathbb{C}^{\textup{gen}}(\widetilde{H})$ des représentations lisses complexes $(\pi,V)$ de $\widetilde{H}$ telle que $\pi \circ i(\lambda) = \lambda \times \textup{Id}_V$ pour tout $\lambda \in \mathbb{C}^\times$. L'ensemble de ses classes d'équivalence de représentations irréductibles est $\textup{Irr}_{\mathbb{C}}^{\textup{gen}}(\widetilde{H})$.

Soit $(\omega_\psi,S) \in \textup{Rep}_\mathbb{C}^{\textup{gen}}(\textup{Mp}(W))$ un modèle de la représentation de Weil complexe associée à $\psi$. Une paire duale réductive irréductible $(H_1,H_2)$ dans le groupe symplectique $\textup{Sp}(W)$ est : 
\begin{enumerate}[label=$\circ$]
\item soit constituée de deux groupes classiques, auquel cas elle est dite de type I ;
\item soit de deux groupes linéaires, auquel cas elle est dite de type II. 
\end{enumerate}
En particulier, on a un morphisme de groupes $\widetilde{H}_1 \times \widetilde{H}_2 \to \textup{Mp}(W)$ donné par la multiplication qui permet de considérer par restriction $(\omega_\psi,S)$ comme une représentation du produit $\widetilde{H}_1 \times \widetilde{H}_2$. Pour toute représentation irréductible $\pi_1$ dans $\textup{Rep}_\mathbb{C}^{\textup{gen}} (\widetilde{H}_1)$, le plus grand quotient $\pi_1$-isotypique de $\omega_\psi$ est de la forme $\pi_1 \otimes_{\mathbb{C}} \Theta(\pi_1)$ avec $\Theta(\pi_1) \in \textup{Rep}_{\mathbb{C}}^{\textup{gen}}(\widetilde{H}_2)$. Dans le sens contraire de même, toute représentation irréductible $\pi_2$ dans $\textup{Rep}_{\mathbb{C}}^{\textup{gen}}(\widetilde{H}_2)$ définit une représentation $\Theta(\pi_2) \in \textup{Rep}_{\mathbb{C}}^{\textup{gen}}(\widetilde{H}_1)$ telle que le plus grand quotient $\pi_2$-isotypique de $\omega_\psi$ est $\Theta(\pi_2) \otimes_{\mathbb{C}} \pi_2$.

Le résultat suivant constitue le c\oe ur de la théorie et a nécessité le concours de différents mathématiciens pour être prouvé dans sa plus grande généralité. On le présente avant de donner un panorama succinct de l'ensemble de ses contributeurs :

\begin{theo} \label{correspondance_theta_classique_thm} Pour toutes représentations irréductibles $\pi_1$ et $\pi_1'$ dans $\textup{Rep}_\mathbb{C}^{\textup{gen}} (\widetilde{H}_1)$ :
\begin{enumerate}[label=\textup{\alph*)}]
\item soit $\Theta(\pi_1)$ est nulle, soit $\theta(\pi_1)$ est irréductible ;
\item quand $\Theta(\pi_1) \neq 0$, on a $\theta(\pi_1) \simeq \theta(\pi_1')$ si et seulement $\pi_1 \simeq \pi_1'$. \end{enumerate} \end{theo}

Plus généralement, le Théorème \ref{correspondance_theta_classique_thm} est vrai pour toute paire duale réductive $(H_1,H_2)$ dans $\textup{Sp}(W)$ d'après \cite[Chap. 2, Rem. 2]{mvw}. Il est important de souligner que les représentations ainsi définies grâce à $\Theta$ et $\theta$ dépendent du choix du caractère lisse non trivial $\psi : F \to \mathbb{C}^\times$.

\begin{defi} Soit $(H_1,H_2)$ une paire duale réductive dans $\textup{Sp}(W)$. La correspondance thêta locale classique sur $F$ (associée à $\psi$) est définie comme la bijection induite par $\theta$ entre les ensembles : 
$$\{ \pi_1 \in \textup{Irr}_{\mathbb{C}}^{\textup{gen}} (\widetilde{H}_{1,S}) \ | \ \Theta(\pi_1) \neq 0\} \overset{\theta}{\simeq} \{ \pi_2 \in \textup{Irr}_{\mathbb{C}}^{\textup{gen}} (\widetilde{H}_{2,S}) \ | \ \Theta(\pi_2) \neq 0\}.$$ \end{defi}

\paragraph{Contributeurs de la correspondance thêta locale classique.} On donne un résumé succinct des contributions qui ont permis de prouver le Théorème \ref{correspondance_theta_classique_thm}, dont la résolution complète a nécessité près de 40 ans.

C'est R. Howe \cite{howe} qui a d'abord conjecturé le Théorème \ref{correspondance_theta_classique_thm} sous la forme que l'on présente, en introduisant de manière systématique la notion de paires duales réductives et en développant les travaux de A. Weil \cite{weil}. Quand, par le passé, ce théorème n'était pas encore démontré, il portait le nom de conjecture de dualité locale de Howe. Cette conjecture de dualité était connue au moins pour les représentations non ramifiées \cite[7-10]{howe}. Un premier succès significatif en direction d'une résolution plus générale est attribué à R. Howe lui-même \cite[Chap. 5]{mvw} et provient d'une conférence qu'il a donnée en 1984 : il prouve à cette occasion que, quand la paire duale $(H_1,H_2)$ est de type I non ramifiée et que la caractéristique résiduelle $p$ de $F$ est différente de $2$, le Théorème \ref{correspondance_theta_classique_thm} est vrai. Cette preuve met en jeu un modèle explicite de la représentation de Weil, le modèle latticiel, ce qui explique la restriction $p \neq 2$ apparaissant sur la caractéristique résiduelle.

En s'inspirant des idées de R. Howe sur la théorie invariante,  S. Rallis a prouvé que pour toute paire duale $(H_1,H_2)$ symplectique-orthogonale avec $F$ de caractéristique $0$ et satisfaisant aux conditions de \cite[Cor. to Th. II.4.1]{rallis}, le Théorème \ref{correspondance_theta_classique_thm} était valide pour toutes représentations unitaires irréductibles $\pi_1$ et $\pi_1'$. L'argument crucial de cette preuve repose sur l'utilisation de la méthode dite de doubling, qui permet de ramener \cite[Prop. II.3.1] {rallis} les problèmes à la compréhension des coinvariants, pour une certaine paire duale $(H_1',H_2')$ dans $\textup{Sp}(\mathbb{W})$ avec $\mathbb{W} = W + (-W)$, de la représentation de Weil associée à $\psi$ dans $\textup{Rep}_\mathbb{C} (\textup{Mp}(\mathbb{W}))$.

Dans le prolongement de ces avancées, S. Kudla \cite{kudla_invent} a donné dans le cas des paires symplectiques-orthogonales une preuve partielle du point a) du Théorème \ref{correspondance_theta_classique_thm}. Quand la représentation $\pi_1$ est (super)cuspidale et la caractéristique de $F$ est $0$, il prouve le résultat plus fort suivant : la représentation $\Theta(\pi_1)$ est irréductible ou nulle. Ses arguments reposent sur l'utilisation d'un autre modèle explicite de la représentation de Weil : le modèle de Schr\"odinger mixte. Il s'inspire aussi du travail de Rallis en introduisant des méthodes récurrentes \cite[5]{kudla_invent} pour le calcul des foncteurs de Jacquet de la représentation de Weil, qui donne accès au support cuspidal des quotients irréductibles de $\Theta(\pi_1)$ pour toute représentation $\pi_1$ irréductible. 

Le livre \cite{mvw} est issu d'un séminaire, qui s'est tenu en 1985 et 1986, et qui était consacré aux représentations métaplectiques sur un corps $p$-adique. Il y est fait état des travaux de R. Howe, S. Rallis, S. S. Kudla, ainsi que R. Rao \cite{rao}\footnote{Ce travail a donné lieu à une publication tardive et existait bel et bien comme prépublication à l'époque de ce séminaire. Il est concomitant bien qu'indépendant à celui de P. Perrin.} et P. Perrin \cite{perrin} sur les cocycles métaplectiques. Plus qu'un simple exposé de ces travaux, les auteurs généralisent les méthodes employées en éliminant au maximum les restrictions sur la caractéristique de $F$ ou sur la paire considérée. Par exemple ils généralisent, en suivant des intuitions attribuées à J.-L. Waldspurger \cite[Chap.3, IV.1]{mvw}, la méthode de doubling, l'irréductibilité pour les représentations cuspidales et le calcul du support cuspidal indépendamment de la caractéristique de $F$ ou de la paire duale de type I considérées. Ce travail de mise en forme et de \og polissage \fg{} \cite[Intro]{mvw} de la théorie apporte également des contributions radicalement nouvelles. En effet, les résultats de \cite[Chap. 4, II.1]{mvw} pour les groupes classiques non quaternioniques, maintenant connus sous le nom d'involution de MVW, présentent un intérêt plus large pour l'étude des représentations des groupes $p$-adiques.

Quand $p \neq 2$, il existe une preuve du Théorème \ref{correspondance_theta_classique_thm} qui généralise à toute paire duale de type I la preuve pour les paires duales de type I non ramifiées. L'idée de départ, plutôt simple, trouve son origine chez R. Howe. Cependant, sa réalisation concrète a nécessité la traduction technique relativement complexe de J.-L. Waldspurger \cite{wald}. Dans ce travail, l'auteur appelle de ses v\oe ux à la mise au point d'une \og démonstration plus claire et moins technique \fg{}. On signale par ailleurs dans le même ouvrage une preuve alternative de R. Howe pour les paires non ramifiées, avec toujours la restriction $p \neq 2$.

Depuis les années 80 cependant, le Théorème \ref{correspondance_theta_classique_thm} était réputé vrai pour les paires duales de type II d'après des arguments de R. Howe -- qui sont restés longtemps non publiés mais aujourd'hui retranscrits dans \cite[Ann. A]{minguez_thesis}. Plus récemment par ailleurs, une nouvelle preuve de ce résultat a été élaborée par A. M\'inguez \cite{minguez_howe} et repose sur des idées inédites jusqu'alors. L'espoir fécond de traduire les idées de A. M\'inguez pour les paires duales de type I a été réalisé par W. T. Gan et S. Takeda \cite{gt}. Ils démontrent le Théorème \ref{correspondance_theta_classique_thm} pour toute paire duale de type I non quaternioniques, avec le double avantage de n'avoir aucune restriction sur $p$ et de présenter une preuve beaucoup plus simple que celles jusqu'alors. La restriction sur la paire duale de type I considérée repose sur l'utilisation de l'involution de MVW évoquée pus haut. Néanmoins, il existe une identité similaire pour les paires duales de type I quaternioniques \cite{sun_zhu}, qui a permis à W. T. Gan et B. Y. Sun \cite{gan_sun} de prouver le Théorème \ref{correspondance_theta_classique_thm} dans le cas restant. Les deux arguments \cite{gt} et \cite{gan_sun} font une utilisation extensive des filtrations de Rallis et Kudla, déjà généralisées à toute paire duale de type I dans \cite[Chap. 3, IV]{mvw}.

\subsubsection*{Représentations $\ell$-modulaires et correspondance thêta}

Pour développer un analogue du Théorème \ref{correspondance_theta_classique_thm} pour les représentations modulaires \textit{i.e.} à coefficients dans $\overline{\mathbb{F}_\ell}$ avec $\ell \neq p$, il est nécessaire de se donner une représentation lisse $(\omega_\psi,S)$, où $\psi : F \to \overline{\mathbb{F}_\ell}$ est un caractère lisse non trivial, dont on va étudier les plus grands quotients isotypiques -- que l'on désigne encore par \og $\Theta$ \fg{}. Deux stratégies s'offrent alors :
\begin{enumerate}[label=$\circ$]
\item la première consiste à examiner les formules explicites qui définissent cette représentation -- par exemple les modèles de Schrödinger et latticiels -- dans le cas complexe, et remplacer partout, quand cela est possible, le corps $\mathbb{C}$ par $\overline{\mathbb{F}_\ell}$ ;
\item la seconde, qui a été poursuivie dans ce travail, reconstruit la théorie en partant de zéro pour définir une représentation de Weil modulaire et les formules explicites afférentes de la Section \ref{decalque_des_rep_S_A_section}.
\end{enumerate}

Bien évidemment, ces deux options se révèlent équivalentes. Quand $(H_1,H_2)$ est une paire duale de type II, la représentation de Weil classique peut être considérée comme appartenant à $\textup{Rep}_\mathbb{C}(H_1 \times H_2)$ car les relevés de $H_1$ et $H_2$ sont scindés dans le groupe métaplectique. Pour cette raison, A. M\'inguez étudia, selon la première stratégie, un analogue ad hoc modulaire $(\omega_\psi,S) \in \textup{Rep}_{\overline{\mathbb{F}_\ell}}(H_1 \times H_2)$ de la représentation de Weil classique. On dit que $\ell$ est banal vis-à-vis de $(H_1,H_2)$ si $\ell$ ne divise pas les pro-ordres de $H_1$ et $H_2$. Dans les notations ici présentes, le résultat principal de \cite{minguez_thesis} est :

\begin{theo} \label{correspondance_theta_modulaire_type_II_thm} Soit $(H_1,H_2)$ une paire réductive duale de type II dans $\textup{Sp}(W)$. On suppose que $\ell$ est banal vis-à-vis de $(H_1,H_2)$. Alors pour toutes représentations irréductibles $\pi_1$ et $\pi_1'$ dans $\textup{Rep}_{\overline{\mathbb{F}_\ell}} (H_1)$ :
\begin{enumerate}[label=\textup{\alph*)}]
\item soit $\Theta(\pi_1)$ est nulle, soit $\theta(\pi_1)$ est irréductible ;
\item quand $\Theta(\pi_1) \neq 0$, on a $\theta(\pi_1) \simeq \theta(\pi_1')$ si et seulement $\pi_1 \simeq \pi_1'$. \end{enumerate} \end{theo}

Maintenant que la théorie de la représentation de Weil modulaire a été développée, il est loisible de considérer la restriction de la représentation de Weil modulaire aux relevés de toute paire réductive duale $(H_1,H_2)$. La situation est la suivante. Soit $(\omega_\psi,S)$ la représentation de Weil modulaire associée à $\psi$. D'après la Section \ref{representation_de_weil_mod_section}, c'est une représentation du groupe $\textup{Mp}(W)$, qui est une extension centrale de $\textup{Sp}(W)$ par $\overline{\mathbb{F}_\ell}^\times$.

Soient $\widetilde{H}_1$ et $\widetilde{H}_2$ les relevés de $H_1$ et $H_2$ dans le groupe métaplectique \textit{i.e.} les images réciproques de ces groupes par la projection canonique $p : \textup{Mp}(W) \to \textup{Sp}(W)$.

\begin{defi_sans_num}[Correspondance thêta $\ell$-modulaire] Soient $\pi_1$ et $\pi_1'$ deux représentations irréductibles dans $\textup{Rep}_{\overline{\mathbb{F}_\ell}}^{\textup{gen}} (\widetilde{H}_1)$. On considère les assertions suivantes :
\begin{enumerate}
\item[($\Theta_1$)] la représentation $\Theta(\pi_1)$ est de longueur finie, donc admet un co-socle noté $\theta(\pi_1)$ ;
\item[($\Theta_2$)]  soit $\Theta(\pi_1)$ est nulle, soit $\theta(\pi_1)$ est irréductible ;
\item[($\Theta_3$)] quand $\Theta(\pi_1) \neq 0$, on a $\theta(\pi_1) \simeq \theta(\pi_1')$ si et seulement $\pi_1 \simeq \pi_1'$. 
\end{enumerate}
Quand ces trois énoncés sont valides pour tout $\pi_1$ et $\pi_1'$, la \og correspondance thêta locale $\ell$-modulaire sur $F$ \fg{} est alors définie comme la bijection induite par $\theta$ entre les ensembles : 
$$\{ \pi_1 \in \textup{Irr}_R^{\textup{gen}} (\widetilde{H}_1) \ | \ \Theta(\pi_1) \neq 0\} \overset{\theta}{\simeq} \{ \pi_2 \in \textup{Irr}_R^{\textup{gen}} (\widetilde{H}_2) \ | \ \Theta(\pi_2) \neq 0\}.$$ \end{defi_sans_num}

\begin{rem_sans_num} On ajoute l'hypothèse de longueur finie pour que le co-socle soit bien défini. En particulier, il en résulte l'existence d'un quotient irréductible quand $\Theta(\pi_1) \neq 0$. Dans le cas complexe, cela est vrai en vertu de \cite[Chap. 3, Th. IV.4]{mvw}. C'est en revanche moins clair dans le cas $\ell$-modulaire en toute généralité, même si cela reste vrai quand $\ell$ est banal vis-à-vis de $(H_1,H_2)$ d'après \cite[Sec. 5.2]{trias_thesis}. \end{rem_sans_num}

\begin{rem_sans_num} On s'attend à ce que la correspondance thêta locale $\ell$-modulaire sur $F$ \textit{i.e.} que l'ensemble des énoncés ($\Theta_1$)-($\Theta_2$)-($\Theta_3$) soient vrais pour tout $\pi_1$ et $\pi_1'$ à condition que $\ell$ soit banal vis-à-vis de $(H_1,H_2)$. Dans le cas non banal, il existe des contre-exemples pour les paires de type II \cite[Sec. 4.5.2]{minguez_thesis} et de type I \cite[Sec. 7.2]{trias_thesis}. \end{rem_sans_num}

\subsubsection*{En direction d'une correspondance thêta locale modulaire}

Dans cet article, on ne cherche pas à prouver que les énoncés ($\Theta_1$)-($\Theta_2$)-($\Theta_3$) sont vrais pour tout $\pi_1$ et $\pi_1'$ quand $(H_1,H_2)$ est une paire duale de type I et $\ell$ est banal vis-à-vis de $(H_1,H_2)$. Cela dépasse largement le cadre de ce travail et  nécessiterait d'adapter aux représentations modulaires l'ensemble des résultats sur les filtrations de Rallis et Kudla \cite[Chap. 3, IV]{mvw}, ainsi que l'involution de MVW \cite[Chap. 4]{mvw} et les travaux de \cite{gt} et \cite{gan_sun}. Ces problèmes feront par ailleurs l'objet d'une série d'articles dans la continuité de celui-ci.

On donne cependant quelques applications des objets que l'on a développés. Soit $R$ un corps parfait de caractéristique $\ell \neq p$. Quitte à prendre en compte l'action de l'algèbre d'endomorphismes de la représentation $\pi_1$, il existe (cf. Section \ref{compatibilité_extension_des_scalaires_section}) des énoncés analogues ($\Theta_1'$)-($\Theta_2'$)-($\Theta_3'$) pour $R$. Cela complique quelque peu la situation, mais seulement en apparence : les trois énoncés précédents sur $R$ et ceux définis sur la clôture algébrique $\bar{R}$, notés ($\Theta_1$)-($\Theta_2$)-($\Theta_3$), sont compatibles au sens des trois Propositions \ref{extension_des_scal_THETA1_prop}, \ref{extension_des_scal_THETA2_prop} et \ref{extension_des_scal_THETA3_prop}. Finalement, on obtient le résultat de compatibilité du Théorème \ref{correspondance_sur_un_corps_non_alg_clos_equiv_thm} que l'on résume ainsi :

\begin{theorem} La correspondance thêta locale sur $F$ à coefficients dans $R$ est valide si et seulement si la correspondance thêta locale sur $F$ à coefficients dans $\bar{R}$ l'est. \end{theorem}

Par conséquent, il est suffisant de considérer la situation où $R$ est un corps algébriquement clos. Néanmoins, l'exemple qui clôt la Section \ref{compatibilité_extension_des_scalaires_section} met en lumière une certaine compatibilité à l'action de Galois pour la descente depuis un corps algébriquement clos.

\paragraph{Un argument de réduction.} Pour terminer, on donne dans la Section \ref{compatibilite_a_la_reduction_type_I_section} des idées de compatibilité à la réduction modulo $\ell$ pour tirer des résultats de la correspondance thêta classique. Le même argument semble fonctionner pour les paires de type II, mais on le présente pour une paire duale $(H_1,H_2)$ de type I.

Soit $W(\overline{\mathbb{F}_\ell})$ l'anneau des vecteurs de Witt, dont on note $K$ le corps des fractions et dont l'idéal maximal est engendré par $\ell$. D'après \cite[II.4.12]{vig}, une représentation absolument irréductible cuspidale $\Pi_1$ dans $\textup{Rep}_K^{\textup{gen}}(\widetilde{H}_1)$ est entière : son caractère central est à valeurs dans $W(\overline{\mathbb{F}_\ell})^\times$ puisque le centre de $H_1$ est compact. La compacité du centre de $H_1$ entraîne également que $\Pi_1$ est un objet projectif -- et injectif -- dans $\textup{Rep}_K^{\textup{gen}}(\widetilde{H}_1)$.

Ensuite, si l'on suppose que $\ell$ est banal vis-à-vis $H_1$ et que $H_1$ admet des sous-groupes discrets co-compacts\footnote{Cette condition est toujours satisfaite quand $F$ est de caractéristique $0$. On indique la Remarque \ref{reduction_mod_ell_longue_rem} pour une discussion plus poussée sur ce point.}, tout $W(k)$-réseau stable de $\Pi_1$ se réduit en une représentation irréductible cuspidale d'après \cite[Lem. 6.8 \& Cor. 6.10]{dat_nu}. Une telle réduction ne dépend pas du choix du $W(k)$-réseau en question car elle est unique à isomorphisme près d'après le principe de Brauer-Nesbitt \cite[I.9.6]{vig}. On la note $\pi_1$. Enfin, la représentation $\pi_1$ est encore un objet projectif -- et injectif -- de $\textup{Rep}_{\overline{\mathbb{F}_\ell}}(\widetilde{H}_1)$ puisque le centre de $H_1$ est compact et $\ell$ est banal vis-à-vis de $H_1$.

\begin{theorem} On suppose que $H_1$ admet des sous-groupes discrets co-compact, que $\ell$ est banal vis-à-vis de $H_1$ et que $F$ est de caractéristique résiduelle $p$ impaire. Soit $\Pi_1$ une représentation absolument irréductible cuspidale dans $\textup{Rep}_K^{\textup{gen}}(\widetilde{H}_1)$. Il existe alors un $W(\overline{\mathbb{F}_\ell})$-réseau stable $L_{\Theta(\Pi_1)}$ de $\Theta(\Pi_1)$ tel que les semi-simplifiées dans $\textup{Rep}_{\overline{\mathbb{F}_\ell}}^{\textup{gen}}(\widetilde{H}_2)$ des représentations :
$$\mathfrak{r}_\ell(L_{\Theta(\Pi_1)}) = L_{\Theta(\Pi_1)} / \ell L_{\Theta(\Pi_1)} \textup{ et } \Theta(\pi_1)$$
soient isomorphes. En particulier $\Theta(\pi_1)$ est de longueur finie. \end{theorem}

La restriction sur la caractéristique résiduelle provient de l'utilisation d'un argument de théorie des types \cite{kurinczuk_stevens} pour montrer, lors du Lemme \ref{idempotent_central_Pi1_coeff_W(k)_lem}, que l'idempotent central $e_{\Pi_1}$ du centre de Bernstein associé à $\Pi_1$ est bien à coefficients dans $W(\overline{\mathbb{F}_\ell})$. On déduit du théorème précédent :

\begin{corollary} On reprend les hypothèses précédentes et on suppose de plus que $\ell$ est banal vis-à-vis de $H_2$. Alors si la représentation $\Theta(\Pi_1)$ est cuspidale non nulle, la représentation $\Theta(\pi_1)$ est irréductible. \end{corollary}

Bien que que l'hypothèse $\Theta(\Pi_1)$ cuspidale semble sortie de nulle part, elle ne le sera surement pas pour les lecteurs familiers de la correspondance thêta. En effet, ce cas est même le plus fondamental quand on étudie les propriétés de $\Theta(\Pi_1)$ avec $\Pi_1$ cuspidale, puisqu'il correspond à l'indice de première occurrence au sens de \cite{kudla_invent}.

\subsection{Contenu}

Les préliminaires de la Section \ref{preliminaires_section} sont l'occasion de définir les notations, ainsi que de poser les bases du facteur de Weil non normalisé, qui est une modification nouvelle du facteur de Weil usuel. On se sert de ce facteur dans la Section \ref{expression_du_cocycle_perrin_like_section} pour donner une interprétation plus géométrique du cocycle métaplectique. Ces considérations sont d'une nature assez technique et n'entravent en rien la compréhension du reste du texte à condition de les admettre.

Les ossatures des Sections \ref{representations_modulaires_section} et \ref{representation_de_weil_mod_section} sont similaires à celles de \cite[Chap. 2, I]{mvw} et \cite[Chap. 2, II]{mvw}. On a repris et généralisé le propos, aussi bien pour insister sur les différences avec le cas complexe que par un souci de donner un texte aussi auto-contenu que possible. On l'espère détaillée et aussi claire que \cite[Chap. 2]{mvw}. On décrit précisément néanmoins la topologie du groupe métaplectique dans la partie \ref{gp_met_et_rep_de_Weil_subsection} en établissant explicitement, dans l'Annexe \ref{lien_avec_ct_section}, le lien habituellement présumé avec les travaux \cite{weil} et \cite{ct}. Un point nouveau est quand la caractéristique du corps de base est $2$ : le groupe métaplectique est alors scindé. La Section \ref{expression_du_cocycle_perrin_like_section} apporte un nouvel éclairage sur la signification du cocycle métaplectique et de la section qui en est à l'origine. On interprète \cite{rao} à l'aide du facteur de Weil non normalisé que l'on a défini dans les préliminaires.

La Section \ref{releves_de_paires_duales_scindages_section} constitue une généralisation de \cite{kudla} sur les scindages des relevés de paires duales. Il y a cependant des différences notables avec la théorie complexe, au sens où les problèmes de scindage dépendent du corps de base que l'on considère. La condition en question est liée à une condition sur une racine carrée du cardinal résiduel.

Pour terminer, la dernière section définit le $\Theta$-lift d'une représentation irréductible d'une paire duale $(H_1,H_2)$. Le corps des complexes étant algébriquement clos, on introduit ici un $\Theta$-lift sur des corps qui ne sont plus nécessairement algébriquement clos. Cela donne lieu à des considérations plus techniques puisqu'il faut prendre en compte l'action d'un anneau des endomorphismes. L'Annexe \ref{representations_un_produit_groupes_section} est l'occasion de développer proprement les points dont on a besoin sur l'extension des scalaires pour les représentations lisses des groupes réductifs, ainsi que ses liens avec la construction du plus grand quotient isotypique. On prouve lors du Théorème \ref{correspondance_sur_un_corps_non_alg_clos_equiv_thm} que quand le corps de base est un corps parfait, il est équivalent de considérer les énoncés ($\Theta_1'$)-($\Theta_2'$)-($\Theta_3'$) sur ce corps que les énoncés ($\Theta_1$)-($\Theta_2$)-($\Theta_3$) sur une clôture algébrique. Enfin, la Section \ref{compatibilite_a_la_reduction_type_I_section} donne des idées pour la compatibilité à la réduction des scalaires, dans l'espoir de transférer des résultats de la correspondance thêta classique au cas modulaire.

\tableofcontents

\section{Préliminaires} \label{preliminaires_section}

\subsection{Notations et conventions}

Dans tout cet article, $F$ désigne un corps qui est, soit fini de caractéristique $p$, soit local non archimédien de caractéristique résiduelle $p$. Néanmoins, on suppose que la caractéristique de $F$ est toujours différente de $2$. Le corps résiduel de $F$ étant fini de caractéristique $p$, on note $q$ son cardinal, qui est une puissance de $p$. La norme $| \cdot |_F$ sur $F$ est, la norme triviale quand $F$ est fini, la norme  $p$-adique normalisée quand $F$ est local non archimédien. Elle est normalisée au sens où la norme de toute uniformisante de $F$ est $q^{-1}$. L'anneau des entiers de $F$ est $\mathcal{O}_F$, où par convention $\mathcal{O}_F$ est le corps $F$ tout entier dans le cas fini.

On réservera $(W,\langle , \rangle )$ pour désigner un espace symplectique de dimension finie $n$ sur $F$ muni de son produit symplectique $\langle , \rangle $. Pour tout sous-espace totalement isotrope $X$ dans $W$, on note $P(X)$ le stabilisateur de $X$ dans $\textup{Sp}(W)$. C'est un parabolique maximal de $\textup{Sp}(W)$ pour lequel $N(X)$ désignera son radical unipotent et $M(X)$ un sous-groupe de Levi. Un lagrangien de $W$ est un sous-espace totalement isotrope maximal, ou autrement dit, un sous-espace totalement isotrope de dimension $m$ avec $n=2m$. Deux lagrangiens $X$ et $Y$ de $W$  donne une polarisation complète si $W = X +Y$. Dans ce cas $Y \simeq X^*$ via $y \mapsto \langle y , \cdot \rangle$, ce qui induit une dualité $a \in \textup{GL}_F(X) \mapsto a^* \in \textup{GL}_F(Y)$ ainsi que $c \in \textup{Hom}_F(X,Y) \mapsto  c^* \in \textup{Hom}_F(X,Y)$. Plus généralement quand $X$ est un sous-espace totalement isotrope quelconque, en donnant une décomposition $W=X+W^0+Y$ où $Y \simeq X^*$ et $W_1$ est un espace symplectique orthogonal à l'espace symplectique $X+Y$, on a un isomorphisme canonique $M(X) \simeq \{ (a,u) \ | \ a \in \textup{GL}_F(X), u \in \textup{Sp}(W^0)\}$. On note $\textup{det}_X$ le caractère de $P(X)$ qui à tout $p = m(a,u) n \in P(X)$, avec $m(a,u) \in M(X)$ et $n \in N(X)$, associe $|\textup{det}(a)|_F$.

La théorie des espaces $\varepsilon$-hermitiens généralise les notions d'espaces orthogonaux, symplectiques et unitaires. Il en sera fait un usage extrêmement bref et limité dans ce travail, aussi invite-t-on, pour les personnes désireuses d'un traitement complet de leur théorie, à consulter \cite[Chap. 1]{mvw}. On rappelle brièvement le peu de vocabulaire dont on va faire usage ici. Un espace $\varepsilon$-hermitien de dimension $n$ est scindé s'il possède un sous-espace totalement isotrope de dimension $m$ tel que $n=2m$. On s'autorise alors, et dans ce cas seulement, à utiliser le mot lagrangien pour signifier un sous-espace totalement isotrope maximal d'un espace scindé. Dans un espace $\varepsilon$-hermitien quelconque, il existe bien évidemment des sous-espaces totalement isotropes maximaux. Enfin, on note par la lettre $U$ son groupe des isométries.

Soit $R$ un anneau commutatif unitaire. Pour tout groupe localement profini $G$, on note $\textup{Rep}_R(G)$ la catégorie des représentations lisses de $G$ à coefficients dans $R$. Le cadre d'étude de ces représentations est largement développé dans \cite[Chap. I]{vig}, dont on donne quelques rappels succincts. On appelle aussi ces représentations des $R[G]$-modules lisses. Le foncteur $V \in R[G]-\textup{mod} \mapsto V^{\infty} \in \textup{Rep}_R(G)$ associe à un $R[G]$-module $V$ sa partie lisse $V^{\infty}$. Pour tout sous-groupe fermé $H$ d'un groupe localement profini $G$, on désignera par $\textup{Ind}_H^G$ et $\textup{ind}_H^G$ les foncteurs d'induction et d'induction compacte respectivement. Ces foncteurs ne sont pas normalisés. D'après \cite[I.2.4]{vig}, l'existence d'un sous-groupe ouvert de $G$ de pro-ordre inversible dans $R$ est une condition nécessaire et suffisante pour qu'il existe une mesure de Haar $\mu_G$ de $G$ à valeurs dans $R$. Pour tout sous-groupe compact ouvert $K$ de $G$, il existe au plus une mesure de Haar $\mu_K$ de $G$ à valeurs dans $R$ qui prenne la valeur $1$ sur $K$. On désigne par $C^\infty(G,R)$ l'ensemble des fonctions sur $G$ à valeurs dans $R$ qui sont localement constante et par $C_c^\infty(G,R)$ le sous-espace de $C_c^\infty(G,R)$ des fonctions à support compact. Le contexte étant clair la plupart du temps, la référence à $R$ sera sous-entendue en écrivant $C^\infty(G)$ et $C_c^\infty(G)$.

On note $\hat{F}_R$ le groupe des caractères lisses de $F$ à valeurs dans $R$. Il est muni d'une structure d'espace vectoriel sur $F$, où l'addition correspond à la multiplication de deux caractères et la multiplication par un scalaire $\lambda \in F$ à $\lambda \cdot \psi : t \mapsto \psi(\lambda t)$ où $\psi \in \hat{F}_R$. Quand $R$ est un corps, l'espace vectoriel $\hat{F}_R$ est de dimension au plus $1$ sur $F$. Il est de dimension exactement $1$ à condition qu'on ait :
$$\mu^p(R) = \left\lbrace \begin{array}{ll} 
\{ \xi \in R^\times \ | \ \exists k \in \mathbb{N}, \xi^{p^k}=1 \} \simeq \mathbb{Q}_p / \mathbb{Z}_p & \textup{si } F \textup{ est de caractéristique $0$ ;} \\
\{ \xi \in R^\times \ | \ \xi^p = 1 \} \simeq \mathbb{Z} / p \mathbb{Z} & \textup{sinon}. \end{array} \right.$$

\begin{rem} \label{remarque_car_R_non_p_si_psi_existe} En particulier, un tel corps $R$ est de caractéristique $\ell$ différente de $p$. \end{rem}

Dans tout cet article, on se place sur un corps $R$ où cette condition est réalisée, de sorte qu'il existe un caractère lisse  non trivial $\psi : F \to R^\times$. On réserve dorénavant \og $R$ \fg{} et \og $\psi$ \fg{} pour désigner un tel corps et un tel caractère. On écrit indifféremment $q$ pour désigner l'entier dans $\mathbb{Z}$ et son image dans $R$ par la morphisme canonique $\mathbb{Z} \to R$. Quand $S$ est un espace vectoriel sur $R$, on note $\textup{GL}_R(S)$ le groupe des endomorphismes inversibles de $S$ et $\textsc{red} : \textup{GL}_R(S) \to \textup{PGL}_R(S)$ le morphisme de groupe quotient.

\subsection{Facteur de Weil non normalisé} \label{facteur_de_weil_non_norm_sect}

Soit $X$ un espace vectoriel sur $F$ de dimension finie $m$. Comme le pro-ordre de $X$ est une puissance de $p$ et la caractéristique $\ell$ de $R$ est différente de $p$, il existe une mesure de Haar $\mu$ de $X$ à valeurs dans $R$ d'après \cite[I.2.4]{vig}. On donne dans cette partie une définition concurrente du facteur de Weil. Elle présente l'avantage d'être plus élémentaite et immédiate. On explique ensuite comment la relier au facteur de Weil habituel, c'est-à-dire tel que défini dans \cite{weil,perrin,rao} quand $R=\mathbb{C}$, et dont la généralisation quand $R$ possède une racine de $q$ est effectuée dans \cite{ct}. 

\begin{prop} \label{facteur_de_Weil_non_normalise_def_prop} Soit $Q$ une forme quadratique non dégénérée sur $X$. Il existe alors un unique élément $\Omega_{\mu}(\psi \circ Q)$ dans $R^\times$ tel que l'on ait pour tout $f \in C_c^\infty(X)$ :
$$\int_X \int_X f(y-x) \psi( Q(x)) d\mu (x) d\mu (y) = \Omega_{\mu}(\psi \circ Q) \int_X f(x) d \mu (x).$$ \end{prop}

\begin{proof} L'application linéaire :
$$\mu' : f \in C_c^\infty(X) \mapsto \int_X \int_X f(y-x) \psi( Q(x)) d\mu (x) d\mu (y) \in R$$
est une mesure de Haar de $X$ à valeurs dans $R$, qui est bien évidemment non nulle. Par unicité de la mesure de Haar \cite[I.2.4]{vig}, il existe un unique élément $c \in R^\times$ de sorte que $\mu' = c \mu$. \end{proof}

\begin{defi} On appelle \emph{facteur de Weil non normalisé} la quantité $\Omega_{\mu}(\psi \circ Q)$. \end{defi}

Ce facteur dépend bien évidemment du choix de $\mu$ comme la notation le suggère. On étend cette définition aux formes quadratiques quelconque. Une forme quadratique $Q$ est non dégénérée si son radical $\textup{rad}(Q)$ est réduit à $0$. Toute forme quadratique $Q$ sur $X$ induit une forme quadratique $Q_{\textup{nd}}$ sur $X / \textup{rad}(Q)$ qui est non dégénérée.

\begin{defi} Soit $Q$ une forme quadratique sur $X$. On définit pour toute mesure de Haar $\mu$ de $X / \textup{rad}(Q)$ à valeurs dans $R$ :
$$\Omega_\mu(\psi \circ Q) : = \Omega_\mu (\psi \circ Q_{\textup{nd}}).$$ \end{defi} 

Avant de présenter quelques propriétés de ce facteur, on introduit quelques considérations assez connues. Si $X'$ est un espace vectoriel isomorphe à $X$, on note $\textup{Iso}_F(X,X')$ l'ensemble des isomorphismes d'espaces vectoriels entre $X$ et $X'$. On écrit $\textup{Aut}_F(X)$ pour signifier $\textup{Iso}_F(X,X)$. Quand $X^* = \textup{Hom}_F(X,F)$ est le dual de $X$, l'ensemble $\textup{Iso}_F(X,X^*)$ est muni d'une involution $\rho \mapsto \rho^*$ donnée par la dualité. On définit :
$$\textup{Iso}_F^{\textup{sym}}(X,X^*) = \{ \rho \in \textup{Iso}_F(X,X^*) \ | \ \rho = \rho^*\}.$$
Chacun de ces ensembles hérite, en tant que sous-ensembles de $\textup{Hom}_F(X,X')$, de la topologie naturelle de cet espace vectoriel de dimension fine.

Soit $\mu$ une mesure de Haar de $X$ à valeurs dans $R$. Pour tout $\phi \in \textup{Aut}_F(X)$, la mesure $\phi \cdot \mu = \mu \circ \phi^{-1}$ est une mesure de Haar de $X$. Elle est proportionnelle à $\mu$ par unicité de la mesure de Haar. Soit $|\phi| \in R^\times$ tel que $\phi \cdot \mu = |\phi| \times \mu$. Alors $|\phi|$ ne dépend pas du choix de $\mu$, on l'appelle le module de $\phi$. On a pour tout sous-groupe compact ouvert $K$ de $X$ :
$$|\phi| = \frac{\textup{vol}_{\phi \cdot \mu}(K)}{\textup{vol}_{\mu}(K)} = \frac{\mu(\phi^{-1} (K))}{\mu(K)}.$$
De plus, l'application module :
\begin{eqnarray*} \textup{Aut}_F(X) & \to & R^\times  \\
 \phi & \mapsto & |\phi|
\end{eqnarray*}
est localement constante et définit un morphisme de groupes. En d'autres termes, le module est un morphisme de groupes continu, où $R^\times$ est muni de la topologie discrète. On peut affiner  la description de l'application module :
$$|\phi| = |\textup{det}_F(\phi)|_F \textup{ pour tout } \phi \in \textup{Aut}_F(X).$$
En particulier, l'image de l'application module est $q^{\mathbb{Z}}$. 

On peut également définir l'application module pour $\textup{Iso}_F(X,X^*)$ comme suit. Soit $\mu$ une mesure de Haar de $X$ à valeurs dans $R$. La mesure duale $\mu^*$ de $\mu$ est l'unique mesure de Haar de $X^*$ telle que l'application de transformée de Fourier :
$$\begin{array}{ccccc} \mathcal{F}_{\mu} &:& C_c^\infty(X) & \to & C_c^\infty(X^*)  \\
 & & f & \mapsto & \mathcal{F}_{\mu} f
\end{array} \textup{ où } \mathcal{F}_{\mu} f : x^* \mapsto \int_X \psi(x^*(x)) f(x) d \mu(x)$$
admette comme inverse :
$$\begin{array}{ccccc} \mathcal{F}_{\mu^*} &:& C_c^\infty(X^*) & \to & C_c^\infty(X)  \\
 & & h & \mapsto & \mathcal{F}_{\mu^*} h
\end{array} \textup{ où } \mathcal{F}_{\mu^*} h : x \mapsto \int_{X^*} \psi(-x^*(x)) h(x^*) d \mu(x).$$
Pour tout $\rho \in \textup{Iso}_F(X,X^*)$, la mesure $\rho \cdot \mu = \mu \circ \rho^{-1}$ est une mesure de Haar de $X^*$, proportionnelle à $\mu^*$ par unicité de la mesure de Haar. Soit $|\rho|_{\mu} \in R^\times$ tel que $\rho \cdot \mu = |\rho|_{\mu} \times \mu^*$. Alors $|\rho|_{\mu}$ dépend du choix de $\mu$, mais seulement à un carré près de $R^\times$. On a pour tout sous-groupe compact ouvert $K$ de $X^*$  :
$${|\rho|}_\mu = \frac{\textup{vol}_{\rho \cdot \mu}(K)}{\textup{vol}_{\mu^*}(K)} = \frac{\mu(\rho^{-1} (K))}{\mu^*(K)}.$$ De plus, l'application module :
\begin{eqnarray*} \textup{Iso}_F(X,X^*) & \to & R^\times  \\
 \rho & \mapsto & |\rho|_{\mu}
\end{eqnarray*}
est localement constante et compatible à l'action de $\textup{Aut}_F(X)$ au sens où pour tout $\phi \in \textup{Aut}_F(X)$ et tout $\rho \in \textup{Iso}_F(X,X^*)$, on a :
$$|\rho \circ \phi|_{\mu} = |\phi| \times |\rho|_{\mu}.$$
Le module ainsi défini est invariant par la dualité au sens où $|\rho|_{\mu} = |\rho^*|_{\mu}$. Quand $K$ est un réseau de $X$, \textit{i.e.} un sous-groupe compact ouvert de $X$ muni d'une structure de $\mathcal{O}_F$-module, la quantité $|\rho|_{\mu_K}$ est une puissance de $q$.

\begin{prop} \label{facteur_de_weil_non_norm_prop} $Q$ désigne une forme quadratique sur $X$ et $\mu$ une mesure sur $X$.
\begin{enumerate}[label=\textup{\alph*)}]
\item Si $Q$ est la forme quadratique nulle, on a $\Omega_\mu(\psi \circ 0) = \mu( \{ 0 \} )$.
\item Pour tout $\lambda \in R^\times$, on a :
$$\Omega_{\lambda \mu} (\psi \circ Q) = \lambda \times \Omega_\mu(\psi \circ Q).$$
\item \label{non_invariance_par_isometrie_weil_non_norm_pt} Pour tout $X'$ de même dimension que $X$ et pour tout $\phi \in \textup{Iso}_F(X,X')$, la forme quadratique $Q_\phi= Q \circ \phi^{-1}$ définie sur $X'$ a pour radical $\phi(\textup{Ker}(Q))$. Alors :
$$\Omega_{\phi \cdot \mu}( \psi \circ Q_\phi ) =  \Omega_\mu(\psi \circ Q).$$
En particulier, si $\phi \in \textup{Aut}_F(X)$ est tel que $\phi(\textup{rad}(Q))=\textup{rad}(Q)$, on a :
$$\Omega_\mu( \psi \circ Q_\phi ) = |\phi|^{-1} \times \Omega_\mu(\psi \circ Q).$$
\item Le facteur de Weil non normalisé est compatible à la somme directe des formes quadratiques au sens suivant. Si $Q_1 \oplus Q_2$ est une forme quadratique sur $X_1 \oplus X_2$ somme de deux formes quadratiques $Q_1$ et $Q_2$ sur $X_1$ et $X_2$ respectivement, alors :
$$\Omega_{\mu_1 \otimes \mu_2}(\psi \circ (Q_1 \oplus Q_2)) = \Omega_{\mu_1}(\psi \circ Q_1) \Omega_{\mu_2}(\psi \circ Q_2).$$
\item L'application :
$$\begin{array}{ccc}	
\textup{Iso}_F^{\textup{sym}}(X,X^*) & \to & R^\times \\
 \rho & \mapsto & \Omega_\mu(\psi \circ Q_\rho) \end{array}, \textup{ où } Q_\rho(x)=\rho(x)(x),$$
est localement constante. En d'autres termes, elle est continue en munissant $R$ de la topologie discrète.
\item \label{facteur_de_weil_non_nomr_vs_classique_pt} On suppose que le corps $R$ contient une racine carrée de $q$, que l'on fixe et note $q^{\frac{1}{2}}$. Soient $\rho \in \textup{Iso}_F^{\textup{sym}}(X,X^*)$ et $K$ un réseau de $X$. Il existe $k \in \mathbb{Z}$ tel que $|\rho|_{\mu_K} = q^k$. Alors ${|\rho|}_\mu^{\frac{1}{2}} = \mu(K) (q^\frac{1}{2})^k$ est une racine carrée de ${| \rho |}_\mu$ qui ne dépend pas du choix de $K$ et le facteur de Weil associé à la forme quadratique non dégénérée $Q_{\frac{1}{2} \rho}$ est :
$$\omega(\psi \circ Q_{\frac{1}{2} \rho}) : = \frac{\Omega_\mu ( \psi \circ Q_{\frac{1}{2} \rho})}{{|\rho|}_\mu^\frac{1}{2}}.$$
\item \label{facteur_de_weil_non_norm_hilbert_symbol_pt} Pour tout $a \in F^\times$, on note $Q_a (x) = a x^2$ la forme quadratique sur $F$. Alors pour toute mesure de Haar $\mu$ de $F$, la quantité :
$$\Omega_{a,b} = \frac{\Omega_\mu(\psi \circ Q_a)}{\Omega_\mu( \psi \circ Q_b)} \in R^\times$$
ne dépend pas du choix de $\mu$. De plus, en notant $(\cdot , \cdot)_F$ le symbole de Hilbert de $F$, on a pour tout $a$ et $b$ dans $F^\times$ :
$$(a,b)_F = \frac{\Omega_\mu(\psi \circ Q_1) \Omega_\mu(\psi \circ Q_{ab})}{\Omega_\mu(\psi \circ Q_a) \Omega_\mu(\psi \circ Q_b)} = \frac{\Omega_{ab,1}}{\Omega_{a,1} \Omega_{b,1}}.$$ \end{enumerate} \end{prop}

\begin{proof} a) b) c) Les deux premiers points sont immédiats. Le troisième résulte du changement de variable $x = \phi^{-1}(x')$ qui donne :
$$\Omega_{\mu}(\psi \circ Q_\phi) = \Omega_{\phi^{-1} \cdot \mu}(\psi \circ Q).$$
Comme $\phi^{-1} \cdot \mu = |\phi|^{-1} \times \mu$, le point b) donne $\Omega_{\phi^{-1} \cdot \mu}(\psi \circ Q) =|\phi|^{-1} \times \Omega_\mu(\psi \circ Q)$.

d) Vient de la définition de la mesure de Haar produit $\mu_1 \otimes \mu_2$ de $X_1 \oplus X_2$, construite à partir de deux mesures de Haar $\mu_1$ de $X_1$ et $\mu_2$ de $X_2$ comme suit. L'application bilinéaire :
$$\begin{array}{ccc}	
C_c^\infty(X_1) \times C_c^\infty(X_2) & \to & C_c^\infty(X_1 \oplus X_2) \\
 (f_1,f_2) & \mapsto & \bigg( F_{f_1,f_2} : (x_1,x_2) \mapsto f(x_1) f(x_2) \bigg) \end{array}$$
induit un isomorphisme d'espaces vectoriels $C_c^\infty(X_1) \otimes_R C_c^\infty(X_2) \simeq C_c^\infty(X_1 \oplus X_2)$. La mesure de Haar produit $\mu_1 \otimes \mu_2$ est alors l'unique mesure définie pour tout $f_1 \in C_c^\infty(X_1)$ et tout $f_2 \in C_c^\infty(X_2)$ par $\mu_1 \otimes \mu_2(f_1 \otimes f_2) = \mu_1(f_1) \mu_2(f_2)$.

e) Ceci est vrai car, à $f \in C_c^\infty(X)$ fixé, l'application :
$$\rho \mapsto \int_X \int_X f(y-x) \psi( Q_\rho (x)) d\mu (x) d\mu (y)$$
est localement constante. 

f) La première partie de l'énoncé constitue des faits généraux élémentaires sur les mesures et les applications modules. Ensuite, le point 2. de \cite[Prop. 3.3]{ct} évalué en $x^*=0$ donne, en notant $\gamma$ le facteur de Weil de $Q_{\frac{1}{2} \rho}$ :
$$\int_X \int_X f(y-x) \psi( Q_{\frac{1}{2}\rho}(x)) d\mu (x) d\mu (y) = \gamma \ |\rho|_{\mu}^{\frac{1}{2}} \int_X f(x) d \mu (x).$$
Par définition du facteur de Weil non normalisé, on a $\Omega_{\mu}(\psi \circ Q_{\frac{1}{2} \rho}) = \gamma |\rho|_{\mu}^{\frac{1}{2}}$.

g) D'après le point b), $\Omega_{a,b}$ est bien indépendant de $\mu$. Enfin, quitte à adjoindre une racine carrée de $q$ à $R$ pour former une extension $R'$ de $R$, on obtient que pour tout $a \in F^\times$, l'élément $|a| = |a|_F$ est un carré dans $R'$. Par conséquent, on a dans $R'$ :
$$(a,b)_F = \frac{\omega(\psi \circ Q_1) \omega(\psi \circ Q_{ab})}{\omega(\psi \circ Q_a) \omega(\psi \circ Q_b)} = \frac{\Omega_\mu(\psi \circ Q_1) \Omega_\mu(\psi \circ Q_{ab})}{\Omega_\mu(\psi \circ Q_a) \Omega_\mu(\psi \circ Q_b)}$$
où la première égalité résulte de \cite[4.3]{ct} et la deuxième du point \ref{facteur_de_weil_non_nomr_vs_classique_pt} ici présent. Dans les membres de droite et de gauche, chacun des termes appartient à $R$ donc l'égalité :
$$(a,b)_F = \frac{\Omega_\mu(\psi \circ Q_1) \Omega_\mu(\psi \circ Q_{ab})}{\Omega_\mu(\psi \circ Q_a) \Omega_\mu(\psi \circ Q_b)}$$
est valable dans $R$. \end{proof}

Le facteur non normalisé n'est \textit{a priori} pas trivial quand la forme quadratique est scindée, ce qui constitue une différence notable d'avec le facteur de Weil habituel. De plus, ce facteur non normalisé dépend de la réalisation de $Q$ considérée au sens où il n'est pas invariant par isométrie d'après le point \ref{non_invariance_par_isometrie_weil_non_norm_pt}. On a cependant une formule qui éclaire sa signification quand on réalise $Q$ dans une base orthogonale de $X$.

Soient $Q$ une forme quadratique non dégénérée sur $X$ et $\mathcal{B}=\{v_1, \dots, v_m \}$ une base orthogonale de $X$ pour cette forme quadratique. Ce choix de base induit un isomorphisme de $X$ vers $F^m$ donné par les coordonnés et que l'on note $\phi_{\mathcal{B}}$. On fixe une mesure de Haar $\mu_F$ de $F$. On considère la mesure de Haar produit $\otimes \mu_F$ sur $F^m$ qui se transporte en une mesure de Haar $\phi_{\mathcal{B}}^{-1} \cdot (\otimes \mu_F)$ de $X$. On pose $a_i = Q(v_i)$. La quantité $\det_{\mathcal{B}}(Q)= \prod a_i$ dépend du choix de la base $\mathcal{B}$, alors que l'invariant de Hasse $h_F(Q) = \prod_{i < j} (a_i,a_j)_F$ n'en dépend pas. On déduit des différents points de la Proposition \ref{facteur_de_weil_non_norm_prop} un corollaire qui en généralise le point \ref{facteur_de_weil_non_norm_hilbert_symbol_pt} :

\begin{cor} \label{facteur_de_weil_inv_de_hasse_cor} On a :
$$\Omega_{\phi_{\mathcal{B}}^{-1} \cdot ( \otimes \mu_F )} (\psi \circ Q ) = \Omega_{\textup{det}_{\mathcal{B}}(Q),1} \times \Omega_{\mu_F}(\psi \circ Q_1)^m h_F(Q).$$
De plus, si $Q_{\textup{Id}_{\mathcal{B}}}$ désigne l'unique forme quadratique non dégénérée diagonale associée à l'identité sur la base $\mathcal{B}$, alors pour toute mesure de Haar $\mu$ de $X$ :
$$\Omega_\mu(\psi \circ Q) = \Omega_{\textup{det}_{\mathcal{B}}(Q),1} \times \Omega_{\mu}(\psi \circ Q_{\textup{Id}_{\mathcal{B}}}) h_F(Q).$$
\end{cor}

\begin{rem} Par conséquent, la quantité $\Omega_{\textup{det}_{\mathcal{B}}(Q),1} \times \Omega_{\mu}(\psi \circ Q_{\textup{Id}_{\mathcal{B}}})$ ne dépend pas du choix de la base $\mathcal{B}$. Pour en donner une interprétation plus géométrique, pour tout autre choix de base $\mathcal{B}'$ qui diagonalise $Q$, la quantité $\textup{det}_{\mathcal{B}'}(Q)$ diffère de $\textup{det}_{\mathcal{B}}(Q)$ par le carré du déterminant du changement de base entre $\mathcal{B}$ et $\mathcal{B}'$. On observe que le point \ref{non_invariance_par_isometrie_weil_non_norm_pt} de la Proposition \ref{facteur_de_weil_non_norm_prop} garantit que cette différence transforme $\Omega_{\mu} (\psi \circ Q_{\textup{Id}_{\mathcal{B}}})$ en $\Omega_{\mu} (\psi \circ Q_{\textup{Id}_{\mathcal{B}'}})$. \end{rem}

\paragraph{Facteur de Weil.} Grâce au point \ref{facteur_de_weil_non_nomr_vs_classique_pt} de la Proposition \ref{facteur_de_weil_non_norm_prop}, il existe un lien avec le facteur de Weil quand le corps $R$ possède une racine carrée de $q$. Une fois une telle racine $q^{\frac{1}{2}} \in R$ fixée, le facteur de Weil d'une forme quadratique $Q$ non dégénérée est donc :
$$\omega(\psi \circ Q) = \frac{\Omega_\mu(\psi \circ Q)}{|\rho|_{\mu}^{\frac{1}{2}}}$$
où $\rho : x \mapsto Q(x+y)-Q(x)-Q(y)$ est la forme symétrique associée à $Q = Q_{\frac{1}{2} \rho}$. L'image du facteur de Weil est contenue, d'après \cite[4.3]{ct}, dans l'ensemble des racines $4$-èmes de l'unité de $R$ si $F$ n'est pas une extension finie de $\mathbb{Q}_2$ sans racine de $-1$. Dans le cas contraire, on a $(-1,-1)_F=-1$ et l'image du facteur de Weil est seulement contenue dans l'ensemble des racines $8$-èmes de l'unité.

Quand $R$ ne contient pas de racine carrée de $q$, il faut utiliser $\Omega(\psi \circ Q)$ en toute généralité. Voici un exemple pour illustrer le propos. Quand $F=\mathbb{F}_3((t))$ et $R=\mathbb{Q}[j]$, il existe un caractère lisse additif non trivial $\psi : F \to R^\times$. Bien que $\sqrt{3} \notin R$, on  a cependant $i \sqrt{3} \in R$. Or, l'élément $\omega(\psi \circ Q) \in \mathbb{C}^\times$ est une racine $4$-ème de l'unité. Cela signifie que pour toute forme quadratique $Q$ telle que $|\rho|_\mu=3$, on peut prendre par exemple $\rho(x)(y) = t x y$, son facteur de Weil non normalisé ne peut être que $i \sqrt{3}$ ou $-i\sqrt{3}$. Sont ainsi encodées dans ce facteur des propriétés arithmétiques de $F$ qui, en un certain sens, ne dépendent pas de $R$.

\paragraph{Normalisation de la transformée de Fourier.} On interprète maintenant ce facteur de Weil non normalisé comme un coefficient de normalisation pour la transformée de Fourier vis-à-vis de $\psi$. Ce facteur permet, pour tout $\rho \in \textup{Iso}_F^{\textup{sym}}(X,X^*)$, de donner une mesure $\mu_\rho$ qui est en un certain sens la plus naturelle pour la transformée de Fourier à coefficients dans $R$.

\begin{prop} \label{normalisation_transformee_de_Fourier_prop} Soient $\rho \in \textup{Iso}_F^{\textup{sym}}(X,X^*)$ et $\mu$ une mesure de Haar de $X$. 

\begin{enumerate}[label=\textup{\alph*)}]
\item La mesure de Haar :
$$\mu_\rho = \Omega_\mu(\psi \circ Q_{\frac{1}{2} \rho})^{-1} \mu$$
ne dépend pas du choix de $\mu$.
\item Si l'on note $\star_{\mu_\rho}$ le produit de convolution dans $C_c^\infty(X)$ et $\times$ la multiplication des fonctions, l'opérateur suivant dit de transformée de Fourier :
$$\mathcal{F}_{\mu_\rho} : f \in (C_c^\infty(X),\star_{\mu_\rho}) \mapsto \bigg( x \mapsto \int_X \psi( \rho(x)(u)) f(u) d \mu_\rho(u) \bigg) \in (C_c^\infty(X),\times)$$
est un isomorphisme d'algèbre.
\item On pose :
$$\varepsilon =  \big(-1,\textup{det}(Q_{\frac{1}{2} \rho})\big)_F \times  (\Omega_{-1,1})^{m}$$
qui vérifie :
$$\varepsilon^2 = \big(-1,-1\big)_F^m.$$
Alors pour tout $f \in C_c^\infty(X)$, on a :
$$\mathcal{F}_{\mu_\rho}^4 f= \varepsilon^2 f \textup{ et } \mathcal{F}_{\mu_\rho}^2 f : x \mapsto \varepsilon f(-x).$$ 
\end{enumerate} 
\end{prop}

\begin{proof} a) D'après le point b) de la Proposition \ref{facteur_de_weil_non_norm_prop}, la mesure $\mu_\rho$ ne dépend pas du choix de $\mu$.

\noindent b) Ensuite, vérifier que $\mathcal{F}_{\mu_\rho}$ est un isomorphisme d'algèbre constitue un fait classique de transformée de Fourier \cite[Prop. 1.2]{ct}, que l'on rappelle succinctement. Tout d'abord, l'application est bien un automorphisme puisqu'elle est linéaire et que pour tout sous-groupe compact ouvert $K$ de $X$, on a $\mathcal{F}_{\mu_\rho} 1_K = \mu_\rho(K) \times 1_{K^\perp}$ où $K^\perp$ est le sous-groupe compact ouvert $\{ x \in X \ | \ \forall u \in X, \psi(\rho(x)(u)) = 1 \}$. De même, il suffit de vérifier, pour les indicatrices seulement, la propriété d'être un morphisme d'algèbres.

\noindent c) Pour terminer, en reprenant les notations du paragraphe précédent, on a :
$$\mathcal{F}_{\mu_\rho}^2 1_K = \mu_\rho(K) \mu_\rho (K^\perp) \times 1_K.$$
On pose donc $\varepsilon = \mu_\rho (K) \mu_\rho(K^\perp)$. Par définition :
$$\varepsilon = \mu_\rho(K) \mu_\rho(K^\perp) = \Omega_\mu(\psi \circ Q_{\frac{1}{2} \rho})^{-2} \times \mu(K) \mu(K^\perp).$$
Alors, si $K'$ désigne le sous-groupe ouvert compact $\{ x^* \in X^* \ | \ \forall u \in X, \psi(x^*(u))=1 \}$ dans $X^*$, on a :
$$\mu(K) \mu(K^\perp) = \frac{\mu(\rho^{-1} K')}{\mu^*(K')} = |\rho|_\mu.$$
On déduit de :
$$\Omega_\mu(\psi \circ Q_{- \frac{1}{2} \rho} )= \frac{|\rho|_\mu}{\Omega_\mu(\psi \circ Q_{\frac{1}{2} \rho} )}$$
qu'on a :
$$\varepsilon = \frac{\Omega_\mu(\psi \circ Q_{-\frac{1}{2} \rho} )}{\Omega_\mu(\psi \circ Q_{\frac{1}{2} \rho} )}.$$

D'après le Corollaire \ref{facteur_de_weil_inv_de_hasse_cor}, comme les deux formes quadratiques $Q_{-\frac{1}{2} \rho}$ et $Q_{\frac{1}{2} \rho}$ peuvent se diagonaliser dans la même base $\mathcal{B}$ de $X$, cette quantité se réécrit :
$$\varepsilon = \frac{\Omega_{\textup{det}_{\mathcal{B}}(Q_{- \frac{1}{2}\rho}),1}}{\Omega_{\textup{det}_{\mathcal{B}}(Q_{\frac{1}{2}\rho}),1}} \times \frac{h_F(Q_{-\frac{1}{2}\rho})}{h_F( Q_{\frac{1}{2}\rho})}.$$
D'une part, comme $\textup{det}_{\mathcal{B}}(Q_{- \frac{1}{2}\rho}) = (-1)^m \textup{det}_{\mathcal{B}}(Q_{\frac{1}{2}\rho})$, on déduit du point g) de la Proposition \ref{facteur_de_weil_non_norm_prop} l'égalité :
\begin{eqnarray*}
\frac{\Omega_{\textup{det}_{\mathcal{B}}(Q_{- \frac{1}{2}\rho}),1}}{\Omega_{\textup{det}_{\mathcal{B}}(Q_{\frac{1}{2}\rho}),1}} &=& \Omega_{(-1)^m,1} \times \big((-1)^m, \textup{det}_\mathcal{B}(Q_{\frac{1}{2} \rho})\big)_F \\
 &=& (\Omega_{-1,1})^m \times (-1,-1)_F^{\frac{m(m-1)}{2}} \times \bigg( \big(-1, \textup{det}_\mathcal{B}(Q_{\frac{1}{2}\rho}) \big)_F \bigg)^m. \end{eqnarray*}
D'autre part, de $(-a_i,-a_j)_F = (-1, - a_i a_j)_F \times (a_i, a_j)_F$, on tire l'égalité :
\begin{eqnarray*} h_F(Q_{- \frac{1}{2} \rho}) &=& \big( -1, (-1)^{\frac{m(m-1)}{2}} \textup{det}(Q_{\frac{1}{2} \rho})^{m-1}) \big)_F \times  h_F(Q_{\frac{1}{2} \rho}) \\
&=& (-1,-1)_F^{\frac{m(m-1)}{2}} \times \bigg( \big(-1,\textup{det}(Q_{\frac{1}{2} \rho}) \big)_F \bigg)^{m-1} \times h_F(Q_{\frac{1}{2}\rho}). \end{eqnarray*}
Donc on obtient l'égalité recherchée pour $\varepsilon$. Toujours en invoquant le point g) de la Proposition \ref{facteur_de_weil_non_norm_prop}, le carré de $\Omega_{-1,1}$ est bien $(\Omega_{-1,1})^2=(-1,-1)_F$.

Concernant les puissances de l'opérateur $\mathcal{F}_{\mu_\rho}$, un argument classique consiste à examiner l'image d'indicatrices de la forme $1_{x + K}$, où $x$ est un élément de $X$ et $K$ un sous-groupe compact ouvert de $X$. Ceci constitue un fait habituel en transformée de Fourier. Pour terminer, le prouver pour les fonctions indicatrices est suffisant puisqu'on déduit par linéarité le résultat pour toute fonction dans $C_c^\infty(X)$. \end{proof}
\begin{defi} On appelle \textit{$R$-transformée de Fourier} l'opérateur $\mathcal{F}_{\mu_\rho}$. \end{defi}

On revient sur l'exemple $F= \mathbb{F}_3((t))$, $\mathcal{O}_F=\mathbb{F}_3[[t]]$ et $R = \mathbb{Q}(j)$. Soient $\psi : F \to R^\times$ le caractère lisse non trivial de noyau $\mathcal{O}_F$ et $\mu$ une mesure de $F$ normalisée sur $\mathcal{O}_F$. En notant, $(\mathcal{O}_F)' = \{ x^* \in F^* \ | \ \forall x \in \mathcal{O}_F, \psi(x^*(x))=1\}$, on a que $\mu^*$ est la mesure de $X^*$ normalisée sur $(\mathcal{O}_F)'$. On considère maintenant le morphisme symétrique :
$$\rho : x \in F \mapsto ( \rho(x) : y \mapsto t xy ) \in F^*$$
dont le module est ${|\rho|}_\mu = \mu(\rho^{-1} (\mathcal{O}_F)') = \mu( t^{-1} \mathcal{O}_F ) = |t^{-1}|_F = 3$. Dans ce cas, on a montré que $\omega(\psi \circ Q_{\frac{1}{2} \rho}) \in \{ \pm i\}$. Donc la transformée de Fourier classique :
$$\mathcal{F} : f \in C_c^\infty(F,\mathbb{C}) \mapsto \bigg( x \mapsto \int_F \psi(\rho(x)(u)) f(u) d \mu_F(u) \bigg) \in C_c^\infty(F,\mathbb{C})$$
définie pour la mesure auto-duale $\mu_F = (\sqrt{3})^{-1} \mu$ n'est pas définie sur $R$, mais l'est en revanche sur $R'=R[\sqrt{3}]$. Elle vérifie pour tout $f \in C_c^\infty(X,R')$, les relations habituelles $\mathcal{F}^4 f = f$ et $\mathcal{F}^2 f : x \mapsto f(-x)$. Par conséquent, le choix naturel de normalisation pour la transformée de Fourier à coefficients dans $R$ vis-à-vis de $\rho$ et $\psi$ consiste plutôt à considérer l'opérateur $\mathcal{F}_{\mu_\rho}$. Dans ce cas, on a $\varepsilon = (- 1) \times \Omega_{-1,1}$ et l'application $\mathcal{F}_{\mu_\rho}$ de $R$-transformée de Fourier vérifie $\mathcal{F}_{\mu_\rho}^4 f = f$ et $\mathcal{F}_{\mu_\rho}^2 f : x \mapsto \varepsilon f(-x)$.

\section{Représentations métaplectiques modulaires} \label{representations_modulaires_section}

Le groupe d'Heisenberg $H(W,\langle , \rangle )$, abrégé en $H$ quand le contexte est clair, est l'ensemble $W \times F$ muni de la topologie produit et de la loi de groupe :
$$(w,t) \cdot (w',t') = \bigg(w+w',t+t'+\frac{\langle w,w' \rangle }{2}\bigg).$$
On identifie $F$ avec le centre de $H$ via l'isomorphisme de groupes topologiques $t \mapsto (0,t)$, ainsi que $W$ avec le sous-ensemble $W \times 0$ de $H$ via l'homéomorphisme $\delta : w \mapsto (w,0)$.

\subsection{Théorème de Stone-von Neumann modulaire}

Le résultat ci-dessous est une généralisation du théorème de Stone-von Neumann pour les représentations à coefficients complexes \cite[Chap. 2, Th. I.2]{mvw}. On lui donne le nom de \emph{théorème de Stone-von Neumann modulaire}.

\begin{theo} \label{stone_von_neumann_modulaire_thm} Soit $\psi \in \hat{F}_R$ un caractère non trivial. À isomorphisme près, il existe une unique représentation irréductible $(\rho_\psi,S)$ dans $\textup{Rep}_R(H)$ dont le caractère central est $\psi$. \end{theo}

\begin{proof} On généralise la preuve de \cite{mvw} qui occupe la première partie du second chapitre. La preuve consiste à tout d'abord exhiber une classe de représentations qui vérifient les conditions du théorème, dont on se sert ensuite pour montrer l'unicité à isomorphisme près.

Un premier candidat qui vient à l'esprit est l'induite compacte $\textup{ind}_F^{H}(\psi)$ qui, en plus d'être lisse, a pour caractère central $\psi$. Cependant, cette représentation n'est pas irréductible. Néanmoins, il existe des induites compactes plus fines que cette dernière et qui se révèleront irréductibles. Pour ce faire, on définit l'orthogonal de tout sous-groupe fermé $A$ de $W$ par $A^\perp = \{ w \in W \ | \ \forall a \in A, \ \psi(\langle w,a \rangle )=1 \}$. Une liste de bons candidats est la suivante :

\begin{lem} \label{stonve_von_neumann_premier_lem} On suppose que $A$ est \emph{auto-dual} i.e. $A=A^\perp$. Alors :
\begin{enumerate}[label=\textup{\alph*)}]
\item il existe un caractère lisse $\psi_A$ du sous-groupe $A_H=A \times F$ de $H$ dont la restriction à $F$ est $\psi$ ;
\item pour tout caractère lisse $\psi_A$ de $A_H$ prolongeant $\psi$, l'induite compacte $\textup{ind}_{A_H}^{H}(\psi_A)$ est irréductible, égale à $\textup{Ind}_{A_H}^{H}(\psi_A)$ et de contragrédiente $\textup{ind}_{A_H}^{H}(\psi_A^{-1})$. \end{enumerate} \end{lem}

\begin{proof} a) Soit $L$ un sous-groupe ouvert de $W$ tel que $\frac{\langle L,L \rangle }{2} \subset \textup{Ker} (\psi)$. En particulier, $L \times \textup{Ker} (\psi)$ est un sous-groupe ouvert de $H$. Comme $A$ est auto-dual, le sous-groupe ouvert $(L\cap A) \times \textup{Ker}(\psi)$ de $A_H$ est distingué dans $A_H$. Le groupe quotient $A \times F / (L\cap A) \times \textup{Ker}(\psi)$ est abélien. Pour montrer l'existence de $\psi_A$, il suffit donc de prouver qu'il existe un morphisme de groupes $\psi'$ de sorte que le diagramme suivant commute :
$$\xymatrix{
	F / \textup{Ker}(\psi) \ar@{^{(}->}[r] \ar[d]_{\psi} & A \times F / (L\cap A) \times \textup{Ker}(\psi) \\
	\mu^p(R) 	\ar@{<--}[ru]_{\psi'} & }.$$

Quand $F$ est de caractéristique $0$, le groupe $\mu^p(R) \simeq \mathbb{Q}_p / \mathbb{Z}_p$ est un groupe abélien divisible donc un objet injectif dans la catégorie des $\mathbb{Z}$-modules. L'existence de $\psi'$ est alors assurée pusique toutes les flèches du diagramme sont des morphismes de $\mathbb{Z}$-modules.

Quand la caractéristique est positive, chacun des groupes du diagramme est muni d'une structure d'espace vectoriel sur $\mathbb{F}_p$ qui assure que les flèches en question sont des applications linéaires. Comme $\mu^p(R) \simeq \mathbb{F}_p$, l'existence de $\psi'$ résulte du théorème de la base incomplète \textit{i.e.} il existe une section $\sigma : \mu^p(R) \to F/\textup{Ker}(\psi)$ de $\psi$ et un supplémentaire $V$ de $\textup{Im}(\sigma)$ dans $A \times F / (L \cap A) \times \textup{Ker}(\psi)$ tels que :
$$A \times F / (L \cap A) \times \textup{Ker}(\psi) = \textup{Im}(\sigma) \oplus V.$$
Ainsi $\psi'$ est le morphisme quotient $A \times F / (L \cap A) \times \textup{Ker}(\psi) \to \mu^p(R)$, qui est obtenu en identifiant $\mu^p(R)$ et $\textup{Im}(\sigma)$ à l'aide de $\sigma$. En composant $\psi'$ avec le morphisme quotient :
$$A_H \to A \times F / (L \cap A) \times \textup{Ker}(\psi)$$
on obtient un caractère $\psi_A$ de $A_H$ dont la restriction à $F$ est bien $\psi$.

b) Soit maintenant $\psi_A$ un caractère de $A_H$ prolongeant $\psi$. On considère la représentation $S_A=\textup{ind}_{A_H}^{H}(\psi_A)$ dont on donne un système de générateurs. Soit $w \in W$. Comme $\psi_A$ est lisse, il existe un sous-groupe ouvert compact $L_w$ de $W$ tel que $\psi_A(a)=1$ pour tout $a \in A_H \cap \delta(w) \delta(L_w) \delta(w)^{-1}$. Par conséquent, pour tout sous-groupe ouvert $L$ d'un tel $L_w$, il existe une fonction dans $S_A$ de support $A_H \delta(w) \delta(L)$. On note dans ce cas $\chi_{w,L}$ l'unique fonction qui a pour support $A_H \delta(w) \delta(L)$, est $L$-invariante à droite et vaut $1$ en $\delta(w)=(w,0)$. Comme les éléments de $S_A$ sont des fonctions sur $H$ qui sont lisses et à support compact modulo $A_H$, le $R$-module $S_A$ est engendré par la famille $\mathcal{B}=\{ \chi_{w,L} \}_{w \in W, L \subset L_w }$.

Pour montrer l'irréductibilité de $S_A$, on souhaite prouver que pour tout $f \in S_A$ non nul, la sous-représentation $\mathcal{H}_R(H) \cdot f$ de $S_A$ contient la famille génératrice $\mathcal{B}$. Pour ce faire, il suffit de montrer que pour tout $w \in W$, il existe un sous-groupe ouvert compact $L_w'$ de $W$ tel que $\{\chi_{w,L}\}_{L \subset L_w'} \subset \mathcal{H}_R(H) \cdot f$. Soient donc $f \in S_A$ non nul et $w \in W$. Quitte à remplacer $f$ par un de ses translatés sous $H$, on peut supposer qu'on a $f((w,0)) \neq 0$. Par lissité de $S_A$, il existe un sous-groupe ouvert compact $L_w'$ de $W$ tel que $L_w' \times \textup{Ker}(\psi)$ soit dans le fixateur de $f$ et $\psi_A(a)=1$ pour tout $a \in A_H \cap \delta(w) \delta(L_w') \delta(w)^{-1}$. Soit $\mu_A$ la mesure de Haar sur $A$ à valeurs dans $R$. On rappelle qu'elle existe car, d'après la Remarque \ref{remarque_car_R_non_p_si_psi_existe}, on a $\ell \neq p$. Soit $L$ un sous-groupe ouvert de $L_w'$. On définit :
$$\phi : a \in A \mapsto \frac{\psi_A((-a,0))}{\textup{vol}(L^\perp \cap A)} \psi(\langle -w,a \rangle ) 1_{L^\perp \cap A}(a) \in R$$
où $1_X$ désigne l'indicatrice de $X$. En particulier $\phi \in C_c^\infty(A)$. On veut montrer que la fonction :
$$\phi \cdot f : h \in H \mapsto \int_A \phi(a) f(h(a,0)) d \mu_A(a) \in R,$$
est un multiple non nul de $\chi_{w,L}$. Tout d'abord, on voit facilement que la fonction $\phi \cdot f$ appartient à $S_A$. Elle appartient à $\mathcal{H}(H) \cdot f$ car $\phi \cdot f$ est, à une puissance de $p$ près, une somme finie de la forme $\sum \phi(a_i) f(h(a_i,0)) = \sum \phi(a_i) ((a_i,0) \cdot f)(h)$. Ensuite, elle est invariante par $\delta(L)$ à droite et a pour support $A_H \delta(w) \delta(L)$ d'après les égalités suivantes :
\begin{eqnarray*} 
\phi \cdot f((w',0))& = & f((w',0)) \times \frac{1}{\textup{vol}(L^\perp \cap A)}  \int_{L^\perp \cap A} \psi(\langle w'-w,a \rangle ) d \mu_A(a) \\
 & = & f((w',0)) 1_{A+w+L}(w').
\end{eqnarray*}
On en déduit donc que $\phi \cdot f = f((w,0)) \chi_{w,L} \in \mathcal{H}(H) \cdot f$ avec $f((w,0)) \neq 0$. 

Pour terminer, on prouve l'égalité $\textup{ind}_{A_H}^{H}(\psi_A)=\textup{Ind}_{A_H}^H(\psi_A)$. En effet, soit $f : H \to R$ invariante à droite par un sous-groupe ouvert et telle que $f(ah) = \psi_A(a) f(h)$ pour tout $a \in A_H$ et tout $h \in H$. Il s'agit de prouver que $f$ est à support compact modulo $A_H$. Soit $L$ un sous-groupe compact ouvert de $W$ tel que $f$ soit invariante à droite par l'action de $\delta(L)$. Alors pour tout $w \in W$ et tout $l \in L \cap A$ :
$$f((w,0))=f((w,0)(l,0))= f((l,\langle w,l \rangle )(w,0))=\psi(\langle w,l \rangle )\psi_A((l,0))f((w,0)).$$
On en déduit que si $f((w,0)) \neq 0$, alors $\psi(\langle w,l \rangle )=\psi_A((-l,0))$ pour tout $l \in L \cap A$, et deux tels éléments $w$ et $w'$ différent par un élément de $(L \cap A)^\perp=L^\perp + A$. Donc le support de $f$ est contenu dans un compact modulo $A_H$.

On applique enfin la compatibilité à la contragrédiente de l'induction \cite[I.5.11]{vig}, en remarquant qu'ici les groupes $H$ et $A_H$ sont limites de leurs pro-$p$-groupes. Cela signifie que les modules $\delta_H$ et $\delta_{A_H}$ -- qui sont bien définis d'après la Remarque \ref{remarque_car_R_non_p_si_psi_existe} -- sont triviaux. On a donc :
$$\textup{ind}_{A_H}^{H}(\psi_A)^\vee \simeq \textup{Ind}_{A_H}^{H}(\psi_A^\vee)=\textup{ind}_{A_H}^{H}(\psi_A^{-1}).$$ \end{proof}

Il existe toujours des sous-groupes auto-duaux de $W$ -- les lagrangiens en fournissent des exemples. Par conséquent, l'existence d'une représentation vérifiant les conditions du Théorème \ref{stone_von_neumann_modulaire_thm} est assurée. Maintenant, l'unicité résulte aisément du lemme suivant :

\begin{lem} \label{stonve_von_neumann_deuxieme_lem} Soit $(\rho_\psi,S) \in \textup{Rep}_R(H)$ vérifiant les hypothèses du Théorème \ref{stone_von_neumann_modulaire_thm}. Alors :
\begin{enumerate}[label=\textup{\alph*)}]
\item $\rho_\psi$ est une sous-représentation de $\textup{ind}_F^{H}(\psi)$ ;
\item la représentation $\textup{ind}_F^{H}(\psi)$ est isotypique. 
\end{enumerate} \end{lem}

\begin{proof} a) On note $(\rho_\psi^\vee,S^\vee)$ la contragrédiente de $(\rho_\psi,S)$. Pour tout tenseur élémentaire $s^\vee \otimes s$ de $S^\vee \otimes S$, on définit le coefficient $f_{s^\vee,s} : h \mapsto s^\vee(\rho_\psi(h) s)$ associé, qui est dans $C_c^\infty(H)$. Soit $\Phi$ l'unique morphisme de $H \times H$-représentations de $S \otimes S^\vee$ dans $C_c^\infty(H)$ tel que $\Phi(s^\vee \otimes s) = f_{s^\vee,s}$ pour tout tenseur élémentaire. Pour prouver le premier point, il suffit de montrer que l'image de $\Phi$ est $\textup{ind}_F^{H}(\psi)$. Tout d'abord, il est immédiat que l'image de $\Phi$ est contenue dans $\textup{Ind}_F^{H}(\psi)$. De plus, les fonctions $f_{s^\vee,s}$ sont à support compact modulo $F$. En effet, en choisissant $L$ un sous-groupe ouvert compact de $W$ tel que $s^\vee$ et $s$ soit fixe par $\delta(L)$, la fonction $f_{s^\vee,s}$ est invariante à gauche et à droite par $\delta(L)$. On obtient donc la relation pour tout $w \in W$ et $l \in L$ :
$$f_{s^\vee,s}((w,0))=f_{s^\vee,s}((-l,0)(w,0)(l,0))=\psi(\langle w,l \rangle ) f_{s^\vee,s}((w,0)).$$
Par conséquent, cette égalité entraîne que l'ensemble $\{ w \in W \ | \ f_{s^\vee,s}((w,0)) \neq 0 \}$ est inclus dans le compact $L^\perp $.

b) Il reste alors à montrer que $\textup{ind}_F^{H}(\psi)$ est isotypique. Soit $W = X+Y$ une polarisation complète. Pour tout lagrangien $A$ de $W$, on pose $\psi_A : (x,t) \in A_H \mapsto \psi(t) \in R^\times$. C'est un caractère de $A_H$ qui étend $\psi$. Soient :
$$S = \textup{ind}_{X_H}^H(\psi_X) \textup{ et } S' =  \textup{ind}_{Y_H}^H(\psi_Y^{-1}).$$
D'après le Lemme \ref{stonve_von_neumann_premier_lem}, ces deux représentations sont irréductibles. De plus,  d'après la Remarque \ref{remarque_car_R_non_p_si_psi_existe}, il existe des mesures de Haar $\mu_X$ et $\mu_Y$ de $X$ et $Y$. Une vérification élémentaire donne que l'application linéaire suivante est un isomorphisme de représentations :
$$\begin{array}{ccc} 
\textup{ind}_{X_H}^H(\psi_X^{-1}) & \rightarrow & \textup{ind}_{Y_H}^H(\psi_Y^{-1}) \\
f & \mapsto & \displaystyle \bigg( h \mapsto \int_Y f((y,0)h) d \mu_Y(y) \bigg) \end{array} .$$
Ainsi, la dualité entre $S^\vee$ et $S$ se transporte en une dualité entre $S'$ et $S$ via :
$$\langle s',s \rangle  = \int_{X \times Y} s'((x,0)) s((y,0)) \psi(\langle x,y \rangle ) d \mu_X(x) d \mu_Y(y).$$

On considère l'opérateur d'entrelacement $(s',s) \in S' \otimes S \to f_{s',s} \in \textup{ind}_F^H(\psi)$ qui est construit comme précédemment. On va montrer grâce aux deux modèles explicites $S'$ et $S$ que cet opérateur est un isomorphisme de $H \times H$-représentations. Tout d'abord, on a des isomorphismes d'espaces vectoriels $S' \simeq C_c^\infty(X)$ et $S \simeq C_c^\infty(Y)$ induits par la restriction à $X$ et $Y$ respectivement. De même, la représentation d'arrivée s'identifie à $C_c^\infty(W)$ par la restriction à $W$. Enfin, les espaces vectoriels $C_c^\infty(X) \otimes C_c^\infty(Y)$ et $C_c^\infty(W)$ sont naturellement isomorphes. On obtient alors qu'il correspond à l'opérateur d'entrelacement précédent un unique endomorphisme de $C_c^\infty(W)$. Une fois explicité l'action de $H$ sur $S$, ce qui est fait dans la partie suivante avec le modèle de Schrödinger, cela donne en notant $w = w_X + w_Y \in W$ :
\begin{eqnarray*} f_{s',s}((w,0)) &=& \langle s',\rho_\psi((w,0)) s \rangle  \\
&=& \psi \bigg( \frac{1}{2}\langle w_Y,w_X \rangle \bigg) \int_W \psi([v,w]) \beta(v) f((v,0)) d\mu_W(v).
\end{eqnarray*}
où le crochet, qui est défini par $[v,w]=\langle v_Y,w_X \rangle +\langle w_Y,v_X \rangle$ si $v=v_X+v_Y$, permet d'identifier $W$ et son dual via $w \mapsto [\cdot,w]$, et la fonction $\beta : v \mapsto \psi(\langle v_X , v_Y \rangle )$ est localement constante et possède un inverse (pour la multiplication des fonctions).

Ainsi, le morphisme  $(s',s) \mapsto f_{s',s}$ est la composée de $f  \mapsto \beta f$, qui est un automorphisme de $\mathcal{S}(W)$, de la transformée de Fourier sur $\mathcal{S}(W)$ et de la multiplication par la fonction $(w,0) \mapsto \psi(\frac{1}{2} \langle w_Y , w_X  \rangle )$ qui est localement constante et possède un inverse (pour la multiplication des fonctions). Avec les hypothèses portant sur $R$, la transformée de Fourier est bijective \cite[I.3.10]{vig}. On vient donc de prouver que $(s',s) \mapsto f_{s',s}$ est un isomorphisme de $H \times H$-représentations \textit{i.e.} $S' \otimes S \simeq \textup{ind}_F^H(\psi)$. De l'irréductibilité de $S$, on déduit finalement que $\textup{ind}_F^H(\psi)$ muni de l'action par translation à droite de $H$ est une somme directe de copies de $S$. \end{proof}

Ceci conclut la preuve du Théorème \ref{stone_von_neumann_modulaire_thm}, l'existence provenant du Lemme \ref{stonve_von_neumann_premier_lem} et l'unicité du Lemme \ref{stonve_von_neumann_deuxieme_lem}. \end{proof}

\subsection{Modèles des représentations métaplectiques et propriétés} \label{modeles_rep_metaplectique_section}

Le théorème de Stone-von Neumman modulaire assure que pour tout caractère lisse non trivial $\psi : F \to R$, il existe une unique classe d'isomorphisme de représentations irréductibles de caractère central $\psi$. On appelle \textit{représentation métaplectique associée à $\psi$} cette unique classe d'isomorphisme $\rho_\psi$ et, par extension, toute représentation qui appartient à cette unique classe. En particulier, les représentations $S_A = \textup{ind}_{A_H}^H( \psi_A)$ construites au Lemme \ref{stonve_von_neumann_premier_lem} fournissent des modèles explicites de la représentation métaplectique associée à $\psi$. Ceux-ci ont été particulièrement étudiés quand $A$ est :
\begin{itemize}[label=$\bullet$] 
\item un lagrangien (modèle de Schrödinger) ;
\item un réseau auto-dual sur un corps local non archimédien (modèle latticiel). \end{itemize}

\paragraph{Modèle de Schrödinger.} Soit $W=X+Y$ une polarisation complète de $W$. Le sous-espace vectoriel $X$ est un sous-espace auto-dual de $W$. On pose $S_X = \textup{ind}_{X_H}^H(\psi_X)$ où $\psi_X((w,t)) = \psi(t)$ est un caractère de $X_H = X \times F$. La restriction à $Y$ définit un isomorphisme de représentations $S_X \simeq C_c^\infty(Y)$ où l'action sur le second membre est rendue explicite pour $h = (w_X + w_Y , t ) \in H$ et $f \in C_c^\infty(Y)$ à l'aide d'un calcul simple :
$$\rho_\psi(h) f : y \in Y \mapsto \psi \bigg( \langle y,w_X \rangle  + \frac{1}{2} \langle w_Y,w_X \rangle  + t \bigg) f(y+ w_Y) \in R.$$

\paragraph{Modèle latticiel.} On suppose que $F$ est local non archimédien. On note $\mathcal{P}_F$ l'idéal maximal de son anneau des entiers $\mathcal{O}_F$. Soit $l_\psi$ le niveau -- ou conducteur -- de $\psi$. C'est le plus petit entier relatif $n$ tel qu'on ait $\mathcal{P}_F^n \subset \text{Ker}(\psi)$. Soit $A$ un réseau de $W$ \textit{i.e.} un sous-$\mathcal{O}_F$-module de type fini et de rang maximal. On a alors :
$$A^\perp = \{ w \in W \ | \ \forall a \in A, \ \langle w,a \rangle  \in \mathcal{P}_F^{l_\psi} \}.$$
C'est de nouveau un réseau de $W$, qui peut s'exprimer facilement dans une base hyperbolique. Par exemple, si $(e_i)_{|i| \in [\![1,\dim W]\!]}$ désigne une base hyperbolique de $W$, il existe des entiers relatifs $a_i$ tels que $A=\bigoplus \mathcal{P}_F^{a_i} e_i$. Alors $A^\perp = \bigoplus \mathcal{P}_F^{l_\psi-a_i} e_{-i}$. Il est auto-dual à condition que l'on ait $l_\psi = a_i + a_{-i}$. Quand cette condition est vérifiée, il existe d'après le Lemme \ref{stonve_von_neumann_premier_lem} un caractère $\psi_A$ de $A_H = A \times F$ dont la restriction à $F$ est $\psi$. Enfin, en posant $S_A = \textup{ind}_{A_H}^H (\psi_A)$, la restriction à $W$ induit un isomorphisme de représentations entre $S_A$ et le sous-espace de fonctions de $C_c^\infty(W)$ qui vérifient :
$$f(a+w) = \psi_A(\langle w,a \rangle )f(w) \text{ , pour tout } a \in A \text{ et } w \in W.$$
L'action de $h=(w,t) \in H$ sur un élément $f \in C_c^\infty(W)$ dans ce sous-espace est rendue explicite à l'aide d'un calcul simple : 
$$ \rho_\psi(h)  f : w' \mapsto  \psi(t) \psi \bigg( \frac{1}{2}\langle w',w \rangle \bigg) f(w'+w).$$

\begin{prop} \label{representation_metaplectique_prop} Soit $\rho_\psi$ la représentation métaplectique associée à $\psi$. On a : 
\begin{enumerate}[label=\textup{\alph*)}]
\item les représentations $\rho_\psi^\vee$ et $\rho_{\psi^{-1}}$ sont isomorphes ;
\item les représentations $\rho_{\psi^{-1}} \otimes \rho_\psi$ et $\textup{ind}_F^H( \psi )$ sont isomorphes, où $H \times H$ agit sur la seconde par multiplication à gauche et à droite ;
\item la représentation $\rho_\psi$ est admissible, absolument irréductible et vérifie le lemme de Schur \textit{i.e.} $\textup{End}_{R[H]}(\rho_\psi) = R$  ;
\item toute représentation lisse $(\rho,S)$ de $H$ qui admet pour caractère central $\psi$ est semi-simple ;
\item soient : 
\begin{itemize}[label=$\bullet$]
\item $(W_1,\langle , \rangle_1)$ et $(W_2,\langle , \rangle_2)$ deux espaces symplectiques sur $F$ ;
\item $W= W_1 \oplus W_2$ leur somme orthogonale ;
\item $H(W_1,\langle , \rangle_1)$ et $H(W_2,\langle , \rangle_2)$ les groupes d'Heinseberg associés ;
\item $\rho_\psi^1$ et $\rho_\psi^2$ les représentations métaplectiques respectives associées à $\psi$.
\end{itemize}
Alors la représentation $\rho_\psi^1 \otimes \rho_\psi^2$ s'identifie à la représentation métaplectique du groupe $H(W,\langle , \rangle )$. La représentation suivante en est donc un modèle :
$$(w_1+w_2,t) \mapsto \psi(t) \times \big( \rho_\psi^1((w_1,0)) \otimes \rho_\psi^2((w_2,0)) \big).$$
\end{enumerate} \end{prop}

\begin{proof} a) b) Résultent de la preuve du point b) du Lemme \ref{stonve_von_neumann_deuxieme_lem}.

c) En prenant un modèle explicite $S_A$ où $A$ est un sous-groupe auto-dual de $W$, on a montré au Lemme \ref{stonve_von_neumann_premier_lem} que $S_A = \textup{ind}_{A_H}^H(\psi_A) = \textup{Ind}_{A_H}^H(\psi_A)$, ce qui entraîne que $S_A$ est admissible. Elle est absolument irréductible car l'induction compacte est compatible à l'extension des scalaires. Cela signifie que si $R'$ est une extension de corps de $R$, la représentation $S_A \otimes R'$, qui est à coefficients dans $R'$, est l'induite compacte à coefficients dans $R'$ du caractère $\psi_A$ vu naturellement comme un caractère à coefficients dans $R'$. Pour terminer, la représentation $\rho_\psi$ est admissible et absolument irréductible, donc elle vérifie le lemme de Schur. En effet, il suffit de le prouver sur une clôture algébrique $\bar{R}$ de $R$. Or $S_A \otimes \bar{R}$ est irréductible et admissible, on a donc $\textup{End}_{\bar{R}[H]}( S_A \otimes \bar{R}) = \bar{R}$ d'après \cite[I.6.9]{vig}. Donc le sous-espace vectoriel $\textup{End}_{R[H]}(S_A) \otimes \bar{R}$ de $\textup{End}_{\bar{R}[H]}( S_A \otimes \bar{R})$ est de dimension $1$, d'où l'on déduit que $\textup{End}_{R[H]}(S_A)$ est de dimension $1$ sur $R$. 

d) Par définition, on a $\rho((0,t))=\psi(t) \textup{Id}_S$ pour tout $t \in F$. Si $F$ est fini, comme $H$ est un $p$-groupe fini et que la caractéristique de $R$ est première à $p$ d'après la Remarque \ref{remarque_car_R_non_p_si_psi_existe}, la représentation $\rho$ est semi-simple. Chaque facteur irréductible est de dimension finie et a comme caractère central $\psi$, donc est isomorphe à $\rho_\psi$ par le théorème de Stone-von Neumann modulaire. 

On suppose dorénavant que $F$ est local non archimédien. Soit $A$ un sous-groupe auto-dual de $W$ et $S_A = \textup{ind}_{A_H}^H(\psi_A)$ un modèle de la représentation métaplectique associée à $\psi$. Pour toute représentation lisse $(\rho,S)$ de $H$, on définit sa partie $\psi_A$-isotypique en posant :
$$S^{\psi_A} = \{ s \in S \ | \ \rho(a)s = \psi_A(a) s \textup{ pour tout } a \in A_H \}.$$ 
D'après \cite[I.4.11]{vig} et la Remarque \ref{remarque_car_R_non_p_si_psi_existe}, le foncteur $S \mapsto S^{\psi_A}$ des $\psi_A$-invariants est exact. En effet, $A_H = A \times F$ est filtré par ses sous-groupes compacts ouverts, dont les pro-ordres sont $p^\infty$, donc inversibles dans $R$.

On montre maintenant que $S_A^{\psi_A}$ est de dimension $1$. Pour ce faire, il suffit de remarquer que toute fonction $f \in S_A^{\psi_A}$ a pour support $A_H$ \textit{i.e.} est un multiple de la fonction $\chi_{0,A}$ -- en reprenant les notations de la preuve du Lemme \ref{stonve_von_neumann_premier_lem} -- qui est, à scalaire près, l'unique fonction dans $S_A$ de support $A_H$. Soit donc $f \in S_A^{\psi_A}$. On a pour tout $a \in A$ et tout $w \in W$, l'égalité $(w,0)(a,0)=(a,\langle w,a \rangle )(w,0)$. D'où :
\begin{eqnarray*}
\psi_A((a,0)) f((w,0))=\rho_\psi((a,0))f((w,0))&=&f((w,0)(a,0)) \\
&=&\psi(\langle w,a \rangle ) \psi_A(a) f((w,0)). 
\end{eqnarray*}
On obtient donc que si $f((w,0)) \neq 0$, alors $\psi(\langle w,a \rangle )=1$ pour tout $a\in A$. En d'autres termes, $w \in A^\perp = A$ \textit{i.e.} le support de $f$ est contenu dans $A_H$. Donc $f$ est un multiple de $\chi_{0,A}$. Comme $S_A^{\psi_A}$ est de dimension $1$, la partie $\psi_A$-isotypique de la représentation métaplectique associée à $\psi$ est non triviale. 

Soit $(\rho,S)$ une représentation lisse de $H$ qui admet comme caractère central $\psi$. Vu que toute représentation non nulle admet un sous-quotient irréductible, on déduit du Théorème \ref{stone_von_neumann_modulaire_thm} et de l'exactitude du foncteur de la partie $\psi_A$-isotypique que $S$ est engendré par $S^{\psi_A}$. Il reste à montrer que pour tout $s \in S^{\psi_A}$, la sous-représentation engendrée par $s$ est irréductible.

Soit $A$ un réseau auto-dual de $W$. Alors en choisissant une mesure de Haar $\mu_W$ de $W$, l'espace de fonctions $\mathcal{H}=\textup{ind}_F^H(\psi^{-1})$ est naturellement muni d'une structure d'algèbre à l'aide du produit de convolution :
$$f \star f' : h \mapsto \int_W f((w,0)) f((-w,0)h) d \mu_W(w).$$
Le sous-espace $\mathcal{H}_A$ des fonctions bi-invariantes par $\psi_A^{-1}$ est de dimension $1$ car il s'identifie à la partie $\psi_A \times \psi_A^{-1}$-isotypique de la $H \times H$-représentation $\textup{ind}_F^H(\psi^{-1})$ :
$$\textup{ind}_F^H(\psi^{-1})^{\psi_A \times \psi_A^{-1}} \simeq (\rho_\psi \otimes \rho_{\psi^{-1}})^{\psi_A \times \psi_A^{-1}} = \rho_\psi^{\psi_A} \otimes \rho_{\psi^{-1}}^{\psi_A^{-1}}$$
où le premier isomorphisme provient du point b). On note $\chi_{0,A}$ la fonction de support $A_H$ telle que $\chi_{0,A}(a) = \psi_A^{-1}(a)$ pour tout $a \in A_H$. Quitte à normaliser la mesure $\mu_W$ sur le compact ouvert $A$, on peut supposer que $\chi_{0,A}$ est un idempotent de $\mathcal{H}$. Ainsi, $\rho(\chi_{0,A})$ est un projecteur dont l'image est $S^{\psi_A}$. Soit $s \in S^{\psi_A}$. La représentation engendré par $s$ est irréductible car sa partie $\psi_A$-isotypique est de dimension $1$. En effet, on a $\rho (\mathcal{H}) S = S$ et $(\rho(\mathcal{H})s)^{\psi_A} = \rho( \chi_{0,A}) \rho(\mathcal{H}) s = \rho(\chi_{0,A} \star \mathcal{H} ) \rho(\chi_{0,A}) s = \rho(\mathcal{H}_A)s$. 

e) La représentation $\rho_\psi^1 \otimes \rho_\psi^2$ du groupe $H(W_1,\langle , \rangle_1) \times H(W_2,\langle , \rangle_2)$, dont le pro-ordre est inversible dans $R$ d'après la Remarque \ref{remarque_car_R_non_p_si_psi_existe}, est irréductible car $\rho_\psi^1$ et $\rho_\psi^2$ sont absolument irréductibles et admissibles. Soit $j$ l'homomorphisme de groupes surjectif :
\begin{eqnarray*} 
H(W_1,\langle , \rangle_1) \times H(W_2,\langle , \rangle_2) & \rightarrow & H(W,\langle , \rangle ) \\
((w_1,t_1),(w_2,t_2)) & \mapsto & (w_1 + w_2 , t_1 + t_2)\end{eqnarray*}
dont le noyau est $\{ ((0,t),(0,-t)) \ | \ t \in F \}$. Par conséquent, la représentation $\rho_\psi^1 \otimes \rho_\psi^2$ se factorise par $H(W,\langle , \rangle )$. Celle-ci reste irréductible, elle est bien évidemment lisse et a pour caractère central $\psi$. Donc c'est un modèle de la représentation métaplectique associée à $\psi$ d'après le Théorème \ref{stone_von_neumann_modulaire_thm}. \end{proof}

\subsection{Changement de modèles $S_{A_1} \rightarrow S_{A_2}$} 

L'unicité affirmée par le Théorème \ref{stone_von_neumann_modulaire_thm} entraîne que deux modèles de la représentation métaplectique de $H$ associée à $\psi$ sont isomorphes. Cette partie explique comment construire ces isomorphismes. Soient $A_1$ et $A_2$ deux sous-groupes auto-duaux de $W$. Puisque $\mathcal{O}_F$ est local, principal et complet, il vient en vertu de \cite[I.C.5]{vig} que le sous-groupe $A_1 + A_2$ est d'indice fini dans un sous-groupe fermé de $W$. Il est donc lui-même fermé, ce qui garantit l'égalité $(A_1 \cap A_2)^\perp = A_1 + A_2$. On note $\psi_{A_1}$ et $\psi_{A_2}$ des caractères de $A_{1,H} = A_1 \times F$ et $A_{2,H}=A_2 \times F$ respectivement, dont la restriction à $F$ est $\psi$. On peut alors définir explicitement un morphisme d'entrelacement entre $S_{A_1}$ et $S_{A_2}$ qu'on appelle \textit{changement de modèles}.

\begin{prop}  \label{entrelac_general} Soit $\mu$ une mesure de Haar sur $A_{1,H} \cap A_{2,H} \backslash A_{2,H} = A_1 \cap A_2 \backslash A_2$ à valeurs dans $R$. Soit $\omega \in W$ vérifiant la condition : 
$$\forall a \in A_1 \cap A_2, \ \psi_{A_1}((a,0)) \psi_{A_2}((a,0))^{-1} = \psi(\langle a , \omega \rangle ).$$
Alors l'application $I_{A_1,A_2,\mu,\omega}$ qui à $f \in S_{A_1}$ associe :
$$I_{A_1,A_2,\mu,\omega}f : h \longmapsto \int_{A_{1,H} \cap A_{2,H} \backslash A_{2,H}} \psi_{A_2}(a)^{-1} f((\omega,0) a h) \ d\mu(a)$$
est un isomorphisme de représentations qui engendre $\textup{Hom}_{R[H]}(S_{A_1},S_{A_2}) \simeq R$. \end{prop}

\begin{proof} Il y a un léger abus de notation puisque $a \mapsto \psi_{A_2}(a)^{-1} f((\omega,0) a h)$ est une fonction sur $A_{2,H}$. Mais cette fonction est invariante à gauche par $A_{1,H} \cap A_{2,H}$. Elle est bien à support compact modulo $A_{1,H} \cap A_{2,H}$ car l'hypothèse $A_1 + A_2$ fermé assure que l'image de $A_{2,H}$ dans $A_{1,H} \backslash H$ soit fermée. D'où la consistance de la notation et le caractère bien défini de l'intégrale. On vérifie sans difficulté que la fonction ainsi définie appartient à $S_{A_2}$. L'application de la proposition est donc un morphisme d'entrelacement entre deux représentations irréductibles de $H$. D'après le lemme de Schur, c'est un isomorphisme à condition qu'il soit non nul. Pour ce faire, il suffit de remarquer que l'image d'une fonction indicatrice $\chi_{w,L}$ de $S_{A_1}$ est bien non nulle. Enfin, l'espace vectoriel $\textup{Hom}_{R[H]}(S_{A_1},S_{A_2})$ est de dimension $1$ d'après le point c) de la Proposition \ref{representation_metaplectique_prop}. \end{proof}

\begin{rem} \label{simplification_I_A1_A2_rem} L'expression du morphisme d'entrelacement précédent se simplifie grandement si $\psi_{A_1}((a,t))=\psi(t)$ et $\psi_{A_2}((a,t))=\psi(t)$, auquel cas $\omega=0$ convient. On peut toujours trouver de tels caractères qui étendent $\psi$ << trivialement >> si l'une de ces conditions est vérifiée :
\begin{itemize}[label=$\bullet$]
\item $p \neq 2$ ;
\item $A_1$ et $A_2$ sont des lagrangiens.
\end{itemize} 
Dans ce cas, on a $I_{A_1,A_2,\mu} \in \textup{Hom}_{R[H]}(S_{A_1},S_{A_2})$ où : 
$$I_{A_1,A_2,\mu} f : h \longmapsto \int_{A_1\cap A_2  \backslash A_2} f((a,0)h) d \mu(a).$$ \end{rem}

\section{Représentation de Weil modulaire} \label{representation_de_weil_mod_section}

\subsection{Groupe métaplectique et représentation de Weil modulaires} \label{gp_met_et_rep_de_Weil_subsection}

Soient $\psi \in \hat{F}_R$ un caractère non trivial et $(\rho_\psi,S) \in \textup{Rep}_R(H)$ un modèle de la représentation métaplectique associée à $\psi$ (cf. Section \ref{modeles_rep_metaplectique_section}). Le groupe symplectique $\textup{Sp}(W)$ agit sur le groupe d'Heisenberg $H$ via :
$$\begin{array}{ccc}
G \times H & \rightarrow & H \\
(g,(w,t)) & \mapsto & g \cdot (w,t) = (gw,t) \end{array} .$$
Cette action laisse invariant le centre de $H$. Elle préserve la représentation métaplectique au sens suivant. Pour $g \in \textup{Sp}(W)$, on note $\rho_\psi^g : h \mapsto \rho_\psi(g^{-1} \cdot h)$. Alors $(\rho_\psi^g,S) \in \textup{Rep}_R(H)$ est un autre modèle de la représentation métaplectique associée à $\psi$. Par conséquent, le Théorème \ref{stone_von_neumann_modulaire_thm} affirme que les représentations $\rho_\psi$ et $\rho_\psi^g$ sont isomorphes. Il existe donc $M_g \in \textup{GL}_R(S)$, unique à scalaire près d'après le point c) la Proposition \ref{representation_metaplectique_prop}, tel que pour tout $h \in H$ : 
$$M_g  \rho_\psi(h) M_g^{-1}=\rho_\psi^g( h).$$
Pour tout $g \in \textup{Sp}(W)$, on suppose un tel $M_g$ fixé. On obtient ainsi une représentation projective du groupe symplectique qui ne dépend pas du choix des $M_g$ précédents :
$$g \in \textup{Sp}(W) \mapsto \textsc{red}(M_g) \in \text{PGL}_R(S).$$
On la note $\sigma_S$. Un argument classique consiste à relever cette représentation projective en une vraie représentation d'une extension centrale du groupe symplectique. En effet, on considère :
$$\xymatrix{
\widetilde{\textup{Sp}}_{\psi,S}^R(W) \ar[d]^{p_S} \ar[r]^{\omega_{\psi,S}} & \textup{GL}_R(S) \ar[d]^{\textsc{red}} \\
\textup{Sp}(W) \ar[r]^{\sigma_S} & \textup{PGL}_R(S) 
}$$
où $\widetilde{\textup{Sp}}_{\psi,S}^R(W) = \textup{Sp}(W) \times_{\textup{PGL}_R(S)} \textup{GL}_R(S)$ est le produit fibré dans le catégorie des groupes de $\sigma_S$ et de $\textsc{red}$ au-dessus de $\textup{PGL}_R(S)$. Les morphismes de groupes $p_S$ et $\omega_{\psi,S}$ désignent respectivement la première et la seconde projection.

\begin{defi} \label{representation_de_Weil_def1} On donne à la représentation $(\omega_{\psi,S},S)$ le nom de \textit{représentation de Weil modulaire sur $W$ associée à $\psi$ et $S$}. \end{defi}

\begin{prop} \label{groupe_metaplectique_prop} Le groupe $\widetilde{\textup{Sp}}_{\psi,S}^R(W)$ est une extension centrale de $\textup{Sp}(W)$ par $R^\times$ et s'inscrit dans la suite exacte :
$$ 1 \to R^\times \overset{i_S}{\to} \widetilde{\textup{Sp}}_{\psi,S}^R(W) \overset{p_S}{\to} \textup{Sp}(W) \to 1$$
où $i_S : \lambda \mapsto (\textup{Id}_W, \lambda . \textup{Id}_S)$.  Pour deux modèles $S$ et $S'$ de $\rho_\psi$, et pour tout isomorphisme de représentations $\phi : S \to S'$, l'isomorphisme d'extension centrale :
$$\begin{array}{cccl}	
  \widetilde{\textup{Sp}}_{\psi,S}^R(W) & \to & \widetilde{\textup{Sp}}_{\psi,S'}^R(W) \\
  (g,M) & \mapsto & (g, \phi M \phi^{-1}) 
\end{array}$$  
ne dépend pas du choix de $\phi$. On le note $\Phi_{S,S'}$. Sauf dans le cas $F = \mathbb{F}_3$ et $\dim_F W =2$, il existe un unique isomorphisme d'extension centrale entre $\widetilde{\textup{Sp}}_{\psi,S}^R(W)$ et $\widetilde{\textup{Sp}}_{\psi,S'}^R(W)$, et celui-ci est donné par $\Phi_{S,S'}$. \end{prop}

\begin{proof} Le noyau de $p_S$ est donné par la flèche $i_S$, d'où la suite exacte en question. Ensuite, d'après le point c) de la Proposition \ref{representation_metaplectique_prop}, l'espace vectoriel $\textup{Hom}_{R[H]}(S,S')$ est de dimension $1$ sur $R$. Donc tout morphisme non nul est un isomorphisme de représentations et l'application $M \mapsto \phi M \phi^{-1}$ ne dépend pas du choix du morphisme non nul $\phi \in \textup{Hom}_{R[H]}(S,S')$. On appelle $\Phi_{S,S'}$ le morphisme de groupes $(g,M) \mapsto (g,\phi M \phi^{-1})$, qui est bien un isomorphisme d'extensions centrales $\phi_{S,S'}$ \textit{i.e.} un isomorphisme de groupes tel que $p_{S'} \circ \Phi_{S,S'} = p_S$ et $\Phi_{S,S'} \circ i_S = i_{S'}$.

Pour terminer, le groupe symplectique est parfait \textit{i.e.} égale à son groupe dérivé, sauf dans le cas exceptionnel où $\textup{Sp}(W)=\text{SL}_2(\mathbb{F}_3)$. Or, deux isomorphismes d'extensions centrales différent par un caractère $\chi : \textup{Sp}(W) \rightarrow R^\times$. Si l'on exclut le cas exceptionnel, ce caractère est forcément trivial et $\Phi_{S,S'}$ est l'unique isomorphisme d'extensions centrales en question. \end{proof}

En d'autres termes, la classe d'isomorphisme de l'extension centrale $\widetilde{\textup{Sp}}^R_{\psi,S}(W)$ ne dépend pas du choix du $S$ et il existe des identifications canoniques données par les isomorphismes d'extensions centrales $\Phi_{S,S'}$. 

\begin{defi} \label{groupe_metaplectique_definition} On appelle \textit{groupe métaplectique modulaire sur $W$ associé à $\psi$} la classe d'isomorphisme définie par les extensions centrales précédentes. Par extension, tout élément de cette classe d'isomorphisme est encore désigné comme \og groupe métaplectique \fg{}. \end{defi}

Dans le cas exceptionnel seulement, il existe d'autres isomorphismes d'extensions centrales. En reprenant les notations du diagramme défini plus haut, les représentations $(\omega_{\psi,S'} \circ \Phi_{S,S'},S')$ et $(\omega_{\psi,S},S)$ sont isomorphes. 

\begin{theo} \label{groupe_metaplectique_theoreme} On note $\widehat{\textup{Sp}}_{\psi,S}^R(W)$ le groupe dérivé de $\widetilde{\textup{Sp}}_{\psi,S}^R(W)$.
\begin{enumerate}[label=\textup{\alph*)}]
\item Si $F$ est fini, ou si la caractéristique $\ell$ de $R$ est $2$, il existe un morphisme de groupes (injectif) :
$$\textup{Sp}(W) \to \widetilde{\textup{Sp}}_{\psi,S}^R(W)$$
qui scinde la suite exacte de la Proposition \ref{groupe_metaplectique_prop}. Hormis dans le cas exceptionnel $F= \mathbb{F}_3$ et $\textup{dim}_F W= 2$ où le groupe symplectique n'est pas parfait, cet homomorphisme est unique. Le plongement précédent de $\textup{Sp}(W)$ induit un isomorphisme :
$$\widehat{\textup{Sp}}_{\psi,S}^R(W) \simeq [\textup{Sp}(W),\textup{Sp}(W)]$$
où $[\textup{Sp}(W),\textup{Sp}(W)]=\textup{Sp}(W)$, sauf dans le cas exceptionnel.
\item Si $F$ est local non archimédien et $\ell \neq 2$, un tel homomorphisme n'existe pas. En revanche, on a une suite exacte :
$$1 \to \{ \pm 1 \} \overset{i_S}{\to} \widehat{\textup{Sp}}_{\psi,S}^R(W) \overset{p_S}{\to} \textup{Sp}(W) \to 1.$$
Le groupe $\widehat{\textup{Sp}}_{\psi,S}^R(W)$ est l'unique sous-groupe de $\widetilde{\textup{Sp}}_{\psi,S}^R(W)$ s'inscrivant dans une telle suite exacte.
\item Le groupe $\widehat{\textup{Sp}}_{\psi,S}^R(W)$ est un groupe parfait. \end{enumerate} \end{theo}

\begin{proof} \textup{a)} Si un tel morphisme $\sigma$ existe, il induit un isomorphisme d'extension centrale :
$$(g,\lambda) \in \textup{Sp}(W) \times R^\times \to \sigma(g) i_S(\lambda) \in \widetilde{\textup{Sp}}_{\psi,S}^R(W).$$
Or, l'ensemble de ces isomorphismes d'extensions centrales est en bijection avec l'ensemble des caractères $\textup{Sp}(W) \rightarrow R^\times$, qui est trivial sauf dans le cas exceptionnel. De plus, l'identification des groupes dérivés résulte de cet isomorphisme. Le morphisme en question existe bien quand $F$ est fini en vertu de \cite[Th. 3.3.]{stei} : le groupe symplectique est alors son propre revêtement universel au sens de \cite{moore}. Cela signifie que toute extension centrale du groupe symplectique est scindée. On traite le cas local non archimédien et $R$ de caractéristique $2$ dans la preuve qui suit.

\textup{b)} On suppose maintenant que $F$ est local non archimédien. Quand $R=\mathbb{C}$, \cite{weil} montre qu'il existe un caractère $A_{\mathbb{C}} \to \mathbb{C}^\times$ d'une extension centrale $A_{\mathbb{C}}$ de $\textup{Sp}(W)$ par $\mathbb{C}^\times$, dont le noyau est $\widehat{\textup{Sp}}(W)$, l'unique extension centrale non triviale de $\textup{Sp}(W)$ par $\{\pm 1\}$. À ce titre, ce groupe est parfait. Il est donc le groupe dérivé de $A_\mathbb{C}$. Ce résultat a été généralisé \cite[\S 5]{ct} au cas d'un anneau intègre $R$ de caractéristique $\ell$ différente de $p$ et pour lequel il existe un caractère non trivial $\psi : F \to R^\times$. Cela constitue précisément les hypothèses qui portent sur le corps $R$ dans cet article. On en déduit qu'il existe un caractère $A_R \to R^\times$ d'une extension centrale $A_R$ de $\textup{Sp}(W)$ par $R^\times$ dont le noyau est $\widehat{\textup{Sp}}(W)$ quand $l \neq 2$ et $\textup{Sp}(W)$ quand $l = 2$. Ces deux groupes sont parfaits. On explique dans l'Annexe \ref{lien_avec_ct_section} comment traduire ces résultats dans le contexte ici présent en identifiant $A_R$ et $\widetilde{\textup{Sp}}_{\psi,S}^R(W)$.

On obtient donc un caractère $\widetilde{\textup{Sp}}_{\psi,S}^R(W) \to R^\times$ dont le noyau est un groupe parfait, donc égal au groupe dérivé. Quand $\ell =2$, on applique le même argument que dans le cas fini pour montrer que le morphisme d'inclusion est l'unique morphisme qui scinde la suite exacte de la Proposition \ref{groupe_metaplectique_prop}. Quand $\ell \neq 2$, un tel sous-groupe parfait est unique. En effet, le groupe $\widetilde{\textup{Sp}}_{\psi,S}^R(W)$ ne contient pas le groupe $\textup{Sp}(W)$ comme sous-groupe, d'après \cite[Th. 5.4]{ct}, donc le sous-groupe en question doit être l'unique extension centrale non triviale de $\textup{Sp}(W)$ par $\{ \pm 1 \}$, qui est un groupe parfait, donc contenu dans le groupe dérivé. Par suite, il est égal au groupe dérivé. \end{proof}

On munit $R$ de la topologie discrète. Soit $S$ un espace vectoriel sur $R$. On le munit également de la topologie discrète. La topologie compacte-ouverte sur $\textup{GL}_R(S)$ est engendrée par la prébase $S_{s,s'} = \{g \in \textup{GL}_R(S) \ | \ g s = s'\}$ où $s$ et $s '$ parcourent $S$. Alors, une représentation d'un groupe topologique $G$ est lisse si et seulement si le morphisme $G \to \textup{GL}_R(S)$ associé est continu.

\begin{prop} \label{gp_met_ext_centrale_top_prop} Le groupe métaplectique $\widetilde{\textup{Sp}}_{\psi,S}^R(W)$ est le produit fibré dans la catégorie des groupes topologiques des morphismes de groupes continus $\sigma_S$ et $\textsc{red}$. À ce titre, c'est un sous-groupe topologique de $\textup{Sp}(W) \times \textup{GL}_R(S)$ qui est une extension centrale topologique de $\textup{Sp}(W)$ par $R^\times$. De plus, les isomorphismes $\Phi_{S,S'}$ décrits précédemment sont des isomorphismes d'extensions centrales topologiques. \end{prop}

\begin{proof} Dans le cas où $F$ est fini, la topologie en jeu est partout la topologie discrète. Les groupes en question sont donc automatiquement des groupes topologiques. On suppose dorénavant que $F$ est local non archimédien. Tout d'abord, $\textsc{red}$ est continu par définition de la topologie quotient. Il s'agit ensuite de montrer que $\sigma_S$ est continue. Il suffit de trouver un seul modèle de la représentation métaplectique pour lequel ce soit vrai. En effet, le morphisme de groupes $\Phi_{S,S'}$ induit un isomorphisme de groupes topologiques :
$$\begin{array}{cccl}	
 \textup{GL}_R(S) & \to & \textup{GL}_R(S') \\
  M & \mapsto & \phi M \phi^{-1} 
\end{array}$$
que l'on note encore $\Phi_{S,S'}$. Celui-ci induit à son tour un isomorphisme de groupes topologiques entre $\textup{PGL}_R(S)$ et $\textup{PGL}_R(S')$. On en déduit que $\sigma_S$ est continu si et seulement $\Phi_{S,S'} \circ \sigma_S = \sigma_{S'}$ est continue. La proposition résulte alors de :
\begin{lem} \label{modele_latticiel_continu_lisse_prop} Soit $L$ un réseau auto-dual de $W$. Soit $S_L$ le modèle latticiel associé à $L$. Alors $\sigma_{S_L}$ est continu. \end{lem}

\begin{proof} Soit $(\rho_\psi,S_L)$ le modèle latticiel de la représentation métaplectique défini dans la Section \ref{modeles_rep_metaplectique_section}. Soit $K$ le stabilisateur de $L$ dans $\textup{Sp}(W)$. C'est un sous-groupe compact ouvert. Soient $k \in K$ et $N_k$ l'application linéaire :
$$\begin{array}{cccl}	
\textup{ind}_{L_H}^H(\psi_L) & \to & \textup{ind}_{L_H}^H(\psi_L^k) \\
 f & \mapsto & k \cdot f 
\end{array}$$
où $k \cdot f : h \mapsto f(k^{-1} \cdot h)$ et $\psi_L^k : (l,t) \mapsto \psi_L((k^{-1} l, t))$. 

Quand $p \neq 2$, on peut choisir le caractère $\psi_L$ de sorte que $\psi_L((l,t))=\psi(t)$ d'après la Remarque \ref{simplification_I_A1_A2_rem}. Ainsi, on a $\psi_L^k=\psi_L$ pour tout $k \in K$. Alors $N_k \in \textup{GL}_R(S_L)$ vérifie :
$$N_k \circ \rho_\psi = \rho_\psi^{k^{-1}} \circ N_k.$$
On déduit de cette dernière égalité qu'on a $N_k \circ \rho_\psi^k = \rho_\psi \circ N_k$. En d'autres termes, $N : k \in K \mapsto N_k \in \textup{GL}_R(S_L)$ est une représentation de $K$ qui relève $\sigma_{S_L}|_K$ au sens où $\sigma_{S_L}(k) = \textsc{red}(N_k)$ pour tout $k \in K$. De plus, cette représentation $N$ est lisse. En effet, il suffit d'observer que l'action de $k \in K$ sur $f \in S_L \simeq C_c^{\infty}(W)$ est donnée par $k.f : w \mapsto f(k^{-1} w)$. 

Quand $p = 2$, le caractère $\psi_L$ ne peut pas être étendu << trivialement >> à $L_H$. Il existe en revanche un entier positif $n$ de sorte que $\varpi^n L \times \textup{Ker}(\psi)$ soit un sous-groupe de $\textup{Ker}(\psi_L)$. On fixe un tel $n$. Soit $K_n$ le noyau du morphisme de réduction $K \to \textup{GL}(L / \varpi^n L)$. Alors par les mêmes arguments que précédemment, l'application linéaire $N_k$ est dans $\textup{GL}_R(S_L)$ pour tout $k \in K_n$. La représentation $k \in K_n \mapsto N_k \in \textup{GL}_R(S)$ est lisse et relève $\sigma_{S_L}|_{K_n}$.

Par conséquent, on a montré qu'il existe un sous-groupe compact ouvert $K$ de $\textup{Sp}(W)$ et une représentation lisse de $K$ notée $\sigma$ tels que $\sigma_{S_L}(k) = \textsc{red}( \sigma(k))$. L'application $\sigma_{S_L}$ est donc continue au voisinage de $1_{\textup{Sp}(W)}$ puisque $\sigma$ et $\textsc{red}$ le sont, on en déduit que $\sigma_{S_L}$ est continu partout puisque c'est un morphisme de groupes. \end{proof}

On peut donc former le produit fibré $\widetilde{\textup{Sp}}_{\psi,S}^R(W)$ de $\sigma_S$ et $\textsc{red}$ au-dessus de $\textup{PGL}_R(S)$ dans la catégorie des groupes topologiques. Par définition, il est un sous-groupe topologique du groupe topologique $\textup{Sp}(W) \times \textup{GL}_R(S)$ qui est muni de la topologie produit. Enfin, l'isomorphisme d'extension centrale $\Phi_{S,S'}$ est bicontinu d'après la discussion qui précède le lemme. \end{proof}

En particulier, le morphisme de groupes $\omega_{\psi,S} : \widetilde{\textup{Sp}}_{\psi,S}^R(W) \to \textup{GL}_R(S)$ est continu en tant que seconde projection du produit fibré. Comme mentionné plus haut, ce dernier point est équivalent à :

\begin{cor} \label{rep_de_weil_lisse_cor1} La représentation de Weil modulaire $(\omega_{\psi,S},S)$ est lisse. \end{cor}

La première projection $p_S$ est également continue et fournit une structure plus riche :

\begin{prop} \label{sections_globales_métaplectiques_prop} La projection $p_S : \widetilde{\textup{Sp}}_{\psi,S}^R(W) \to \textup{Sp}(W)$ muni le groupe métaplectique d'une structure de revêtement topologique de base $\textup{Sp}(W)$ et de fibre $R^\times$. De plus, il existe des sections globales de $p_S$. \end{prop}

\begin{proof} Tout d'abord, la première assertion implique la seconde. En effet, si $p_S$ définit un revêtement topologique, alors $p_S$ admet des sections locales par définition. On va construire une section globale en exploitant le fait que la topologie sur $\textup{Sp}(W)$ est totalement discontinue. Par hypothèse, il existe un sous-groupe ouvert $U$ de $\textup{Sp}(W)$ et une section locale $\sigma$ de $p_S$ définie sur $U$ de sorte que l'application :
$$(g,\lambda) \in U \times R^\times \mapsto \sigma(g) \lambda \in p_S^{-1}(U)$$
soit un homéomorphisme. On considère le quotient $\textup{Sp}(W)/U$ qui est un ensemble discret. On se dote d'un système de représentants des classes à gauche modulo $U$ \textit{i.e.} d'une application $h : \textup{Sp}(W)/U \to \textup{Sp}(W)$ telle que $h(i) U = i$. L'image $h(\textup{Sp}(W)/U)$ de $h$ est un sous-ensemble discret de $\textup{Sp}(W)$. Alors, pour toute section $A$ de $p_S$ sur le sous-ensemble discret $h(\textup{Sp}(W)/U)$, \textit{i.e.} telle que $p_S(A_{h(i)})=h(i)$ pour tout $h(i) \in h(\textup{Sp}(W)/U)$, l'application suivante est une section globale de $p_S$ :
$$g \in \textup{Sp}(W) \mapsto A_{h(gU)} \times \sigma(h(gU)^{-1} g) \in \widetilde{\textup{Sp}}_{\psi,S}^R(W).$$

Maintenant, il reste à montrer la première partie de l'énoncé, à savoir que $p_S$ admet des sections locales. Cela suffit car les fibres de $p_S$ sont bien isomorphes à $R^\times$ muni de la topologie discrète. Comme $p_S$ est un morphisme de groupes, on peut seulement concentrer ses efforts au voisinage de $1_{\textup{Sp}(W)}$. Or, dans la preuve du Lemme \ref{modele_latticiel_continu_lisse_prop}, on a montré qu'il existait un sous-groupe ouvert $K$ de $\textup{Sp}(W)$ et une représentation lisse $\sigma : K \to \textup{GL}_R(S_L)$ de sorte que $\sigma_{S_L}(k) = \textsc{red} (\sigma(k))$. L'application suivante est alors un homéomorphisme de groupes :
$$(k,\lambda) \in K \times R^\times \mapsto (k, \lambda \sigma(k)) \in p_{S_L}^{-1}(K).$$
Donc $k \mapsto (k,\sigma(k))$ est une section locale de $p_{S_L}$. 

Enfin, comme les applications $\Phi_{S_L,S}$ sont des isomorphismes d'extensions centrales topologiques, on récupère le résultat pour tout modèle $S$ de la représentation métaplectique. \end{proof}

\begin{rem} En particulier, toute section globale $\sigma$ de $p_S$ définit un isomorphisme d'extensions centrales topologiques :
$$(g, \lambda) \in \textup{Sp}(W) \times_c R^\times \mapsto \sigma(g) (1_{\textup{Sp}(W)}, \lambda \textup{Id}_S) \in \widetilde{\textup{Sp}}_{\psi,S}^R(W)$$
le groupe de gauche étant muni de la topologie produit et sa loi donnée par le $2$-cocyle :
$$c(g,g') = i_S^{-1}(\sigma(g) \sigma(g') \sigma(g g')^{-1}).$$
De plus, l'application $(g,g') \in G \times G \mapsto c(g,g') \in R^\times$ est localement constante car continue. \end{rem}

\begin{cor} \
\begin{enumerate}[label=\textup{\alph*)}]
\item Le groupe métaplectique $\widetilde{\textup{Sp}}_{\psi,S}^R(W)$ est un groupe localement profini. 
\item Le groupe $\widehat{\textup{Sp}}_{\psi,S}^R(W)$ est un sous-groupe ouvert de $\widetilde{\textup{Sp}}_{\psi,S}^R(W)$. \end{enumerate} \end{cor}

\begin{proof} Il ne s'agit de prouver ces énoncés que quand $F$ est local non archimédien puisque, quand $F$ est fini, la topologie n'intervient pas.

\textup{a)} Dans la preuve de la Proposition \ref{sections_globales_métaplectiques_prop}

lm
 précédente, les homéomorphismes de groupes :
$$(k,\lambda) \in K \times R^\times \mapsto (k, \lambda \sigma(k)) \in p_S^{-1}(K)$$
fournissent une base de voisinages de $(1_{\textup{Sp}(W)},\textup{Id}_S)$ dans le groupe métaplectique constituée d'ouverts compacts. Par conséquent, le groupe métaplectique est un groupe localement profini. 

\textup{b)} On montrera lors de la Proposition \ref{scindage_groupe_compact_prop} qu'il existe un sous-groupe ouvert de $K$ dont l'image par l'homéomorphisme précédent :
$$k \mapsto (k,\sigma(k))$$
est contenue dans $\widehat{\textup{Sp}}_{\psi,S}^R(W)$. Ce sous-groupe contenant un ouvert, il est donc bien ouvert. Cet argument montre même que ce groupe possède une structure de revêtement topologique de base $\textup{Sp}(W)$ et de fibre $\{ \pm 1 \}$ quand  $\ell \neq 2$. \end{proof}

\subsection{Modèles de la représentation de Weil modulaire}

\subsubsection{Décalque des représentations $S_A$} \label{decalque_des_rep_S_A_section}

On note $(\rho_d,\textup{Ind}_F^H(\psi))$ la représentation où $H$ agit par translation à droite. Pour tout sous-groupe auto-dual $A$ de $W$, on peut voir le modèle $(\rho_\psi,S_A)=(\rho_\psi,\textup{ind}_{A_H}^H(\psi_A))$ de la représentation métaplectique comme contenu canoniquement dans $\rho_d$. 

L'action d'un élément $g \in \textup{Sp}(W)$ sur $H$ donne un isomorphisme :
$$\begin{array}{ccccl}	
I_g : & \textup{ind}_{A_H}^H(\psi_A) & \to & \textup{ind}_{g A_H}^H(\psi_A^g) \\
 & f & \mapsto & g \cdot f 
\end{array}$$
où $g \cdot f : h \mapsto f(g^{-1} \cdot h)$ et $\psi_A^g : a \in gA_H \mapsto \psi_A(g^{-1} \cdot h) \in R^\times$. Alors, pour tout $h \in H$ :
$$I_g \circ \rho_d(h) = \rho_d (g^{-1} \cdot h) \circ I_g.$$ 
En composant avec les morphismes d'entrelacement $I_{gA,A,\mu,\omega}$ de la Proposition \ref{entrelac_general}, on obtient : 
$$S_A \overset{I_g}{\longrightarrow} S_{gA} \overset{I_{gA,A,\mu,\omega}}{\longrightarrow} S_A,$$
qui vérifie :
$$I_{gA,A,\mu,\omega} \circ I_g \circ \rho_d(h) = \rho_d(g^{-1} \cdot h) \circ I_{gA,A,\mu,\omega} \circ I_g.$$
Ainsi :
$$(g,I_{gA,A,\mu,\omega} \circ I_g) \in \widetilde{\textup{Sp}}_{\psi,S_A}^R(W).$$
En particulier, comme pour tout $g \in \textup{Stab}(A) \cap \textup{Stab}(\psi_A)$ l'opérateur $I_{gA,A,\mu,\omega}$ est un multiple de l'identité, l'application :
$$g \in \textup{Stab}(A) \cap \textup{Stab}(\psi_A) \mapsto (g , I_g) \in \widetilde{\textup{Sp}}_{\psi,S_A}^R(W)$$
est un morphisme de groupes.

\paragraph{Modèle de Schr\"odinger.} Soit $X$ un lagrangien de $W$. On considère le modèle de Sch\"odinger sur $X$ de la représentation métaplectique associé à $\psi$. On rappelle que dans les modèles de Schr\"odinger, le caractère $\psi_X$ est trivial sur $X$ \textit{i.e.} $\psi_X((x,t))=\psi(t)$ pour tout $x \in X$ et $t \in F$. Alors, en notant $P(X)$ le parabolique de $\textup{Sp}(W)$ qui stabilise le sous-espace $X$ et le caractère $\psi_X$, on a d'après la remarque précédente un morphisme de groupes :
$$p \in P(X) \mapsto (p,I_p) \in  \widetilde{\textup{Sp}}_{\psi,S_X}^R(W).$$
En choisissant une polarisation complète $W=X+Y$, on identifie $S_X$ avec $C_c^\infty(Y)$ via la restriction des fonctions à $Y$. Alors, en notant $p = \left[ \begin{array}{cc} a & b \\
0 & (a^*)^{-1} \end{array} \right] \in P(X)$ :
$$I_p f : (y,0) \mapsto \psi(\frac{1}{2} \langle a^* y,b^* y\rangle ) f((a^* y,0)).$$
En particulier, cela donne des plongements du Levi $M(X)$ et du radical unipotent $N(X)$ dans le groupe métaplectique. Avant de décrire l'image de ce plongement, on remarque que pour un élément $g = \left[ \begin{array}{cc} 0 & c \\
(c^*)^{-1} & 0 \end{array} \right] \in \textup{Sp}(W)$, on a d'après la Remarque \ref{simplification_I_A1_A2_rem} :
$$(g,I_{Y,X,\mu_X} \circ I_g) \in \widetilde{\textup{Sp}}_{\psi,S_X}^R(W)$$
où $I_{Y,X,\mu_X}$ est simplement la transformée de Fourier. Un calcul donne :
$$I_{Y,X,\mu_X} \circ I_g f : (y,0) \mapsto \int_X \psi(\langle x,y\rangle ) f(c^{-1} x) d \mu_X (x).$$

\begin{rem} \label{extension_des_scalaires_reduction_des_scalaires_rem} Ces formules fournissent des compatibilités à l'extension et à la réduction des scalaires. On développe maintenant ces deux points.

\begin{itemize}[label=$\bullet$]
\item Soit $R'$ une extension de corps de $R$. Alors $S_X^{R'} = S_X \otimes_R R'$ est un modèle de la représentation métaplectique associée à $\psi$ et à coefficients dans $R'$. On a alors une inclusion évidente $i_{R'}$ induite par celle de $S_X$ dans $S_X^{R'}$ :
$$\widetilde{\textup{Sp}}_{\psi,S_X}^R(W) \subset \widetilde{\textup{Sp}}_{\psi,S_X \otimes_R R'}^{R'}(W)$$
compatible aux applications de projection. On écrit par abus de notation :
$$\omega_{\psi,S_X^{R'}} = \omega_{\psi,S_X} \otimes_R R'.$$
\item Soit maintenant $\mathcal{A}$ un anneau local de corps des fractions $R$ et de corps résiduel $k$ de caractéristique résiduelle $\ell \neq p$. Soit $S_X^\mathcal{A}$ le sous-espace de $S_X$ qui s'identifie à $C_c^\infty(Y,\mathcal{A})$ dans $C_c^\infty(Y)$. D'après les formules précédentes, c'est un $\mathcal{A}$-module qui est stable par l'action du groupe métaplectique réduit. On note cette représentation  :
$$(\omega_{\psi,S_X}^{\mathcal{A}},S_X^{\mathcal{A}}) \in \textup{Rep}_{\mathcal{A}}(\widehat{\textup{Sp}}_{\psi,S_X}^R(W))$$
car l'action est la restriction de celle de $\omega_{\psi,S_X}$. Si $\psi$ est à valeurs dans $\mathcal{A}$, alors sa réduction $\psi_k$ est bien défini et est un caractère non trivial à valeurs dans $k$. Cette situation est toujours réalisable quand $\mathcal{A}$ est un anneau de valuation par exemple. Soit $S_X^k$ le modèle de la représentation métaplectique associé à $\psi_k$ et $X$ à valeurs dans $k$. Alors on a :
$$S_X^{\mathcal{A}} \otimes_\mathcal{A} k = S_X^k$$
en tant que représentation du groupe d'Heisenberg. Cela induit un morphisme :
$$i_k : \widehat{\textup{Sp}}_{\psi,S_X}^R(W) \to \widehat{\textup{Sp}}_{\psi_k,S_X^k}^k(W)$$
qui est surjectif et compatible aux projections. On écrit par abus de notation :
$$\omega_{\psi,S_X}^\mathcal{A} \otimes_{\mathcal{A}} k = \omega_{\psi_k,S_X^k}.$$ \end{itemize} \end{rem} 

\paragraph{Modèle de Schrödinger "mixte".} \label{schrodinger_mixte} Soit $X$ un sous-espace totalement isotrope non nul et non maximal de $W$. Soit $Y$ un sous-espace totalement isotrope de $W$ dual de $X$ de sorte que $W= X  + W^0 + Y$, où $W^0$ est le sous-espace symplectique orthogonal, dans $W$, au sous-espace symplectique $X+Y$. Soit $(\rho_\psi^0,S^0)$ un modèle de la représentation métaplectique de $H(W^0,\langle , \rangle )$ associé à $\psi$. On réalise la représentation métaplectique de $H(X+Y,\langle , \rangle )$ à travers le modèle de Schr\"odinger $C_c^\infty(Y)$ associé à $X$ et $\psi$. Alors, d'après le point e) de la Proposition \ref{representation_metaplectique_prop}, la représentation $S=C_c^\infty(Y) \otimes S^0$ est un modèle de la représentation métaplectique de $H(W,\langle , \rangle )$ associé à $\psi$. On note $P(X)$ le parabolique de $\textup{Sp}(W)$ qui stabilise le sous-espace $X$ et $j : P(X) \rightarrow \textup{Sp}(W^0)$ la projection naturelle sur la partie symplectique du Levi $M(X)$ de $P(X)$. Cette dernière possède une section naturelle donnée par l'inclusion de $\textup{Sp}(W^0)$ dans $M(X)$. Par abus de notation $p \in P(X) \mapsto (p u^{-1} , u) \in \textup{Ker} (j) \rtimes \textup{Sp}(W^0)$ est un isomorphisme de groupes. On vérifie sans difficulté :

\begin{lem} \label{isomorphisme_parabolique_gp_met_lem} On a un isomorphisme de groupes :
$$\begin{array}{cccl}	
P(X) \times_{\textup{Sp}(W^0)}  \widetilde{\textup{Sp}}_{\psi,S^0}^R(W^0)  & \overset{\sim}{\to} & p_S^{-1}(P(X)) \\
(p,\tilde{u}) & \mapsto & (p , I_{p u^{-1}} \otimes \omega_{\psi,S^0} (\tilde{u}))
\end{array}$$
où le produit fibré de gauche est donné par $j$ et $p_{S^0}$. \end{lem}
\noindent En particulier, on peut considérer l'action de $\textup{Ker} (j)$ via le morphisme de groupes :
$$p \in \textup{Ker} j \mapsto (p,I_p \otimes \textup{Id}_{S^0}) \in \widetilde{\textup{Sp}}_{\psi,S}^R(W).$$
On donne maintenant les actions de sous-groupes remarquables -- même si le dernier n'en est pas un ! -- pour $f \in C_c^\infty(Y) \otimes S^0 = C_c^\infty(Y,S^0)$ :
\begin{itemize}[label=$\bullet$]
\item pour tout $p=(a,u) \in M(X) = \textup{GL}(X) \times \textup{Sp}(W^0)$, on a :
$$(p,\tilde{u}) \cdot f : y \mapsto \omega_{\psi,S^0}(\tilde{u}) \cdot (f (a^* y)) \ ;$$
\item pour tout $p = \left[ \begin{array}{ccc} \textup{Id}_X & 0 & s \\
0 & \textup{Id}_{W^0} & 0 \\
0 & 0 & \textup{Id}_Y \end{array} \right] \in \textup{Sp}(W)$, on a :
$$(p,(1_{\textup{Sp}(W)},\textup{Id}_{S^0})) \cdot f : y \mapsto \psi( \frac{1}{2} \langle s y , y \rangle ) f(y) \ ;$$
\item pour tout $p = \left[ \begin{array}{ccc} \textup{Id}_X & v & 0 \\
0 & \textup{Id}_{W^0} & - v^* \\
0 & 0 & \textup{Id}_Y \end{array} \right] \in \textup{Sp}(W)$, on a :
$$(p,(1_{\textup{Sp}(W)},\textup{Id}_{S^0})) \cdot f : y \mapsto  \rho_\psi^0((v^* y,0)) \cdot (f(y)).$$
\end{itemize}

\paragraph{Modèle latticiel.} Soit $F$ local non archimédien de caractéristique résiduelle différente de $2$. Soit $A$ un réseau autodual de $W$. L'hypothèse sur $F$ permet d'étendre le caractère $\psi$ trivialement à $A$ d'après la Remarque \ref{simplification_I_A1_A2_rem} \textit{i.e.} de choisir $\omega=0$. Pour $g \in \textup{Sp}(W)$, $A/gA\cap A$ est fini. On le munit de la mesure de comptage $\mu$. Un calcul explicite donne :
$$I_{gA,A,\mu} \circ I_g f : (w,0) \mapsto \sum_{a \in A/gA \cap A} \psi(\frac{1}{2} \langle a,w \rangle) f((g^{-1}(a+w),0)).$$
Ainsi, en notant $K$ le stabilisateur de $A$ dans $\textup{Sp}(W)$, on a un morphisme de groupes  :
$$k \in K \mapsto (k,I_k) \in  \widetilde{\textup{Sp}}_{\psi,S_A}(W)$$
qui définit une représentation lisse $k \in K \mapsto \omega_{\psi,S_A}((k,I_k))=I_k \in \textup{GL}_R(S_A)$.

\begin{rem} Le modèle latticiel existe en caractéristique résiduelle $2$. Seulement, le fait que le caractère $\psi_A$ ne soit plus trivial sur $A$ amène à des formules qui ne sont guère exploitables. On les donne à titre de remarque. Soit $A$ un réseau auto-dual de $W$ et $\psi_A$ un caractère qui étend $\psi$ à $A \times F$. Pour tout $g \in \textup{Sp}(W)$, on munit systématiquement $A/ gA \cap A$ de la mesure de comptage. Pour tout $k$ dans $K$, on a :
$$I_{kA,A,\mu,\omega_k} \circ I_k f : (w,0) \mapsto  f((k^{-1} \omega_k,0)(k^{-1}w,0)) = \psi (\frac{1}{2} \langle \omega_k, w \rangle ) f((k^{-1}( \omega_k +w),0))$$
où $\omega_k$ est pris comme dans la Proposition \ref{entrelac_general}, et peut être choisi nul si $k$ appartient au sous-groupe ouvert compact $K \cap \textup{Stab}(\psi_A) = \{ k \in K \ | \ \psi_A^k=\psi_A\}$. Pour $g \in \textup{Sp}(W)$, en choisissant $\omega_g$ comme dans la Proposition \ref{entrelac_general}, la fonction $I_{gA,A,\mu,\omega_g} \circ I_g f$ est :
$$(w,0) \mapsto \sum_{a \in A/ g A \cap A} \psi_A((a,0))^{-1} \psi(\frac{1}{2}\langle \omega_g , a \rangle) \times \psi(\frac{1}{2} \langle \omega_g + a , w \rangle)f((g^{-1}(\omega_g+a+w),0)).$$ \end{rem}

\subsubsection{Un autre modèle}

Soit $(\rho_\psi,S)$ un modèle de la représentation métaplectique de $H$. Soit $g \in \textup{Sp}(W)$. Fixons une mesure de Haar $\mu_g$ sur l'espace vectoriel $W / \textup{Ker} (\textup{Id}_W -g^{-1})$. On vérifie que pour tout $s \in S$, la fonction :
$$w \in W \mapsto \psi(\frac{\langle w,g^{-1} w \rangle }{2}) \rho_\psi((\textup{Id}_W-g^{-1})w,0) s \in S$$
est constante sur les classes modulo $\textup{Ker}(\textup{Id}_W-g^{-1})$. 

\begin{lem} Soient $g \in \textup{Sp}(W)$ et $\mu_g$ une mesure de Haar de $W / \textup{Ker}(\textup{Id}_W-g^{-1})$.

\begin{itemize}[label=$\bullet$]
\item si $F$ est fini on définit $M[g] \in \textup{End}_R(S)$ par :
$$M[g] : s \mapsto \int_{W/\textup{Ker}(1-g^{-1})}\psi(\frac{\langle w,g^{-1} w \rangle }{2}) \rho_\psi  ((\textup{Id}_W-g^{-1})w,0)) s \ d\mu_g (w) ;$$
\item si $F$ est local non archimédien, pour tout réseau $L$ de $W/\textup{Ker}(\textup{Id}_W-g^{-1})$ on définit l'application :
$$M_L[g] : s \mapsto \int_L \psi(\frac{\langle w,g^{-1} w \rangle }{2}) \rho_\psi  ((\textup{Id}_W-g^{-1})w,0)) s \ d \mu_g w$$
Pour tout $s \in S$, l'élément $M_L[g] s$ ne dépend pas de $L$ au sens suivant :
il existe un réseau $L_s \subset W/\textup{Ker}(\textup{Id}_W-g^{-1})$, et un élément noté $M[g] s \in S$ tels que si $L$ est un réseau de $W/\textup{Ker}(\textup{Id}_W-g^{-1})$, si $L_s \subset L$, on a l'égalité $M_L[g] s=M[g] s$. L'application ainsi définie $M[g] : s \mapsto M[g] s$ appartient à $\textup{End}_R(S)$.
\end{itemize}
Alors $M[g] \in \textup{Hom}_H(\rho_\psi,\rho_\psi^g)$ \textit{i.e.} $(g,M[g]) \in \widetilde{\textup{Sp}}_{\psi,S}^R(W)$.
\end{lem}

\begin{proof}  On commence par une remarque d'ordre plus générale pour comprendre d'où viennent ces formules. Soit $G$ un groupe localement profini. Soient $(\pi,V)$ et $(\pi',V)$ deux représentations lisses de $G$ sur le même $R$-espace vectoriel $V$. Soit $H$ un sous-groupe compact. On le suppose de pro-ordre inversible dans $R$ -- ou contenant un sous-groupe ouvert dont le pro-ordre est inversible -- et on choisit une mesure $\mu_H$ de $H$ dans $R$. Alors l'endomorphisme :
$$f : v \in V \mapsto \int_H \pi'(k)^{-1} \pi(k) v \ d \mu_K(k) \in V$$
est dans $\textup{Hom}_H(\pi,\pi')$.

On prend dorénavant $G=W/\textup{Ker}(\textup{Id}_W-g^{-1})$, $H=L$, $V=S$, $\pi=\rho_\psi$ et $\pi'=\rho_\psi^g$. Comme cela a été dit plus haut, la fonction : 
$$w \in W \mapsto \rho_\psi^g((w,0))^{-1} \rho_\psi((w,0)) s = \psi(\frac{\langle w , g^{-1} w \rangle}{2})\rho_\psi(((\textup{Id}_W-g^{-1}) w,0))s \in S$$
est bi-invariante par $\textup{Ker}(\textup{Id}_W-g^{-1})$, ce qui permet de la considérer comme un élément de $C^\infty(W/\textup{Ker}(\textup{Id}_W-g^{-1}),S)$. Alors en l'intégrant sur le compact $L$, qu'on peut prendre égal à $W$ dans le cas fini, on obtient que $M_L[g] \in \textup{Hom}_L(\rho_\psi,\rho_\psi^g)$. Pour montrer l'existence du réseau $L_s$, la preuve est identique à celle de \cite[Chap.2, Lem. II.2]{mvw}. Ensuite, $M[g]$ est un endomorphisme de $S$ et $M_L[g] s$ ne dépend pas de $L$, donc $M[g] \in \textup{Hom}_H(\rho_\psi,\rho_\psi^g)$ puisque $H$ est réunion de ses sous-groupes compacts et $M_L \in \textup{Hom}_L(\rho_\psi,\rho_\psi^g)$. Comme $M[g]$ est non nul, il est inversible car les représentations $\rho_\psi$ et $\rho_\psi^g$ sont irréductibles. On a bien $(g,M[g])$ qui appartient au groupe métaplectique. \end{proof}

\begin{lem} \label{paires-duales-commutant-lem} Soient $g_1, g_2 \in \textup{Sp}(W)$. Supposons que $g_1 g_2 = g_2 g_1$. Alors : $$M[g_1] M[g_2] = M[g_2] M[g_1].$$ \end{lem}

\begin{proof} On renvoie à la preuve de \cite[Chap. 2, Lem. II.5]{mvw}, dont on pointe maintenant les points délicats. Un changement de variables $w = g_1^{-1} w'$ est effectué. Le jacobien de celui-ci, \textit{i.e.} le module dans le cadre des mesures de Haar, vaut : 
$$|\det(g_1^{-1}|_{W/\textup{Ker}(1-g_2)})| = |\det(g_1|_W)|^{-1} |\det(g_1|_{\textup{Ker}(1-g_2)})|.$$
Ce sont des puissances de $p$, donc des quantités bien définies dans $R$. Le premier terme de droite vaut $1$ car $g_1 \in \textup{Sp}(W)$. Un argument de \cite[IV.2.]{ss} donne une description explicite du commutant de $g_2$, et le second terme de droite vaut alors $1$ lui aussi. \end{proof}

\subsection{Sous-groupes scindés}

Soient $X$ un lagrangien de $W$. Soit $S_X$ le modèle de la représentation métaplectique associé. Les formules du modèle de Schrödinger donnent facilement :

\begin{prop} \label{scinde_si_ss_gp_d_un_parabolique_prop} Pour tout sous-groupe $H$ de $P(X)$, l'isomorphisme :
$$p_{S_X}^{-1} ( P(X) ) \simeq P(X)\times R^\times$$
se restreint en un isomorphisme d'extensions centrales :
$$p_{S_X}^{-1}(H) \simeq H \times R^\times.$$ \end{prop}

De plus, en exploitant les formules du modèle latticiel :

\begin{prop} \label{scindage_groupe_compact_prop} On suppose que $F$ est local non archimédien. Soit $A$ un réseau auto-dual de $W$. Par convention, la restriction de $\psi_A$ à $A$ est triviale si la caractéristique résiduelle n'est pas $2$. Le plongement naturel du sous-groupe ouvert compact $K = \textup{Stab}(A) \cap \textup{Stab}(\psi_A)$ de $\textup{Sp}(W)$ :
$$k \in K \mapsto (k,I_k) \in \widetilde{\textup{Sp}}_{\psi,S_A}^R(W)$$
est à valeurs dans $\widehat{\textup{Sp}}_{\psi,S_A}^R(W)$. \end{prop}

\begin{proof} Comme dans \cite[Chap. 2, Lem. II.10]{mvw}, on commence par la remarque suivante : si le corps résiduel de $F$ n'est pas de caractéristique $2$ et a au moins $4$ éléments, alors le groupe $K$ est parfait puisqu'il est égal à $\textup{Stab}(A)$. Pour une preuve générale, on va montrer que pour toute polarisation complète $W=X+Y$ telle que $A= A \cap X + A\cap Y$, la restriction du plongement de l'énoncé à $N(X) \cap K$ coïncide avec la restriction du plongement canonique de $N(X)$ normalisé par $P(X)$ de la proposition précédente. En effet, si $W= X+Y$ est une polarisation complète, un calcul simple donne pour tout $f \in S_X \simeq C_c^\infty(Y)$ et tout $k = \left[ \begin{array}{cc} \textup{Id}_X & s \\
0 & \textup{Id}_Y \end{array} \right] \in K \cap N(X)$ :
$$\Phi_{S_A,S_X} \circ I_k f = k \cdot f$$
où $\Phi_{S_A,S_X} : \textup{GL}_R(S_A) \to \textup{GL}_R(S_X)$ est l'isomorphisme canonique défini dans la Proposition \ref{groupe_metaplectique_prop}. En d'autres termes : 
$$\Phi_{S_A,S_X} \circ I_k f : y \mapsto \psi(\frac{1}{2}\langle sy,y\rangle ) f(y)$$
donc $(k , \Phi_{S_A,S_X} \circ I_k) \in \widehat{\textup{Sp}}_{\psi,S_X}^R(W)$ d'après la proposition précédente. Par suite, on a $(k,I_k) \in \widehat{\textup{Sp}}_{\psi,S_A}^R(W)$ pour tout $k \in K \cap N(X)$. On peut procéder de même pour $N(Y)$. Or $K \cap N(X)$ et $K \cap N(Y)$ engendrent $K$, donc $(k,I_k) \in \widehat{\textup{Sp}}_{\psi,S_A}^R(W)$ pour tout $k \in K$. \end{proof}

\subsection{Propriétés de la représentation de Weil} \label{prop_de_la_rep_de_weil_subsection}

Soient $\textup{Mp}(W)$ le groupe métaplectique sur $W$ associé à $\psi$ (Définition \ref{groupe_metaplectique_definition}) et $S$ un modèle de la représentation métaplectique associée à $\psi$. D'après le Théorème \ref{groupe_metaplectique_theoreme}, il existe un isomorphisme d'extensions centrales $\varphi_S$, unique sauf dans le cas exceptionnel, entre $\textup{Mp}(W)$ et $\widetilde{\textup{Sp}}_{\psi,S}^R(W)$. Pour tout autre modèle $S'$, on fixe des isomorphismes entre $\textup{Mp}(W)$ et $\widetilde{\textup{Sp}}_{\psi,S}^R(W)$ en posant $\varphi_{S'} = \Phi_{S,S'} \circ \varphi_S$ en reprenant les notations de la Proposition \ref{groupe_metaplectique_prop}.

\begin{defi}  Les représentations $(\omega_{\psi,S} \circ \varphi_S,S)$ et $(\omega_{\psi,S'} \circ \varphi_{S'},S')$ de $\textup{Rep}_R(\textup{Mp}(W))$ sont isomorphes. On appelle \textit{représentation de Weil modulaire sur $W$ associée à $\psi$} la classe d'isomorphisme définie par les représentations précédentes. Par extension, on désigne encore tout élément de cette classe d'isomorphisme comme \og représentation de Weil \fg{}. En général, on emploie la notation $\omega^R_{\psi,W}$. \end{defi}

\begin{rem} Il persiste donc une ambigüité dans le cas exceptionnel où le choix de $\varphi_S$ est déterminant. \end{rem}

Soit donc $\omega_{\psi,W}^R$ la représentation de Weil modulaire sur $W$ associée à $\psi$.

\begin{prop} La représentation $\omega_{\psi,W}^R$ est lisse et admissible. \end{prop}

\begin{proof} La lissité provient du Corollaire \ref{rep_de_weil_lisse_cor1}. En ce qui concerne l'admissibilité, il suffit de le prouver en prenant comme modèle explicite de la représentation de Weil son modèle latticiel quand $F$ est local non archimédien. On choisit alors un réseau auto-dual $A$ dans $W$. Par définition, $S_A = \textup{ind}_{A \times F}^{W \times F} (\psi_A)$ où $\psi_A$ étend $\psi$ à $A \times F$. Soit $K'$ un sous-groupe compact ouvert inclus dans le stabilisateur de $A$. Pour tout $f \in (S_A)^{K'}$, soit $L$ un réseau de $W$ tel que :
\begin{itemize}[label=$\bullet$]
\item $f$ soit bi-invariante sous $L$ \textit{i.e.} $f((l+w,0))=f((w,0))$ pour tout $l \in L$ et $w \in W$ ;
\item pour tout $k \in K'$ et tout $l \in L$, on a $\psi_A((k^{-1}l,0))=1$. 
\end{itemize}
On peut supposer que $L \subset A$ quitte à choisir le réseau $L\cap A$. On a alors pour tout $l \in L$, pour tout $k \in K'$ et tout $w \in W$ : 
\begin{eqnarray*} f((w,0))=f((w+l,0))=f(k^{-1}(l+w,0)) &=& f((k^{-1} l,\frac{1}{2}\langle k^{-1} w, k^{-1} l \rangle ) (k^{-1}w,0)) \\
&=& \psi_A((k^{-1}l,0)) \psi(\frac{1}{2}\langle w,l \rangle ) f(k^{-1}(w,0)) \\
 &=& \psi(\frac{1}{2}\langle w,l \rangle ) f((w,0)) . \end{eqnarray*}
On en déduit que $\textup{supp}(f)$ est inclus dans $2 L^\perp$ où $L^\perp$ est défini avant le Lemme \ref{stonve_von_neumann_premier_lem}. Donc $(S_A)^{K'}$ est donc au plus de dimension $|(A \times F) \backslash (2 L^\perp \times F) / K'|$. \end{proof}

Soit $\widehat{\textup{Mp}}(W)$ le groupe dérivé de $\textup{Mp}(W)$. Soit $Z$ le centre de $\textup{Mp}(W)$. On note $Z^2 = \{ z \in Z \ | \ z^2 = 1 \}$. Il y a des morphismes de groupes canoniques :
$$\textup{Mp}(W) / \widehat{\textup{Mp}}(W) \twoheadrightarrow  Z / Z^2 = R^\times / \{ \pm 1 \}.$$
Soit $\chi \in \textup{Rep}_R(Z)$ la restriction de $\omega_{\psi,W}^R$ à $Z$. On considère $\chi^2$ comme un caractère de $\textup{Mp}(W)$ grâce à la surjection précédente puisque $\textup{ker}(\chi^2)$ contient $Z^2$.

On déduit facilement du point a) de la Proposition \ref{representation_metaplectique_prop} : 

\begin{prop} \label{contragrediente_rep_de_weil_prop} La contragrédiente de $\omega_{\psi,W}^R$ est isomorphe à $\omega_{\psi^{-1},W}^R \otimes \chi^2$ dans $\textup{Rep}_R(\textup{Mp}(W))$. En particulier, la restriction au sous-groupe dérivé induit un isomorphisme de représentations :
$$\omega_{\psi,W}^R \simeq \omega_{\psi^{-1},W}^R \text{ dans } \textup{Rep}_R(\widehat{\textup{Mp}}(W)).$$ \end{prop}

De même, en examinant le point e) de la Proposition \ref{representation_metaplectique_prop} :

\begin{prop} \label{groupe_meataplectique_somme_produit_prop} Si $W=W_1\oplus W_2$ est une somme orthogonale, il existe un unique (resp. canonique, dans le cas exceptionnel) morphisme de groupes : 
$$i_{W_1,W_2} : \widetilde{\textup{Sp}}_{\psi,S_1}^R(W_1) \times \widetilde{\textup{Sp}}_{\psi,S_2}^R(W_2)  \to \widetilde{\textup{Sp}}_{\psi,S}^R(W)$$
qui relève le plongement $\textup{Sp}(W_1) \times \textup{Sp}(W_2) \to \textup{Sp}(W)$ et commute aux projections. Son noyau est $\{ ((1_{\textup{Sp}(W_1)},\lambda \textup{Id}_{S_2}),(1_{\textup{Sp}(W_2)},\lambda^{-1} \textup{Id}_{S_2})) \ | \ \lambda \in R^\times \}$ qui est isomorphe à $R^\times$ plongé \og anti-diagonalement \fg{}. On obtient donc que la représentation :
$$\omega_{\psi,W}^R \circ i_{W_1,W_2} \in \textup{Rep}_R(\widetilde{\textup{Sp}}_{\psi,S_1}^R(W_1) \times \widetilde{\textup{Sp}}_{\psi,S_2}^R(W_2))$$
est dans la classe d'isomorphisme de $\omega_{\psi,W_1}^R \otimes \omega_{\psi,W_2}^R$. \end{prop}

\subsection{Expression du cocycle métaplectique} \label{expression_du_cocycle_perrin_like_section}

Les notations qui suivent sont reprises de \cite{rao} et \cite{kudla}. On choisit une polarisation complète $W= X + Y$ de $W$ et l'on fixe une base $\{e_1, \dots e_m\}$ de $X$. Cela détermine alors une unique base $\{f_1, \dots, f_m\}$ de $Y$ dite \og duale \fg{} \textit{i.e.} telle que pour tout $i$ et $j$ dans $[\![ 1,m ] \!]$, on ait $\langle e_i,f_j \rangle = \delta_{i,j}$. Pour tout $S \subset [\![ 1,m ] \!]$, on note $X_S$ le sous-espace engendré par la famille $(e_i)_{i \in S}$ qui admet $X_{{}^c S}$ comme supplémentaire dans $X$, où ${}^c S$ est le complémentaire de $S$ dans $[\![ 1,m ]\!]$. On adopte des notations similaires pour $Y$. Alors $W_S = X_S + Y_S$ est un sous-espace symplectique de $W$ de supplémentaire orthogonal $W_{{}^c S}$. On note $w_S \in \textup{Sp}(W)$ l'endomorphisme défini par :
$$w_S(e_i) = \left\{ \begin{array}{lc} f_i & \textup{si } i \in S\\ e_i & \textup {si } i \notin S \end{array} \right.  \textup{ et } w_S(f_i) = \left\{ \begin{array}{cc} - e_i & \textup{si } i \in S\\ f_i & \textup {si } i \notin S \end{array} \right. .$$
De plus, pour tout $j \in [\![ 0,m ]\!]$, on pose $w_j = w_{[\![ 1,j ]\!]}$ et $\Omega_j = P(X) w_j P(X)$. Avec ces notations, la décomposition de Bruhat s'écrit :
$$\textup{Sp}(W) = \coprod_j \Omega_j.$$

Soit $g \in \Omega_j$. Soient $p_1$ et $p_2$ dans $P(X)$ tels que $g = p_1 w_j p_2$. On note $\phi_1$ l'isomorphisme entre $gX \cap X \backslash X$ et $w_j X \cap X \backslash X$ induit par :
$$x \in X \mapsto \overline{p_1^{-1} x} \in w_j X \cap X \backslash X.$$
Soit $Q_j$ la forme quadratique non dégénérée sur $w_j X \cap X \backslash X$ définie par :
$$Q_j(x) = \frac{1}{2} \langle w_S x , x \rangle.$$
Pour toute mesure de Haar $\mu$ sur $w_j X \cap X \backslash X $, on note $\mu_{w_j} = \Omega_\mu(\psi \circ Q_j)^{-1} \mu$ la mesure de Haar définie dans la Proposition \ref{normalisation_transformee_de_Fourier_prop}, dont on rappelle qu'elle s'interprète comme la mesure naturelle sur $R$ qui normalise la transformée de Fourier vis-à-vis de $\psi$ et de l'isomorphisme symétrique $\rho : x \mapsto \langle w_S x , \cdot \rangle$.

\begin{lem} \label{mesure_mu_g_def_lem} La mesure de Haar sur $g X \cap X \backslash X$ par :
$$\mu_g = \Omega_{1,\textup{det}_X(p_1 p_2)} \times \phi_1 \cdot \mu_{w_j}$$
ne dépend pas des choix de $p_1$ et $p_2$. \end{lem}

\begin{proof} Soient $g = p_1 w_j p_2 = p_1' w_j p_2'$ deux décompositions de $g$. Alors $\phi_1^{-1} \circ \phi_1'$ est un automorphisme de $w_j X \cap X \backslash X$. Pour prouver que la définition est consistante, on est ramené à prouver que :
$$\Omega_{1,\textup{det}_X(p_1 p_2)} \times |\textup{det}(\phi_1^{-1} \circ \phi_1')|_F = \Omega_{1,\textup{det}_X(p_1' p_2')}.$$
Or, d'après \cite[Lem. 3.4]{rao}, on a :
$$\det(\phi_1^{-1} \circ \phi_1')^2 = \textup{det}_X ( p_1^{-1} p_1' p_2' p_2^{-1})$$
puisque $p_1^{-1} p_1' w_j p_2' p_2^{-1} = w_j$. Or, d'après la Proposition \ref{facteur_de_weil_non_norm_prop}, on a :
$$\Omega_{1,\textup{det}(\phi_1^{-1} \circ \phi_1')^2 \textup{det}_X(p_1 p_1)} = |\textup{det} (\phi_1^{-1} \circ \phi_1')|_F \times \Omega_{1,\textup{det}_X(p_1 p_2)}.$$
Donc la définition de $\mu_g$ est bien consistante. \end{proof}

Soit $g \in \Omega_j$. Soient $p_1$ et $p_2$ dans $P(X)$ tels que $g = p_1 w_j p_2$. On définit $x(g)$ comme l'image de $\det_X(p_1 p_2)$ dans $F^\times / F^{\times 2}$. Alors $x(g)$ est bien défini car il ne dépend pas du choix de $p_1$ et $p_2$ d'après \cite[Lem. 3.4]{rao}. De plus, on a $x(w_S)=1$ pour tout $S \subset [\![ 1,m ]\!]$. Le point g) de la Proposition \ref{facteur_de_weil_non_norm_prop} donne alors :

\begin{cor} \label{mesure_rao_mult_element_dun_parab_cor} Pour tout $g \in \textup{Sp}(W)$ et tout $p \in P(X)$, on a :
$$\mu_{g p } = (x(p),x(g))_F \times \Omega_{1,\textup{det}_X(p)} \times \mu_g.$$
Et en notant $\phi_p : x \in gX \cap X \backslash X \mapsto \overline{p x} \in pgX \cap X \backslash X$ :
$$\mu_{pg} = (x(p),x(g))_F \times \Omega_{1,\textup{det}_X(p)} \times \phi_p \cdot \mu_g.$$ \end{cor}

On rappelle que $p_{S_X} : \widetilde{\textup{Sp}}_{\psi,S_X}^R(W) \to \textup{Sp}(W)$ est la projection associée au modèle de Schr\"odinger $S_X$. On réemploie les notations de la Section \ref{decalque_des_rep_S_A_section} pour les opérateurs de changement de modèles $I_{A_1,A_2,\mu,\omega}$. On définit une section de $p_{S_X}$ à l'aide des mesures précédentes :
$$\sigma : g \in \textup{Sp}(W) \mapsto \sigma(g)=I_{gX,X,\mu_g,0} \circ I_g \in \widetilde{\textup{Sp}}_{\psi,S_X}^R(W)$$
dont on note $\hat{c}$ le $2$-cocyle associé \textit{i.e.} :
$$\hat{c} : (g,g') \in \textup{Sp}(W) \times \textup{Sp}(W) \mapsto \sigma(g) \sigma(g') \sigma(gg')^{-1} \in R^\times.$$

\begin{lem} \label{cocycle_metaplectique_reduit_lem} \
\begin{enumerate}[label=\textup{\alph*)}]
\item Pour tout $g \in \textup{Sp}(W)$ et tout $p \in P(X)$ : 
$$\hat{c}(g,p) = \hat{c}(p,g) = \hat{c}(p^{-1},g)= (x(p),x(g))_F.$$
\item  Pour tout $S$ et $S'$ dans $[\![1,m]\!]$, en notant $l = | S \cap S'|$, on a :
$$\hat{c}(w_S,w_{S'}) = (-1,-1)_F^{\frac{l(l-1)}{2}}.$$
\item Soient $S \subset [\![1, m]\!]$ et $W = W_S + W_{{}^c S}$, de sorte que le sous-groupe de $\textup{Sp}(W)$ :
$$G_S = \{ g \in \textup{Sp}(W) \ | \ g(W_S) \subset W_S \textup{ et } g|_{W_S} = \textup{Id}_{W_S} \}$$ 
soit canoniquement isomorphe à $\textup{Sp}(W_{{}^c S})$ via la restriction à $W_{{}^c S}$. On adopte des notations similaires pour ${}^c S$. Alors pour tout $g \in G_S$ et tout $g' \in G_{{}^c S}$, on a l'égalité :
$$\hat{c}(g,g') = \hat{c}(g',g) = (x(g) , x(g'))_F.$$ \end{enumerate} \end{lem}

\begin{proof} a) Soient $p \in P(X)$ et $g \in \textup{Sp}(W)$. Alors, comme $\sigma(p) = \Omega_{1,\textup{det}_X(p)}\times I_p$, on déduit du Corollaire \ref{mesure_rao_mult_element_dun_parab_cor} précédent que :
$$\sigma(pg) = (x(p) , x(g))_F \times \sigma(g) \sigma(p) \textup{ et } \sigma(pg) = (x(p),x(g))_F \times \sigma(p) \sigma(g).$$
Cela montre que $\hat{c}(p,g)=\hat{c}(g,p) = (x(p),x(g))_F$. Enfin, $x(p) = x(p^{-1})$ par définition.

\noindent b) On note $\gamma_S$ l'isomorphisme induit par $w_S$ de $X_S$ vers $Y_S$. Alors pour tout $f \in C_c^\infty(Y)$ :
$$\sigma(w_S) f : y \mapsto \int_{w_S X \cap X \backslash X} \psi(\langle a,y \rangle) f((-\gamma_S a,0)) d \mu_{w_S}(a).$$
On pose $\rho_S : x \in X_S \mapsto \langle \gamma_S x , \cdot \rangle \in X_S^*$ qui un isomorphisme symétrique. En reprenant les notations de la Proposition \ref{normalisation_transformee_de_Fourier_prop}, soit :
$$M_{\gamma_S} : f \in C_c^\infty(Y_S) \mapsto \mathcal{F}_{\mu_{\rho_S}} (f \circ (-\gamma_S) ) \circ (-\gamma_S)^{-1} \in C_c^\infty(Y_S)$$
où pour $f' \in C_c^\infty(X_S)$ et $x \in X_S$ :
$$\mathcal{F}_{\mu_{\rho_S}} f' (x) =\int_{X_S} \psi( \langle \gamma_S x, a \rangle) f'(a) d \mu_{w_S}(a).$$
Dans la décomposition $C_c^\infty(Y) = C_c^\infty(Y_S) \otimes C_c^\infty(Y_{{}^c S})$, on a $\sigma(w_S) = M_{\gamma_S} \otimes \textup{Id}$. Il en va de même avec $w_{S'}$ et $\sigma(w_{S'})$.

Maintenant, la décomposition $C_c^\infty(Y) = C_c^\infty(Y_{S \cap S'}) \otimes C_c^\infty(Y_{S \Delta S'}) \otimes C_c^\infty(Y_{{}^c (S \cup S')})$ donne que la composée $\sigma(w_S) \circ \sigma(w_{S'})$ s'écrit comme $(M_{\gamma_S} \circ M_{\gamma_{S'}}) \otimes M_{\gamma_{S \Delta S'}} \otimes \textup{Id}$ où $\gamma_{S \Delta S'}$ est associé à $w_{S \Delta S'}$. Or, la restriction de $M_{\gamma_S}$ et $M_{\gamma_{S'}}$ à $C_c^\infty(Y_{S \cap S'})$ sont toutes deux égales à $M_{\gamma_{S \cap S'}}$. Ensuite :
\begin{eqnarray*} M_{\gamma_{S \cap S'}}^2 f &=& \mathcal{F}_{\mu_{\rho_{ S\cap S'}}} (M_{\gamma_{S \cap S'}} f \circ (-\gamma_{S \cap S'})) \circ (-\gamma_{S\cap S'})^{-1} \\
&=& \mathcal{F}_{\mu_{\rho_{ S\cap S'}}} \big( \mathcal{F}_{\mu_{\rho_{ S\cap S'}}}  (f \circ (-\gamma_{S \cap S'})) \big) \circ (-\gamma_{S\cap S'})^{-1} \\
&=& \mathcal{F}_{\mu_{\rho_{S \cap S'}}}^2 (f \circ (-\gamma_{S \cap S'}) ) \circ (-\gamma_{S\cap S'})^{-1}. \end{eqnarray*}
Finalement, la Proposition \ref{normalisation_transformee_de_Fourier_prop} donne que pour tout $f \in C_c^\infty(Y_{S \cap S'})$ et tout $y \in Y_{S \cap S'}$ :
$$M_{\gamma_{S \cap S'}}^2 f (y )= (-1, \textup{det}(Q_{\frac{1}{2} \rho_{S \cap S'}}))_F (\Omega_{-1,1})^l \times  f(-y).$$

En remarquant que $w_S w_{S'}=  w_{S \Delta S'} a_{S \cap S'}$ où $a_{S \cap S'} =(-\textup{Id}_{W_{S \cap S'}}) \oplus \textup{Id}_{W_{{}^c (S \cap S')}}$ est dans $P(X)$, la mesure $\mu_{w_S w_{S'}} = \Omega_{1, (-1)^l} \times \mu_{w_{S \cap S'}}$. Dans la décomposition précédente, $\sigma(w_S w_{S'})$ s'écrit alors $A_{S \cap S'} \otimes M_{\gamma_{S \Delta S'}} \otimes \textup{Id}$ où $A_{S \cap S'} f(y) = \Omega_{1, (-1)^l} \times f(-y)$. Par conséquent, comme :
$$(-1,\textup{det}(Q_{\frac{1}{2} \rho_{S \cap S'}}))_F = (-1,2^{-l})_F = (-1 , 2^l)_F = \big( (-1,2)_F \big)^l  = \big( (-1,-1)_F \big)^l,$$
il vient :
$$ \hat{c}(w_S,w_{S'}) = \big( (-1,-1)_F \Omega_{-1,1} \big)^l \times \Omega_{(-1)^l,1}. $$
En appliquant récursivement le point g) de la Proposition \ref{facteur_de_weil_non_norm_prop} :
$$\Omega_{(-1)^l,1} = (\Omega_{-1,1})^l \times (-1,-1)_F^{\frac{l(l-1)}{2}}.$$
Enfin, on conclut grâce à l'égalité $(\Omega_{-1,1})^2 = (-1,-1)_F$ que $\hat{c}(w_S,w_{S'}) = (-1,-1)_F^{\frac{l(l-1)}{2}}$.

\noindent c) D'une part, pour tout $f \in S_X$, un calcul explicite de $\sigma(g) \circ \sigma(g') f ((0,0))$ donne :
$$\int_{gX \cap X \backslash X} \int_{g'X \cap X \backslash X} f(((g')^{-1} a' , 0 )  ((g')^{-1} g^{-1} a , 0 ) ) d \mu_{g'}(a') d \mu_g(a).$$
D'autre part, le morphisme de réduction :
$$x \in X \mapsto (p_g(x),p_{g'}(x)) \in (g' X \cap X \backslash X ) \times ( gX \cap X \backslash X )$$
a pour noyau $g g' X \cap X$. Or, $g$ et $g'$ commutent entre eux et chacun de ces morphismes induit le morphisme identité par passage au quotient sur $g'X \cap X \backslash X$ et $g X \cap X \backslash X$ respectivement. De plus, on déduit du point g) de la Proposition \ref{facteur_de_weil_non_norm_prop} qu'on a :
$$\mu_g \otimes \mu_{g'} = (x(g),x(g'))_F \mu_{g g'}.$$
Par conséquent, un changement de variable dans l'intégrale donne l'égalité :
$$\sigma(g) \circ \sigma(g') f ((0,0)) = (x(g),x(g'))_F \times \sigma(g g') f ((0,0))$$
pour tout $f \in S_X$. Cela entraîne donc l'égalité $\sigma(g) \circ \sigma(g') = (x(g),x(g'))_F \times \sigma(g g')$.\end{proof}

\begin{defi} On note $\textup{Sp}(W) \times_{\hat{c}} R^\times$ l'ensemble $\textup{Sp}(W) \times R^\times$ muni de la loi de groupe :
$$(g, \lambda) \cdot (g', \lambda') = (g g' , \hat{c}(g,g') \lambda \lambda').$$ \end{defi}

\begin{theo} \label{cocycle_metaplectique_thm} On distingue les cas comme dans le Théorème \ref{groupe_metaplectique_theoreme}.
\begin{enumerate}[label=\textup{\alph*)}]
\item Si  $F$ est fini, ou si la caractéristique $\ell$ de $R$ est $2$, alors $\hat{c}$ est le $2$-cocyle trivial. Par conséquent, $\sigma$ est une section de $p_{S_X}$ qui est un morphisme de groupes. De plus, un tel morphisme de groupes est unique sauf dans le cas exceptionnel $F= \mathbb{F}_3$ et $\textup{dim}_F W= 2$, et définit toujours un isomorphisme d'extensions centrales :
$$(g, \lambda) \in \textup{Sp}(W) \times R^\times \mapsto (g,\lambda \sigma(g)) \in \widetilde{\textup{Sp}}_{\psi,S_X}^R(W).$$
\item  Si $F$ est local non archimédien et $\ell \neq 2$, alors $\hat{c}$ est à valeurs dans $\{ \pm 1 \}$. Par conséquent, $\sigma$ est l'unique section réalisant l'isomorphisme d'extensions centrales :
$$(g, \lambda) \in \textup{Sp}(W) \times_{\hat{c}} R^\times \mapsto (g, \lambda \sigma(g)) \in \widetilde{\textup{Sp}}_{\psi,S_X}^R(W).$$
et dont la restriction à $\textup{Sp}(W) \times_{\hat{c}} \{ \pm 1 \}$ induit un isomorphisme d'extensions centrales avec $\widehat{\textup{Sp}}_{\psi,S_X}^R(W)$. \end{enumerate} \end{theo}

\begin{proof} Dans le premier cas, on va montrer que $\hat{c}$ est le cocycle trivial ; dans le second, que $\hat{c}$ est à valeurs dans $\{ \pm 1 \}$. Les formules obtenues sont valables dans les deux cas, bien que triviales dans le premier. Soient $p_1, p_2, p \in P(X)$ et $g_1, g_2 \in \textup{Sp}(W)$. Alors on déduit du Lemme \ref{cocycle_metaplectique_reduit_lem} en décomposant $\sigma(g_1 p^{-1}) \circ \sigma(p g_2)$ :
$$\hat{c} (g_1 p^{-1} , p g_2) =  \hat{c}(p,p) \hat{c}(g_1 , g_2) \hat{c}(g_1 ,g_2)$$
ainsi que $\sigma(p_1 g_1) \circ \sigma(g_2 p_2)$ :
$$\hat{c}(p_1g_1,g_2p_2)=\hat{c}(p_1 ,g_1) \hat{c}(g_2,p_2) \hat{c}(p_1,g_1g_2) \hat{c}(p_1 g_1 g_2 ,p_2) \hat{c}(g_1,g_2).$$
Par conséquent :
$$\hat{c}(p_1 g_1 p^{-1} , p g_2 p_2) \hat{c} (g_1,g_2)^{-1} = \left\{ \begin{array}{ll} 1 & \textup{ dans le premier cas ;} \\
\pm 1 & \textup{ dans le second}. \end{array} \right.$$

Il reste à prouver que pour $g_1$ et $g_2$ bien choisis, le cocycle $\hat{c}(g_1,g_2)$ est trivial ou à valeurs dans $\{ \pm 1 \}$. Pour ce faire, on utilise une décomposition associé à l'invariant de Leray \cite[Th. 2.16]{rao}. Pour tout $g_1$ et $g_2$ dans $\textup{Sp}(W)$, il existe $S_1$, $S_2$ et $S$ dans $[\![1,m]\!]$, et un isomorphisme $\rho : Y_S \to X_S$ antisymétrique, \textit{i.e.} $\rho^* = -\rho$, de sorte que $(S_1 \cup S_2) \cap S = \emptyset$ et  $g_1 = p_1 w_{S \cup S_1} u_{\rho} p^{-1}$ et $g_2 = p w_{S \cup S_2} p_2$. Il suffit donc de calculer pour le morphisme $\rho$ en question $\hat{c}(w_{S \cup S_1} u_\rho, w_{S\cup S_2})$. Or, en calculant $\sigma(w_{S \cup S_1} u_\rho) \circ \sigma(w_{S\cup S_2})$ de différentes manières, on obtient :
$$\hat{c}(w_{S \cup S_1} u_\rho, w_{S\cup S_2}) = \hat{c}(w_S u_\rho, w_S) \hat{c}(w_{S_1} w_{S_2} , w_S u_\rho w_S)\hat{c}(w_{S_1} ,w_{S_2}).$$
Et d'après le Lemme \ref{cocycle_metaplectique_reduit_lem}, on sait que $\hat{c}(w_{S_1} w_{S_2} , w_S u_\rho w_S)=((-1)^{|S_1 \cap S_2|} , x(w_S u_\rho w_S))_F$ et $\hat{c}(w_{S_1},w_{S_2})$ est une puissance de $(-1,-1)_F$.

Il reste à étudier le facteur $\hat{c}(w_S u_\rho, w_S)$. Le lemme suivant permet de conclure que $\hat{c}$ prend bien les valeurs recherchées :

\begin{lem} \label{cocycle_formule_w_s_u_rho_w_s_lem} On a :
$$\hat{c}(w_S u_\rho,w_S)=(-2, \textup{det}(Q_{\gamma_S \rho \gamma_S}))_F \times h_F(Q_{\gamma_S \rho \gamma_S})$$
où $Q_{\gamma_S \rho \gamma_S}(x)= \langle x , \gamma_S \rho \gamma_S x \rangle$ est une forme quadratique non dégénérée sur $X_S$. \end{lem}

\begin{proof} Il s'agit donc de calculer, pour tout $S \subset [\![1,m]\!]$ et tout $\rho : Y_S \to X_S$ de sorte que $u_\rho \in \textup{Sp}(W)$,   la composée $\sigma(w_S u_\rho) \circ \sigma(w_S)$ en termes de $\sigma(w_S u_\rho w_S)$. Soit $f \in S_X$. Comme $\sigma(w_S u_\rho) = \sigma(w_S) \circ  \sigma(u_\rho)$ d'après le Lemme \ref{cocycle_metaplectique_reduit_lem} :
$$\sigma(w_S u_\rho) \circ \sigma(w_S) f ((0,0)) = \int_{X_S} (\sigma(u_\rho) \circ \sigma(w_S) f)((w_S^{-1} a,0)) d \mu_{w_S}(a).$$
Or, $(\sigma(u_\rho) \circ \sigma(w_S) f)((w_S^{-1} a,0)) = \psi(\frac{1}{2} \langle w_S^{-1} a , (-\rho) w_S^{-1} a \rangle) \times ( \sigma(w_S) f ) ((w_S^{-1} a, 0 ))$ d'après les formules du modèle de Schr\"odinger de la Section \ref{decalque_des_rep_S_A_section}. Mais :
\begin{eqnarray*} \sigma(w_S) f ((w_S^{-1} a , 0 )) &=& \int_{X_S} f((w_S^{-1} a',0)(w_S^{-2} a , 0 ) d \mu_{w_S}(a') \\
&=& \int_{X_S} \psi( \langle a', w_S^{-1} a \rangle) f((w_S^{-1}a',0))) d \mu_{w_S}(a') \\
&=& |\phi_{\rho,S}|^{-1} \int_{X_S} \psi( \langle w_S^{-1}a, \rho w_S^{-1} a'' \rangle) f((- w_S^{-1} \rho w_S^{-1} a'',0) d \mu_{w_S}(a'')
\end{eqnarray*}
où le changement de variable opéré est $a'=\phi_{\rho,S} (a'') = -\rho w_S^{-1} a''$ pour l'automorphisme $\phi_{\rho,S}$ de $X_S$ induit par $-\rho w_S^{-1}$ sur $X_S$.

Ensuite, comme $u_\rho^{-1} = u_{- \rho}$ :
$$- w_S^{-1} \rho w_S^{-1} a'' = w_S^{-1} u_\rho^{-1} w_S^{-1} a'' - w_S^{-2} a''.$$
Par conséquent :
$$f((-w_S^{-1} \rho w_S^{-1} a'',0))= \psi(\frac{1}{2} \langle w_S^{-1} a'' , (-\rho) w_S^{-1} a'' \rangle) f((w_S^{-1} u_\rho^{-1} w_S^{-1} a'' , 0)).$$
On obtient donc pour $\sigma(w_S u_\rho) \circ \sigma(w_S) f ((0,0))$ au facteur $|\phi_{\rho,S}|^{-1}$ près :
$$\int_{X_S} \int_{X_S} \psi \bigg(\frac{1}{2} \langle w_S^{-1} (a-a'') , (-\rho) w_S^{-1} (a - a'') \rangle \bigg) \ (I_{w_S u_\rho w_S}f)((a'',0)) d\mu_{w_S}(a'') d \mu_{w_S}(a).$$
À l'aide du facteur de Weil non normalisé défini dans la Section \ref{facteur_de_weil_non_norm_sect}, cette dernière expression se simplifie en :
$$\Omega_{\mu_{w_S}}(\psi \circ Q_S) \times \int_{X_S} (I_{w_S u_\rho w_S} f)((a'',0)) d\mu_{w_S}(a'')$$
où $Q_S(x) = - \frac{1}{2} \langle x , w_S \rho w_S^{-1} x \rangle$.

De plus $w_S u_\rho w_S \in G_{{}^c S}$ et admet comme décomposition dans $W_S = X_S + Y_S$ :
$$w_S u_\rho w_S =  \left[ \begin{array}{cc} * & * \\
\gamma_S \rho \gamma_S & * \end{array} \right] \textup{ où } w_S = \left[ \begin{array}{cc} * & * \\
\gamma_S & * \end{array} \right] \textup{ et } \gamma_S^*=-\gamma_S.$$
Comme $\gamma_S \rho \gamma_S$ est de rang $|S|$, il existe $p_1$ et $p_2$ dans $P(X_S)$ telles que $w_S u_\rho w_S = p_1 w_S p_2$. De plus, il existe une décomposition de la forme :
$$w_S u_\rho w_S =  \left[ \begin{array}{cc} \textup{Id}_{X_S} & * \\
0 & \textup{Id}_{Y_S} \end{array} \right] w_S \left[ \begin{array}{cc} a & * \\
0 & (a^*)^{-1} \end{array} \right] .$$
En particulier, une telle décomposition impose que $\gamma_S a= \gamma_S \rho \gamma_S$. En d'autres termes $a = \rho \gamma_S \in \textup{GL}_F(X_S)$. Avec ces notations, on a $\phi_{\rho,S} = \rho \gamma_S$ et $Q_S(x) =  \frac{1}{2} \langle x , \gamma_S \rho \gamma_S x \rangle$. 

L'expression de la mesure $\mu_{w_S u_\rho w_S}$ est alors :
$$\mu_{w_S u_\rho w_S} = \Omega_{1, \textup{det}(\rho \gamma_S)} \times  \mu_{w_S}.$$
On obtient donc la formule :
$$\hat{c}(w_S u_\rho , w_S) = |\phi_{\rho,S}|^{-1} \times \Omega_{\mu_{w_S}}(\psi \circ Q_S) \times  \Omega_{\textup{det}(\rho \gamma_S),1},$$
que l'on se propose de simplifier dans la suite de la preuve.

D'une part, dans la base standard $\mathcal{B}$ de $X_S$, on a d'après le Corollaire \ref{facteur_de_weil_inv_de_hasse_cor} :
$$\Omega_{\mu_{w_S}} ( \psi \circ Q_{\frac{1}{2} \gamma_S \rho \gamma_S}) = \Omega_{\textup{det}_\mathcal{B}(Q_{\frac{1}{2} \gamma_S \rho \gamma_S}),1} \times h_F(Q_{\frac{1}{2} \gamma_S \rho \gamma_S}) \times \Omega_{\mu_{w_S}} (\psi \circ Q_{\gamma_S}).$$
Or $\Omega_{\mu_{w_S}}(\psi \circ Q_{\gamma_S}) = (\Omega_{1,\frac{1}{2}})^{|S|} \times \Omega_{\mu_{w_S}}(\psi \circ Q_{\frac{1}{2}\gamma_S}) = (\Omega_{1,\frac{1}{2}})^{|S|}$ où la dernière égalité se déduit par définition de la mesure $\mu_{w_S}$. De plus : 
$$h_F(Q_{\frac{1}{2} \gamma_S \rho \gamma_S}) = (2,\textup{det}(Q_{\gamma_S \rho \gamma_S})^{|S|-1})_F \times h_F(Q_{\gamma_S \rho \gamma_S}).$$
et :
$$\Omega_{\textup{det}_\mathcal{B}(Q_{\frac{1}{2} \gamma_S \rho \gamma_S}),1} = (2^{-|S|},\textup{det}(Q_{\gamma_S \rho \gamma_S}))_F \times \Omega_{2^{-|S|},1} \times \Omega_{\textup{det}_\mathcal{B}(Q_{\gamma_S \rho \gamma_S}), 1}$$
En remarquant que $\Omega_{2^{-|S|},1} = (\Omega_{\frac{1}{2},1})^{|S|}$, il vient :
$$\Omega_{\mu_{w_S}} ( \psi \circ Q_{\frac{1}{2} \gamma_S \rho \gamma_S}) = (2, \textup{det}(Q_{\gamma_S \rho \gamma_S}))_F \times \Omega_{\textup{det}_\mathcal{B}(Q_{\gamma_S \rho \gamma_S}),1}  h_F(Q_{\gamma_S \rho \gamma_S}).$$

D'autre part, la représentation matricielle de la forme quadratique $Q_{\gamma_S \rho \gamma_S}$ dans la base $\mathcal{B}$ est $Q_{\gamma_S \rho \gamma_S}(x) = {}^t X M X$ où $M = \textup{Mat}_\mathcal{B}(\rho \gamma_S)$ et $ X \in F^{|S|}$ est le vecteur coordonné associé à $x \in X$ dans la base $\mathcal{B}$. Cela signifie que $\textup{det}_\mathcal{B}(Q_{\gamma_S \rho \gamma_S}) = \textup{det}(\rho \gamma_S)$. Comme pour tout $a \in F^\times$, on a $(\Omega_{a,1})^2 = |a| \times (a,a)_F = |a| \times (-1,a)_F$ :
$$\hat{c}(w_S u_\rho,w_S) = (-2,\textup{det}(Q_{\gamma_S \rho \gamma_S}) )_F \times h_F (Q_{\gamma_S \rho \gamma_S}).$$ \end{proof}

Pour terminer la preuve du point a), la propriété d'unicité dans le cas non exceptionnel vient de l'énoncé du Théorème \ref{groupe_metaplectique_theoreme}. Celle du point b) est un fait classique pour les isomorphismes d'extensions centrales car $\textup{Sp}(W)$ est un groupe parfait. \end{proof}

\begin{rem} La formule que l'on obtient dans la preuve pour $\hat{c}(w_S u_\rho, w_S)$ diverge de celle de \cite{rao}, mais seulement en apparence, au sens où ces deux cocycles sont égaux en cohomologie. En effet, le choix de mesure qu'il opère correspond avec les notations de la Proposition \ref{mesure_mu_g_def_lem} à $\mu_{g,\textup{Rao}} = \Omega_{\frac{1}{2}, \frac{1}{2} \textup{det}_X(p_1 p_2)} \times \phi_1 \cdot \mu_{w_j} = (2, x(g))_F \times \mu_g$ donc :
$$\sigma_{\textup{Rao}}(g) = (2,x(g))_F \times \sigma(g).$$
Plus généralement, un choix de mesure du type $\mu_{g,\alpha} = \Omega_{\alpha, \alpha \textup{det}_X(p_1 p_2)} \times \phi_1 \cdot \mu_{w_j}$ pour $\alpha \in F^\times$ donne un $2$-cocycle $\hat{c}_\alpha$ dans la même classe de cohomologie que $\hat{c}$. Il faut modifier en conséquence le point b) du Lemme \ref{cocycle_metaplectique_reduit_lem} et le Lemme \ref{cocycle_formule_w_s_u_rho_w_s_lem} pour $\hat{c}_\alpha$, mais tout le développement précédent reste valable. \end{rem}

\begin{cor} Excepté dans le cas exceptionnel $F=\mathbb{F}_3$ et $\textup{dim}_F(W)=2$, la section $\sigma$ est à valeurs dans $ \widehat{\textup{Sp}}_{\psi,S_X}^R(W)$. \end{cor}

\begin{cor}  \label{cocycle_metaplectique_reduit_cor} Soient $g_1$ et $g_2$ dans $\textup{Sp}(W)$. Par définition de l'invariant de Leray, il existe $p_1, p_2, p \in P(X)$, $S\subset[\![1,m]\!]$, un isomorphisme antisymétrique $\rho : Y_S \to X_S$  et $S_1,S_2 \subset {}^c S$ tels que $g_1 = p_1 w_{S \cup S_1} u_\rho p^{-1}$ et $g_2 = p w_{S \cup S_2} p_2$. Avec ces notations, et en posant $l= |S_1 \cap S_2|$, on a :
\begin{multline*} \hat{c}(g_1,g_2) = (x(g_1),x(g_2))_F \times  (x(g_1)x(g_2),-x(g_1 g_2))_F \times (-1,-1)_F^{\frac{l(l -1)}{2}} \\
\times ((-1)^l,x(w_S u_\rho w_S))_F \times \hat{c}(w_S u_\rho , w_S).
\end{multline*} \end{cor}

\begin{proof} En reprenant la preuve précédente :
\begin{multline*} \hat{c}(p_1 w_{S \cup S_1} u_\rho p^{-1} , p w_{S \cup S_2} p_2) = \hat{c}(p_1 , w_{S \cup S_1} u_\rho p^{-1}) \times \hat{c}(p_1 ,w_{S \cup S_1} u_\rho w_{S \cup S_2} p_2) \\ \times \hat{c}(w_{S \cup S_1} u_\rho p^{-1} , p w_{S \cup S_2} p_2)
\end{multline*}
Or $\hat{c}(p_1 , w_{S \cup S_1} u_\rho p^{-1}) =  (x(p_1),x(p))_F$ et : 
$$\hat{c}(p_1 ,w_{S \cup S_1} u_\rho w_{S \cup S_2} p_2) = (x(p_1), x(w_{S \cup S_1} u_\rho w_{S \cup S_2}) x(p_2))_F.$$
Comme $x(w_{S \cup S_1} u_\rho w_{S \cup S_2}) = x(w_{S_1} w_{S_2} ) x(w_S u_\rho w_S) =(-1)^l x(w_S u_\rho w_S)$ :
$$\hat{c}(g_1,g_2) = (x(p_1) , x(p))_F \times (x(p_1),(-1)^l x(w_S u_\rho w_S))_F \times \hat{c}(w_{S \cup S_1} u_\rho p^{-1} , p w_{S \cup S_2} p_2).$$
Les détails sont omis car similaires, mais on a de même :
\begin{multline*}
\hat{c}(w_{S \cup S_1} u_\rho p^{-1} , p w_{S \cup S_2} p_2) = (x(p),x(p_2))_F \times ((-1)^l x(w_S u_\rho w_S),x(p_2))_F \\
\times \hat{c}(w_{S \cup S_1} u_\rho p^{-1} , p w_{S \cup S_2}).
\end{multline*}
De plus $x(g_1 g_2 ) =x(p_1) (-1)^l x(w_S u_\rho w_S) x(p_2)$. Et puisque :
$$\hat{c}(w_{S \cup S_1} u_\rho p^{-1} , p w_{S \cup S_2}) = (x(p),x(p))_F \times \hat{c}(w_{S \cup S_1} u_\rho , w_{S \cup S_2}),$$
il vient :
$$\hat{c}(g_1,g_2) = (x(g_1) ,x(g_2))_F \times (x(g_1) x(g_2), - x(g_1) x(g_2))_F \times \hat{c}(w_{S \cup S_1} u_\rho, w_{S \cup S_2}).$$
Et d'après la discussion préliminaire au Lemme \ref{cocycle_formule_w_s_u_rho_w_s_lem} :
$$ \hat{c}(w_{S \cup S_1} u_\rho, w_{S \cup S_2}) = ((-1)^l,x(w_S u_\rho w_S))_F \times (-1,-1)_F^\frac{l(l-1)}{2} \times \hat{c}(w_S u_\rho , w_S).$$
\end{proof}

\section{Relevés de paires duales et scindages} \label{releves_de_paires_duales_scindages_section}

Soit $(H_1,H_2)$ une paire duale réductive dans le groupe symplectique $\textup{Sp}(W)$. On rappelle (cf. \cite[Chap. I, 1.17]{mvw}) que $H_1$ et $H_2$ sont deux sous-groupes de $\textup{Sp}(W)$ qui sont les centralisateurs l'un de l'autre. Ils sont de plus réductifs.

Soit $p_S : \widetilde{\textup{Sp}}_{\psi,S}^R(W) \to \textup{Sp}(W)$ la projection associée à un modèle $S$ de la représentation de Weil modulaire. On note alors :
$$\widetilde{H}_{1,S} = p_S^{-1} (H_1) \textup{ et } \widetilde{H}_{2,S} = p_S^{-1}(H_2).$$

\begin{prop} Le centralisateur de $\widetilde{H}_{1,S}$ dans $\widetilde{\textup{Sp}}_{\psi,S}^R(W)$ est $\widetilde{H}_{2,S}$, et vice-versa. \end{prop}

\begin{proof} Soit $g$ dans le centralisateur de $\widetilde{H}_{1,S}$. Alors $p_S(g)$ est dans le centralisateur de $H_1$ c'est-à-dire $p_S(g) \in H_2$. De plus, le Lemme \ref{paires-duales-commutant-lem} assure que le centralisateur de $H_2$ contient $\widetilde{H}_{1,S}$. D'où l'égalité. Enfin, les mêmes arguments s'appliquent pour $\widetilde{H}_{1,S}$ et le centralisateur de $\widetilde{H}_{2,S}$. \end{proof}

Par conséquent, la paire $(\widetilde{H}_{1,S},\widetilde{H}_{2,S})$ est une paire duale dans $\widetilde{\textup{Sp}}_{\psi,S}^R(W)$. Dorénavant, toutes les paires duales considérées dans le groupe métaplectique seront de ce type. On dit que cette paire est le \textit{relevé au groupe métaplectique} de la paire $(H_1,H_2)$. En général, la théorie se ramène à l'étude des paires duales réductives qui sont irréductibles. Elles sont en un certain sens minimales puisqu'elles ne proviennent pas de produits de paires duales plus petites -- pour une définition rigoureuse \cite[\textit{Ibid.}]{mvw}. Leur classification se divise entre les paires de type I et celles de type II.

Soit $H$ un sous-groupe de $\textup{Sp}(W)$. Le groupe $\widetilde{H}_S = p_S^{-1}(H)$ est dit \textit{scindé} s'il existe un morphisme de groupes $H \to \widetilde{H}_S$ qui est une section de $p_S$. Dans ce cas, il existe un isomorphisme d'extensions centrales entre l'extension centrale triviale $H \times R^\times$ et $\widetilde{H}_S$. De plus, l'ensemble de ces isomorphismes est en bijection avec l'ensemble des caractères de $H$ à valeurs dans $R^\times$.

D'après le Théorème \ref{cocycle_metaplectique_thm}, tout relevé d'un sous-groupe de $\textup{Sp}(W)$ est scindé si $F$ est fini ou si $R$ est de caractéristique $2$. On suppose donc dorénavant, et ce jusqu'à la fin de cette partie, que $F$ est local non archimédien et $R$ n'est pas de caractéristique $2$.

\subsection{Scindages des relevés de paires duales via un parabolique}

Voici une première réponse concernant les scindages de relevés de paires duales :

\begin{prop} \label{scindage_releve_paires_cas_scinde_prop} Soit $(H_1,H_2)$ une paire duale irréductible dans $\textup{Sp}(W)$. 
\begin{itemize}[label=$\bullet$]
\item Si cette paire est de type II, alors $\widetilde{H}_{1,S}$ et $\widetilde{H}_{2,S}$ sont scindés.
\item Si elle est de type I, il existe $W = W_1 \otimes W_2$ telle que $(H_1,H_2)=(U(W_1),U(W_2))$. Si $W_2$ est scindé, alors $\widetilde{H}_{1,S}$ est scindé.  \end{itemize} \end{prop}

\begin{proof} Grâce à la Proposition \ref{scinde_si_ss_gp_d_un_parabolique_prop}, il suffit de prouver que les groupes en question sont inclus dans un groupe parabolique de $\textup{Sp}(W)$. Tout d'abord, si la paire est de type II, il existe un lagrangien $X$ de $W$ tel que $H_1$ et $H_2$ soient contenus dans $P(X)$ d'après \cite[Chap. I, 1.19]{mvw}. Si celle-ci est de type I, il existe \cite[\textit{Ibid}]{mvw} une décomposition $W = W_1 \otimes W_2$ telle que la paire duale $(H_1,H_2)$ soit de la forme énoncée. Si $W_2$ est scindé, alors $U(W_1)$ est inclus dans le parabolique $P(W_1 \otimes X_2)$ pour tout lagrangien $X_2$ de $W_2$. \end{proof}

Il est possible de décrire de manière plus précise où vit ce scindage. La question suivante est intéressante à divers titres : existe-t-il un scindage de $\widetilde{H}_{1,S}$ à valeurs dans $\widehat{\textup{Sp}}_{\psi,S}(W)$ ? La réponse sera précisée à l'aide de :

\begin{lem} \label{scindage_parab_groupe_met_reduit_lem} On rappelle que le contexte considéré ici est $F$ local non archimédien et $R$ de caractéristique différente de $2$. Soit $X$ un lagrangien dans $W$.
\begin{enumerate}[label=\textup{\alph*)}]
\item Si $F$ est une extension finie de $\mathbb{Q}_2$ et $-1$ n'est pas un carré dans $F$, alors il n'existe pas scindages de $P(X)$ à valeurs dans $\widehat{\textup{Sp}}_{\psi,S}(W)$.
\item Sinon, en posant $s=\textup{val}_F(2)$ où $\textup{val}_F$ est la valuation normalisée dans $F$, il existe $4 \times q^s$ scindages de $P(X)$ à valeurs dans $\widehat{\textup{Sp}}_{\psi,S}(W)$.
\end{enumerate}  \end{lem}

\begin{rem} Quand la caractéristique résiduelle est impaire, on a $s=0$. \end{rem}

\begin{proof} On traite d'abord le cas où $W$ est de dimension $2$. On choisit comme modèle $S=S_X$. D'après le Lemme \ref{cocycle_metaplectique_reduit_lem}, il est équivalent de dénombrer l'ensemble des sections de :
$$(p,\lambda) \in P(X) \times_{\hat{c}} R^\times \mapsto p \in P(X)$$
qui sont des morphismes de groupes à valeurs dans $P(X) \times_{\hat{c}} \{ \pm 1\}$. On rapelle qu'ici $P(X) \simeq M(X) \ltimes N(X) = F^\times \ltimes F$ et :
$$\hat{c} : (x,y)=((a_x,n_x),(a_y,n_y)) \in P(X) \mapsto (a_x,a_y)_F \in \{\pm 1 \}.$$
Comme l'application $n \in N(X) \mapsto (n,1) \in P(X) \times_{\hat{c}} R^\times$ est un morphisme de groupes, et que son image est distinguée, l'existence d'une section qui est un morphisme de groupes de $P(X)$ à valeurs dans $P(X) \times_{\hat{c}} \{ \pm 1 \}$ est équivalente à l'existence d'une section de $\varepsilon : (a,\varepsilon) \in F^\times \times_{\hat{c}} \{ \pm 1 \} \mapsto a \in F^\times$ qui est un morphisme de groupes. De manière équivalente, la question se ramène à : existe-t-il une application $\varepsilon : a \in F^\times \mapsto \varepsilon_a \in \{ \pm 1 \}$ de sorte que $a \in F^\times \mapsto (a,\varepsilon_a) \in F^\times \times_{\hat{c}} \{ \pm 1 \}$ soit un morphisme de groupes ?

Étant donné que $\hat{c}(a,a') = (a,a')_F$, l'application $\varepsilon$ doit vérifier pour tout $a, a' \in F^\times$ :
$$\varepsilon_a \varepsilon_{a'} = (a,a')_F \varepsilon_{a a'}.$$
En particulier, on déduit de cette relation et de $(a,a)_F = (-1,a)_F$ pour tout $a \in F$ que : 
\begin{itemize}[label=$\bullet$]
\item si $-1$ est un carré dans $F$, cette relation impose que $\varepsilon$ est constante sur les classes modulo les carrés, donc se factorise en une application $F^\times/F^{\times 2} \to \{ \pm 1 \}$, encore notée $\varepsilon$ ;
\item si $-1$ n'est pas un carré dans $F$, alors $\varepsilon$ est constante sur les classes modulo les puissances $4$-ème, donc se factorise en une application $F^\times/F^{\times 4} \to \{ \pm 1 \}$, encore notée $\varepsilon$.
\end{itemize}

\noindent a) Pour régler ce cas, on utilise un résultat plus fort. On rappelle que l'espace symplectique considéré $W$ est de dimension $2$ :

\begin{lem} Si $F$ est une extension finie de $\mathbb{Q}_2$ et $-1$ n'est pas un carré dans $F$, il n'existe pas de scindages de $\{ \pm \textup{Id}_W \}$ à valeurs $\widehat{\textup{Sp}}_{\psi,S}(W)$. Il en existe deux à valeurs dans $\widetilde{\textup{Sp}}_{\psi,S}(W)$. \end{lem}

\begin{proof} En effet, soit $\lambda$ un scindage de $\{ \pm \textup{Id}_W \}$ à valeurs dans $\{ \pm \textup{Id}_W \} \times_{\hat{c}} R^\times$ qui est un morphisme de groupes. Cela impose en particulier les relations $\lambda(\textup{Id}_W) = (\textup{Id}_W , 1)$ et $\lambda(-\textup{Id}_W) = (-\textup{Id}_W, \alpha)$ où $\alpha \in R^\times$. Or comme $\lambda$ est un morphisme de groupes :
$$\lambda(-\textup{Id}_W) \lambda(-\textup{Id}_W) =  \lambda(\textup{Id}_W).$$
Par conséquent $\alpha^2  = (-1,-1)_F = -1$. Une racine carrée de $-1$ appartient bien à $R$ car il existe un caractère additif $\psi : F \to R^\times$ non trivial. Donc il n'existe pas de scindage de $\{ \pm \textup{Id}_W \}$ à valeurs dans $\widehat{\textup{Sp}}_{\psi,S}(W)$, bien qu'il en existe toujours à valeurs dans $\widetilde{\textup{Sp}}_{\psi,S}(W)$. Ces derniers sont donnés par $\alpha \in \{ \pm i \}$, ce qui en fait bien deux puisque $R$ n'est pas de caractéristique $2$.

Quand $W$ est de dimension quelconque, on peut toujours se ramener à la situation du lemme précédent en considérant un sous-espace symplectique de dimension $2$. En effet, il existe deux sous-espaces symplectiques $W'$ et $W''$ de $W$ tels que $W= W' \oplus W''$, la dimension de $W'$ est $2$ et $X= X \cap W' \oplus X \cap W''$. Alors le groupe $\{ (\pm \textup{Id}_{W'},\textup{Id}_{W''}) \}$ est un sous-groupe évident de $\textup{Sp}(W)$, contenu dans $P(X)$. D'après le lemme précédent, il n'existe pas de scindages de $\{ (\pm \textup{Id}_{W'},\textup{Id}_{W''}) \}$ à valeurs dans $\widehat{\textup{Sp}}_{\psi,S_X}(W)$. Donc \textit{a fortiori}, il n'en existe pas pour $P(X)$. \end{proof}

\noindent b) On suppose pour commencer que $-1$ est un carré dans $F$. L'ensemble des caractères de $F^\times / F^{\times 2}$ à valeurs dans $\{ \pm 1 \}$ sont aux nombres de $|F^\times / F^{\times 2}| = 4 \times q^s$, où $s=0$ si la caractéristique résiduelle est impaire ; et $s$ est le degré de ramification de $F$ sur $\mathbb{Q}_2$ sinon. Par conséquent, s'il existe un scindage à valeurs dans $\widetilde{\textup{Sp}}_{\psi,S}(W)$, il en existe en tout $4 \times q^s$. Il s'agit maintenant de construire un tel scindage. Le groupe $F^\times / F^{\times 2}$ est muni d'une structure de $\mathbb{F}_2$-espace vectoriel évidente et pour laquelle l'application :
$$(a,a') \in F^\times/F^{\times 2} \mapsto (a,a')_F \in \{ \pm 1 \}$$
est une forme bilinéaire symétrique non dégénéré en identifiant $\{ \pm 1 \}$ à $\mathbb{F}_2$. De plus, tout vecteur est isotrope car $-1$ est un carré dans $F$. On peut donc trouver une base hyperbolique de $F^\times/F^{\times 2}$ notée $\{ e_i , f_i \}_{1 \leq i \leq q^s}$ qui vérifie :
$$(e_i,f_j)_F = - (-1)^{\delta_{i,j}}.$$
On pose alors $\varepsilon_{e_i} = \varepsilon_{f_i} = 1$. Et pour tout $x \in F^\times$, dont on note $\{x_{e_i},x_{f_i} \}_{1 \leq i \leq q^s}$ les coordonnées dans la base précédente, on pose enfin :
$$\varepsilon_{x} = \prod_{1 \leq i \leq q^s} (x_{e_i},x_{f_i})_F.$$
Il est aisé de vérifier que pour tout $x, y \in F^\times/F^{\times 2}$, on a la relation $\varepsilon_x \varepsilon_y = (x,y)_F \varepsilon_{x y}$. Ce qui donne un morphisme de groupes $p =(a,n) \in P(X) \mapsto (p,\varepsilon_a) \in P(X) \times_{\hat{c}} \{ \pm 1 \}$, qui est le scindage désiré quand $W$ est de dimension $2$.

Enfin, il reste à étudier le cas où $-1$ n'est pas un carré dans un corps local non archimédien $F$ de caractéristique résiduelle impaire. Le groupe $F^\times/F^{\times 4} = < -1 , \varpi_F>$, où $\varpi_F$ est une uniformisante dans $F^\times$, est isomorphe à $\mathbb{Z}/2\mathbb{Z} \times \mathbb{Z}/4 \mathbb{Z}$. Par conséquent, il existe seulement $4$ caractères à valeurs dans $\{ \pm 1 \}$. Comme dans le paragraphe précédent, il suffit de construire une application $\varepsilon : a \in F^\times/F^{\times 4} \mapsto \varepsilon_a \in \{ \pm 1 \}$ qui donne un scindage $a \in F^\times \mapsto (a,\varepsilon) \in F^\times \times_{\hat{c}} \{ \pm 1 \}$. On rappelle les relations $(a,a)_F = (-1,a)_F$ et :
$$(-1,\varpi_F)_F=(-1,-\varpi_F)_F=-1 \textup{ et } (-1,-1)_F =1.$$
On tire de $\varepsilon_a \varepsilon_{a'} = (a,a')_F \varepsilon_{a a'}$ des conditions sur $\varepsilon$ à savoir :
\begin{itemize}[label=$\bullet$]
\item $\varepsilon_1 =1$ \textup{ et } $\varepsilon_{\varpi_F^2}=-1$ ;
\item $\varepsilon_{\varpi_F^3} = - \varepsilon_{\varpi_F}$ \textup{ et } $\varepsilon_{-\varpi_F^3} = - \varepsilon_{-\varpi_F}$ ;
\item $\varepsilon_{\varpi_F} \varepsilon_{-\varpi_F} = \varepsilon_{-1}$ \textup{ et } $\varepsilon_{-\varpi_F^2} = -\varepsilon_{-1}$.
\end{itemize}
Selon les valeurs de $\varepsilon_{-1}$, on trouve que : quand $\varepsilon_{-1}=- 1$ on doit avoir $\varepsilon_{\varpi_F} = -\varepsilon_{- \varpi_F}$ ; quand $\varepsilon_{-1}=1$ on doit avoir $\varepsilon_{\varpi_F} = \varepsilon_{- \varpi_F}$. Ce qui restreint à $2$ choix dans le premier cas, et $2$ dans le deuxième, soit $4$ au total. Ainsi, n'importe lequel de ces candidats devrait être un scindage. Pour ce faire, on vérifie à la main que l'application définie par $\varepsilon_{-1} = -1$ et $\varepsilon_{\varpi_F} = \varepsilon_{-\varpi_F} = 1$, et soumise aux trois relations précédentes, est consistante vis-à-vis de la relation $\varepsilon_a \varepsilon_{a'} = (a,a')_F \varepsilon_{a a'}$. Ce problème comporte un nombre finie de vérifications étant donné que $F^\times/F^{\times 4}$ est fini. Comme le symbole de Hilbert est symétrique, il suffit de considérer les $8$ cas diagonaux $a=a'$ et les $(64-8)/2 = 28$ cas non diagonaux $a \neq a'$. On peut retirer les cas où $a=1$ car immédiats, ce qui laisse $28-7=21$ vérifications aussi fastidieuses qu'élémentairement vraies.

Quand $W$ est de dimension quelconque, le groupe dérivé de $P(X)$ est isomorphe à $\textup{SL}(X) \ltimes N(X)$. Ce dernier possède un (unique) scindage à valeurs dans $\textup{Sp}(W) \times_{\hat{c}} R^\times$ donné par $p \in \textup{SL}(X) \ltimes N(X) \mapsto (p,1) \in \textup{Sp}(W) \times_{\hat{c}} R^\times$, et dont l'image encore notée $\textup{SL}(X) \ltimes N(X)$ est distinguée dans $P(X) \times_{\hat{c}} R^\times$. On peut donc quotienter par le sous-groupe $\textup{SL}(X) \ltimes N(X)$ pour se ramener au problème équivalent de trouver un scindage de $P(X)/(\textup{SL}(X) \ltimes N(X)) = F^\times$ à valeurs dans :
$$(P(X) \times_{\hat{c}} \{ \pm 1 \}) / (\textup{SL}(X) \ltimes N(X)) = F^\times \times_{\hat{c}} \{ \pm 1 \}.$$
Ce qui ramène au cas de la dimension $2$ que l'on vient d'étudier. \end{proof}

On tire donc l'amélioration suivante en examinant la preuve de la Proposition \ref{scindage_releve_paires_cas_scinde_prop} à la lumière du Lemme \ref{scindage_parab_groupe_met_reduit_lem} :

\begin{cor} \label{scindage_paire_dual_parab_gp_met_reduit_cor} Soit $(H_1,H_2)$ une paire duale irréductible dans $\textup{Sp}(W)$. On note  $\widehat{H}_{1,S}$ et $\widehat{H}_{2,S}$ les images réciproques respectives dans $\widehat{\textup{Sp}}_{\psi,S}(W)$. On suppose que l'on exclut le cas où $F$ est une extension finie de $\mathbb{Q}_2$ et $-1$ n'est pas un carré dans $F$.
\begin{itemize}[label=$\bullet$]
\item Si cette paire est de type II, alors $\widehat{H}_{1,S}$ et $\widehat{H}_{2,S}$ sont scindés ;
\item si elle est de type I, il existe $W = W_1 \otimes W_2$ telle que $(H_1,H_2)=(U(W_1),U(W_2))$. Si $W_2$ est scindé, alors $\widehat{H}_{1,S}$ est scindé.
\end{itemize} \end{cor}

\subsection{Scindages de relevés de paires duales irréductibles de type I} \label{scindages_de_releves_paires_irr_type_I_section}

Soit $D$ une algèbre à division dont le centre $E$ contient le corps local non archimédien $F$. Soit $\tau : D \to D$ une involution de $D$ \textit{i.e.} d'un anti-automorphisme de carré l'application identité. On suppose que $F$ est l'ensemble des points fixes de $\tau$. Les possibilités pour $(D,E,F,\tau)$ sont :
\begin{enumerate}
\item $D=E=F$ et $\tau=\textup{Id}_D$ ;
\item $D$ est l'unique algèbre de quaternion sur $E=F$ et $\tau$ est l'involution canonique ;
\item $D=E$ est une extension quadratique de $F$ et $F$ est le générateur de $\textup{Gal}(E/F)$.
\end{enumerate}
Soient $(W_1,\langle , \rangle_1)$ un espace $\varepsilon_1$-hermitien (à droite) sur $D$ et $(W_2,\langle , \rangle_2)$ un espace $\varepsilon_2$-hermitien (à droite) sur $D$ avec $\varepsilon_2=-\varepsilon_1$. On peut munir $W_2$ d'une structure d'espace vectoriel à gauche pour tout $d \cdot w_2 = w_2 \tau(d)$ pour tout $d \in D$ et $w_2 \in W_2$.

Soit $t_{D/F} : D \to F$ une trace de $D$ sur $F$ \textit{i.e.} $t_{D/F}$ est une forme linéaire non nulle telle que pour tout $d , d' \in D$ :
$$t_{D/F}(\tau(d) d') =t_{D/F}(d \tau(d')).$$
Autrement dit $t_{D/F}$ est une trace invariante par $\tau$, ce qui inclut le cas de la trace réduite. Il existe alors un unique produit symplectique $\langle , \rangle$ sur le $F$-espace vectoriel $W = W_1 \otimes_D W_2$ de sorte que pour tout $w_1, w_1' \in W_1$ et tout $w_2, w_2' \in W_2$ :
$$\langle w_1 \otimes w_2 , w_1' \otimes w_2' \rangle = t_{D/F} \bigg( \langle w_1 , w_1' \rangle_1 \times \tau \big( \langle w_2 ,w_2' \rangle_2 \big) \bigg).$$
Dans cette situation, en considérant l'action naturelle des groupes d'isométries $U(W_1)$ et $U(W_2)$ sur $W = W_1 \otimes W_2$, la paire $(U(W_1),U(W_2))$ est une paire duale irréductible de type I dans $\textup{Sp}(W)$ qui ne provient pas d'une restriction des scalaires sur $F$ d'après \cite[Chap. I, I.20]{mvw}. Réciproquement, toute paire duale $(H_1,H_2)$ dans un groupe symplectique $\textup{Sp}(W)$ qui est irréductible de type I et ne provient pas d'une restriction des scalaires sur $F$ est obtenue à partir d'un tel procédé \cite[\textit{Ibid.}]{mvw}.

Soit donc $(H_1,H_2)=(U(W_1),U(W_2))$ une paire duale dans un groupe symplectique $\textup{Sp}(W)$ où $W= W_1 \otimes_D W_2$. Soit $S$ un modèle de la représentation métaplectique associée à un caractère additif $\psi : F \to R^\times$. On généralise \cite[Th. 3.1]{kudla}. En revenant avec les distinctions précédentes pour $(D,E,F,\tau)$, les paires sont constituées de :
\begin{enumerate}
\item un groupe symplectique et un groupe orthogonal ;
\item deux groupes unitaires quaternioniques, l'un hermitien et l'autre anti-hermitien ;
\item deux groupes unitaires classiques, l'un hermitien et l'autre anti-hermitien.
\end{enumerate}

\begin{theo} On suppose que $R$ contient une racine carrée de $q$ et on la fixe.
\begin{enumerate}[label=\textup{\alph*)}]
\item Si $W_1$ est symplectique et $W_2$ est orthogonal de dimension impaire, alors l'extension $\widetilde{H}_{1,S}$ est isomorphe au groupe métaplectique sur $W_1$. En particulier, cette extension n'est pas scindée sur $H_1$.
\item Dans tous les autres cas, l'extension $\widetilde{H}_{1,S}$ est scindée sur $H_1$.
\end{enumerate} \end{theo}

\begin{proof} Quand $R$ contient une racine carrée de $q$, on peut définir le facteur de Weil classique $\omega$ d'après le point \ref{facteur_de_weil_non_nomr_vs_classique_pt} de la Proposition \ref{facteur_de_weil_non_norm_prop}. En reprenant les notations du Corollaire \ref{cocycle_metaplectique_reduit_cor}, le $2$-cocyle $\hat{c}$ est alors cohomologiquement équivalent au $2$-cocycle :
$$c(g_1,g_2) = \omega(\psi \circ Q_{\frac{1}{2} \rho}).$$
Étant donné les propriétés du facteur de Weil classique, cette quantité ne dépend que de la classe d'isomorphisme de $\rho$ \textit{i.e.} de son déterminant, de son invariant de Hasse et de son rang. Ce $2$-cocyle est à valeurs dans $\mu_8(R) = \{ \lambda \in R^\times \ | \ \lambda^8 =1 \}$.

Par conséquent, l'extension $\widetilde{H}_{1,S}$ est isomorphe à $H_1 \times_{c_1} R^\times$  où le $2$-cocyle $c_1$ est obtenu comme le composé de $c$ avec le plongement $i_1 : h_1 \in H_1 \mapsto h_1 \otimes_D \textup{Id}_{W_2} \in \textup{Sp}(W)$. Dans le cas b), et en supposant que $W_1$ est scindé, la résolution du $2$-cocycle $c_1$ est en tout point similaire à celle explicitée dans \cite[Th. 3.1]{kudla} en choisissant des bases adaptées de $W_1$ et $W_2$. Il en va de même pour le cas a).

Il reste alors à considérer le cas b) quand $W_1$ n'est pas scindé. La technique de doublement effectuée dans \cite[Prop. 4.1]{kudla} est encore valable. On la rappelle succinctement. On note $-W_1$ l'espace $\varepsilon_1$-hermitien $(W_1,-\langle , \rangle_1)$. Alors l'espace $\varepsilon_1$-hermitien $\mathbb{W}_1 = W_1 \oplus (-W_1)$ est scindé et $\mathbb{W} = (W_1 \otimes_D W_2) \oplus ((-W_1) \otimes_D W_2) = \mathbb{W}_1 \otimes_D W_2$ est un espace symplectique sur $F$. De plus, en posant $W = W_1 \otimes_D W_2$ et $-W=(-W_1) \otimes_D W_2$, on a une identification canonique entre $\textup{Sp}(W)$ et $\textup{Sp}(-W)$ en tant que sous-groupes de $\textup{GL}_F(W)$. Le morphisme naturel :
$$(u,u') \in \textup{Sp}(W) \times \textup{Sp}(W) \to u \oplus u' \in \textup{Sp}(\mathbb{W})$$
se relève aux groupes métaplectiques grâce la Proposition \ref{groupe_meataplectique_somme_produit_prop} pour $\mathbb{W}=W \oplus (-W)$, en choisissant des modèles des représentations métaplectiques $S$ et $S^-$ associées à $W_1 \otimes_D W_2$ et $(-W_1) \otimes_D W_2$ pour le même caractère $\psi$. Ce relevé s'écrit :
$$\widetilde{\textup{Sp}}_{\psi,S}(W) \times \widetilde{\textup{Sp}}_{\psi, S^-}(-W) \to \widetilde{\textup{Sp}}_{\psi,S \otimes S^-} (\mathbb{W}).$$
Son noyau étant $\{ \big((\textup{Id}_W,\lambda \textup{Id}_S),(\textup{Id}_W,\lambda^{-1} \textup{Id}_{S^-})\big) \ | \ \lambda \in R^\times\}$, le sous-groupe $\widetilde{H}_{1,S} \times \{ 1 \}$ s'injecte dans le groupe de droite.

La paire $(H_1',H_2') = (U((W_1 \oplus (-W_1)), U(W_2))$ est une paire duale irréductible de type I dans $\textup{Sp}(\mathbb{W})$. D'après le paragraphe précédent, l'application :
$$(h_1 \otimes \textup{Id}_{W_2} , M) \in \widetilde{H}_{1,S} \mapsto ((h_1 \otimes \textup{Id}_{W_S}) \oplus \textup{Id}_W, M \otimes \textup{Id}_{S^-}) \in \widetilde{\textup{Sp}}_{\psi,S \otimes S^-} (\mathbb{W})$$
est un morphisme de groupes injectif. L'image de $\widetilde{H}_{1,S}$ est contenue dans $\widetilde{H}_{1,S \otimes S^-}'$, qui est un sous-groupe scindé car $W_1 \oplus (-W_1)$ est scindé et $W_2$ n'est pas orthogonal de dimension impaire. \end{proof}

\begin{rem} En toute généralité, il faut redoubler d'attention quand on se place sur un corps $R$ qui est minimal. Par exemple, si $F = \mathbb{F}_3((t))$ et $R=\mathbb{Q}(j)$, il n'est pas garanti que tout relevé de paire duale qui tombe dans le cas b) soit scindée. En revanche, en considérant le corps $R' = \mathbb{Q}(j,i) =\mathbb{Q}(j,\sqrt{3})$, tout relevé l'est. \end{rem}

On a le résultat suivant qui permet de délimiter les cas potentiellement problématiques issus de la remarque précédente :

\begin{cor} On suppose exclu le cas $W_1$ symplectique et $W_2$ orthogonal de dimension impaire.
\begin{enumerate}[label=\textup{\alph*)}]
\item Si $-1$ est un carré dans $F^\times$, alors $\widehat{H}_{1,S}$ est scindé ;
\item Si $F$ est de caractéristique résiduelle $2$, alors $\widetilde{H}_{1,S}$ est scindé ;
\item Si la paire est symplectique-orthogonale et $p \neq 2$, alors $\widehat{H}_{1,S}$ est scindé.
\end{enumerate} \end{cor}

\begin{proof} a) Tout d'abord, la condition $-1 \in F^{\times 2}$ entraîne, d'après \cite[4]{ct}, que l'image du facteur de Weil classique est à valeurs dans $\{ \pm 1\}$. Ainsi, les $2$-cocyles $c$ et $\hat{c}$ admettent des résolutions à valeurs dans $\{ \pm 1 \}$ en examinant les formules issues de \cite[Th. 3.1]{kudla}. Cela signifie bien que $\widehat{H}_{1,S}$ est scindé. On remarque \textit{a fortiori} que $q$ admet bien une racine carrée. En effet, la situation se divise en trois cas. Soit $q$ est une puissance impaire de $p \neq 2$,  auquel cas $p \equiv 1 [4]$ et $R$ contient une racine carrée de $p$ d'après un argument clasique sur les sommes de Gauss. Soit $q$ est une puissance paire de $p \neq 2$, ce qui signifie que $F$ contient l'unique extension non ramifiée de degré $2$ de $\mathbb{Q}_p$ ou $\mathbb{F}_p((t))$. Soit $p=2$ et $F$ contient le corps $\mathbb{Q}_2(\sqrt{-1})$.

\noindent b) Ensuite, si $F$ est de caractéristique résiduelle $2$, l'existence d'un caractère additif non trivial de $F$ entraîne que $R$ contient des racines carrées de $-1$ et $2$. En particulier, il contient une racine carrée de $q$.

\noindent c) Si $W_1$ est orthogonal, alors $W_2$ est symplectique. En particulier, l'espace $W_2$ est scindé et on peut donc utiliser le Corollaire  \ref{scindage_paire_dual_parab_gp_met_reduit_cor} en considérant $U(W_1)$ comme un sous-groupe d'un parabolique $P(W_1 \otimes_F X_2)$ où $X_2$ est un lagrangien de $W_2$.

Si $W_1$ est symplectique et $W_2$ est orthogonal de dimension paire, on sait alors qu'il existe une résolution du cocycle métaplectique à valeurs dans une extension $R'$ de $R$ contenant une racine carrée de $q$. En notant $\hat{c}_1$ le $2$-cocycle métaplectique réduit sur $\textup{Sp}(W_1 \otimes_F W_2)$ restreint à $H_1=\textup{Sp}(W_1)$, on a une inclusion évidente $H_1 \times_{\hat{c}_1} R^\times$ dans $H_1 \times_{\hat{c}_1} R'^\times$. Or, le groupe $H_1$ est parfait, donc son image par un scindage à valeurs dans $H_1 \times_{\hat{c}_1} R'^\times$ est incluse dans le groupe dérivé $H_1 \times_{\hat{c}_1} \{ \pm 1\}$, et donc incluse dans le sous-groupe $H_1 \times_{\hat{c}_1} R^\times$. Le scindage de $\widetilde{H}_{1,S}$ ainsi obtenu est donc défini sur $R$ et à valeurs dans le groupe métaplectique réduit \textit{i.e.} $\widehat{H}_{1,S}$ est scindé. De plus, ce scindage est l'unique scindage de $\widetilde{H}_{1,S}$ puisque le groupe symplectique est parfait. \end{proof}

\begin{rem} Un cas subsiste dans lequel on ne sait pas dire si le scindage est défini ou non pour un corps minimal $R$ ou ne contenant pas de racine de $q$. Il est le suivant. Quand $p$ est impair et $-1$ est n'est pas un carré dans $F$ -- ceci implique que $p \equiv 3 [4]$, on ne sait pas dire si $\widetilde{H}_{1,S}$ est scindé ou non sur un corps $R$ minimal quand $D$ est quaternionique ou quadratique. On donne un exemple sous forme de question pour illustrer ce propos. Soit $F = \mathbb{F}_3((t))$ et $R = \mathbb{Q}(j)$. Soit $W$ un espace symplectique de dimension $2$ et $W = W_1 \otimes_D W_2$ une décomposition de $W$ où $W_1$ et $W_2$ sont deux espaces unitaires de dimension $1$, l'un hermitien et l'autre anti-hermitien, avec $D=E$ une extension quadratique de $F$. En posant $H_1 = U(W_1)$, a-t-on que $\widetilde{H}_{1,S}$ est scindé ? \end{rem}

\section{En direction d'une correspondance thêta modulaire}

Soit $S$ un modèle de la représentation métaplectique sur $R$ associée à $\psi$. Pour toute paire duale $(H_1,H_2)$ dans $\textup{Sp}(W)$, la représentation de Weil modulaire $\omega_{\psi,S}$ du groupe métaplectique $\widetilde{\textup{Sp}}_{\psi,S}(W)$ se restreint au produit $\widetilde{H}_{1,S} \times \widetilde{H}_{2,S}$ des relevés de ces paires duales, que l'on note encore $\omega_{\psi,S}$. Dans la Section \ref{releves_de_paires_duales_scindages_section}, on donne des conditions sur ces relevés pour que la représentation $\omega_{\psi,S}$ se restreigne -- ou non -- en une vraie représentation de $H_1 \times H_2$, selon que ces relevés soient scindés ou non.

Pour éviter ces considérations portant sur des scindages, on étudie la représentation $\omega_{\psi,S}$ en considérant pour tout sous-groupe fermé $H$ de $\textup{Sp}(W)$ la catégorie abélienne :
$$\textup{Rep}_R^{\textup{gen}}(\widetilde{H}_S) = \{ \pi \in \textup{Rep}_R(\widetilde{H}_S) \ | \ \pi((\textup{Id}_W,\lambda \textup{Id}_S)) = \lambda \textup{Id}_S \}.$$
On en note $\textup{Irr}_R^{\textup{gen}}(\widetilde{H}_{1,S})$ les objets simples. En général, la contragrédiente de $\pi$ appartient à $\textup{Rep}_R(\widetilde{H}_S)$ mais n'appartient pas à $\textup{Rep}_R^{\textup{gen}}(\widetilde{H}_S)$. Cependant, quitte à tensoriser par un caractère ou à considérer la catégorie $\textup{Rep}_R^{\textup{gen}}(\widehat{H}_S)$ sur le relevé réduit $\widehat{H}_S$ comme dans la Proposition \ref{contragrediente_rep_de_weil_prop}, on peut se servir d'une notion analogue à la contrégradiente de $\pi$ dans $\textup{Rep}_R^{\textup{gen}}(\widetilde{H}_S)$. Aussi est-il plus adéquat de considérer la catégorie $\textup{Rep}_R^{\textup{gen}}(\widehat{H}_S)$, ce qui ne change pas la teneur des énoncés en jeu puisque les catégories $\textup{Rep}_R^{\textup{gen}}(\widehat{H}_S)$ et $\textup{Rep}_R^{\textup{gen}}(\widetilde{H}_S)$ sont équivalentes.

Le Lemme \ref{pi-coinvariant} formalise la définition de plus grand quotient isotypique pour une représentation irréductible donnée. Une simple application de ce lemme permet de définir, à l'aide d'une représentation irréductible du premier relevé, une représentation du second, connue sous le nom de $\Theta$-lift.

\begin{theo} \label{Theta_pi_1_def_theo} Soient $\pi_1 \in \textup{Irr}_R^{\textup{gen}} (\widehat{H}_{1,S})$ admissible et $D_1 = \textup{End}_{\widehat{H}_{1,S}}(\pi_1)$.

Il existe alors une représentation $\Theta(\pi_1)$ dans $\textup{Rep}_R^{\textup{gen}}(\widehat{H}_{2,S})$, munie d'une structure de $D_1$-module à droite compatible avec sa structure de représentation, de sorte que le plus grand quotient $\pi_1$-isotypique $(\omega_{\psi,S})_{\pi_1}$ de $\omega_{\psi,S}$ soit isomorphe à $\Theta(\pi_1) \otimes_{D_1} \pi_1$. De plus, en tant que $R[\widehat{H}_{1,S}]-D_1$-bimodule, le bimodule $\Theta(\pi_1)$ est unique à isomorphisme près. \end{theo}

Ce résultat est bien connu quand $R$ est le corps des nombres complexes. Quand le corps $R$ est algébriquement clos, ou que $\pi_1$ est absolument irréductible, le corps $D_1$ précédent est simplement $R$. La structure de $D_1$-module à droite correspond alors à la structure naturelle de $R$-module (à gauche) de la représentation $\Theta(\pi_1)$. En toute généralité, comme $\pi_1$ est admissible, on sait que le corps $D_1$ est une algèbre à division de dimension finie sur son centre, qui est une extension finie du corps central $R$.

On discutera dans la sous-section suivante des conjectures qui portent sur $\Theta(\pi_1)$ et qui constituent le cadre d'étude pour la \og correspondance thêta module \fg{}. Ensuite, en se basant sur la Remarque \ref{extension_des_scalaires_reduction_des_scalaires_rem} concernant les modèles explicites, les deux dernières sous-sections étudieront le comportement de $\Theta(\pi_1)$ vis-à-vis de l'extension des scalaires, ainsi que des possibilités de réduction des scalaires.

\subsection{Définition de la correspondance modulaire et résultats connus} \label{definition_de_la_corresp_modulaire_thm}

Dorénavant, le corps $F$ est local non archimédien. Pour un panorama des résultats connus sur la correspondance thêta classique sur un corps local non archimédien -- c'est-à-dire pour les représentations à coefficients complexes -- et de ses contributeurs, on indique l'introduction du présent travail. On rappelle simplement que, quand $R$ est le corps des nombres complexes, les deux énoncés ($\Theta_2$) et ($\Theta_3$) ci-dessous constituent ce que l'on appelle communément la \og correspondance thêta classique sur un corps local non archimédien \fg{}. Ils sont transposés ici pour les représentations à coefficients quelconques \textit{i.e.} pour les représentations modulaires. On suppose toutefois que $R$ est algébriquement clos.

\begin{defi}[Correspondance modulaire] Pour tout $\pi_1$ et $\pi_1'$ dans $\textup{Irr}_R^{\textup{gen}} (\widehat{H}_{1,S})$, on considère les assertions suivantes :
\begin{enumerate}
\item[($\Theta_1$)] la représentation $\Theta(\pi_1)$ est de longueur finie, donc admet un co-socle noté $\theta(\pi_1)$ ;
\item[($\Theta_2$)]  soit $\Theta(\pi_1)$ est nulle, soit $\theta(\pi_1)$ est irréductible ;
\item[($\Theta_3$)] quand $\Theta(\pi_1) \neq 0$, on a $\theta(\pi_1) \simeq \theta(\pi_1')$ si et seulement $\pi_1 \simeq \pi_1'$. 
\end{enumerate}
Quand ces trois énoncés sont valides, la \og correspondance thêta $R$-modulaire sur un corps local non archimédien \fg{} est alors définie comme la bijection induite par $\theta$ entre les ensembles : 
$$\{ \pi_1 \in \textup{Irr}_R^{\textup{gen}} (\widehat{H}_{1,S}) \ | \ \Theta(\pi_1) \neq 0\} \overset{\theta}{\simeq} \{ \pi_2 \in \textup{Irr}_R^{\textup{gen}} (\widehat{H}_{2,S}) \ | \ \Theta(\pi_2) \neq 0\}$$
où $\Theta(\pi_2)$ est obtenu en calculant les $\pi_2$-coinvariants vis-à-vis de $\widehat{H}_{2,S}$ au lieu de $\widehat{H}_{1,S}$, c'est-à-dire en inversant l'ordre des paires duales dans le Théorème \ref{Theta_pi_1_def_theo}. \end{defi}

\paragraph{État des connaissances.} La question de déterminer la validité de ces énoncés en toute généralité constitue un problème difficile en soi qui dépasse largement l'objet ce travail. Cependant, on ne s'attend pas à ce que les trois énoncés précédents soient vrais en toute généralité \textit{i.e.} sans condition plus précise sur la caractéristique $\ell$ de $R$ vis-à-vis du pro-ordre des groupes $H_1$ et $H_2$ qui constituent la paire duale étudiée. Pour justifier cette remarque, on cite les résultats prouvés par A. M\'inguez pour les paires duales de type II. On dit que la caractéristique d'un corps est banale si elle ne divise pas le pro-ordre des groupes qui sont en jeu.

\begin{theo}[\cite{minguez_thesis}] Soit $(H_1,H_2)$ une paire de type II dans un groupe symplectique $\textup{Sp}(W)$ avec $F$ local non archimédien. Alors, si $R=\overline{\mathbb{F}_\ell}$ avec $\ell$ banal vis-à-vis des pro-ordres de $H_1$ et $H_2$, les énoncés \textup{($\Theta_1$)-($\Theta_2$)-($\Theta_3$)} sont vrais pour toute représentation irréductible. \end{theo}

Par conséquent, on ne peut parler de la correspondance thêta $R$-modulaire que dans le cadre de ce théorème, qui permet de considérer la bijection définie par $\theta$ entre les sous-ensembles donnés plus haut. Il développe également \cite[Sec. 4.5.2]{minguez_thesis} un contre-exemple qui met en défaut ($\Theta_2$), toujours pour les paires duales de type II, quand $\ell$ est non banal vis-à-vis des groupes $H_1$ et $H_2$. Son travail repose cependant sur l'utilisation d'un analogue \textit{ad hoc} de la représentation de Weil pour ces paires-là en remplaçant les formules connues sur $\mathbb{C}$ pour ces modèles par des formules similaires sur $\overline{\mathbb{F}_\ell}$.

Le présent travail permet de définir un cadre complet d'étude pour les représentations à coefficients dans un corps $R$ quelconque en généralisant les résultats connus pour les représentations à coefficients complexes, à savoir : le théorème de Stone-von Neumann ; le groupe métaplectique ; la représentation de Weil.

En ce qui concerne les paires duales de type I, étudier la validité des trois énoncés ($\Theta_1$)-($\Theta_2$)-($\Theta_3$) constitue un problème complètement nouveau. Cependant, traiter de telles questions nécessiterait de nombreux développements qui dépassent le cadre ici présent, aussi repoussera-t-on cette étude aux articles qui feront suite. On s'attend à ce que ces énoncés ($\Theta_1$)-($\Theta_2$)-($\Theta_3$) soient vrais quand $R = \overline{\mathbb{F}_\ell}$ avec $\ell$ suffisamment grand vis-à-vis des pro-ordres de la paire duale considérée. Néanmoins, la condition minimale sur $\ell$ pourrait se révéler plus forte qu'une simple hypothèse de banalité vis-à-vis de $(H_1,H_2)$. Il existe également un contre-exemple à ($\Theta_2$) pour les paires de type I en caractéristique non banale qu'on peut d'ores-et-déjà consulter \cite[Sec. 7.2]{trias_thesis}. 
 
\subsection{Compatibilités à l'extension des scalaires} \label{compatibilité_extension_des_scalaires_section}

Le corps $F$ est encore local non archimédien ici. On ne suppose pas dans ce paragraphe que le corps $R$ est algébriquement clos. Si l'on fixe $\bar{R}$ une clôture algébrique de $R$, on peut donner des énoncés analogues sur $R$ à ceux sur $\bar{R}$ du paragraphe précédent ($\Theta_1$)-($\Theta_2$)-($\Theta_3$). Le but de ce paragraphe est de donner une compatibilité à l'extension des scalaires pour $\Theta(\pi_1)$ ; et de prouver ainsi que les énoncés analogues ($\Theta_1'$)-($\Theta_2'$)-($\Theta_3'$) sur $R$ définis ci-dessous sont équivalents, en un sens que l'on va préciser, aux énoncés sur $\bar{R}$, à condition que $R$ soit un corps parfait. On considère donc les trois énoncés suivants. Pour tout $\pi_1$ et $\pi_1'$ dans $\textup{Irr}_R^{\textup{gen}} (\widehat{H}_{1,S})$ :
\begin{enumerate}
\item[($\Theta_1'$)] le $R[\widehat{H}_{2,S}]-D_1$-bimodule $\Theta(\pi_1)$ est de longueur finie, on note $\theta(\pi_1)$ son co-socle ;
\item[($\Theta_2'$)]  soit $\Theta(\pi_1)$ est nulle, soit $\theta(\pi_1)$ est un $R[\widehat{H}_{2,S}]-D_1$-bimodule simple ;
\item[($\Theta_3'$)] quand $\Theta(\pi_1) \neq 0$, on a $\theta(\pi_1) \simeq \theta(\pi_1')$ si et seulement $\pi_1 \simeq \pi_1'$.
\end{enumerate}

\begin{rem} Il est sous-entendu que l'isomorphisme $\theta(\pi_1) \simeq \theta(\pi_1')$ dans ($\Theta_3'$) signifie en particulier que $D_1 = \textup{End}_{\widehat{H}_{1,S}}(\pi_1) \simeq \textup{End}_{\widehat{H}_{1,S}}(\pi_1')$. \end{rem}

\paragraph{Extension des scalaires.} On suppose dorénavant que le corps $R$ est un corps parfait. Soit $\pi_1 \in \textup{Irr}_R^{\textup{gen}} (\widehat{H}_{1,S})$. On considère la décomposition provenant du Théorème \ref{decomposition_extension_des_scalaires_thm} de l'extension des scalaires à $\bar{R}$ :
$$\pi_1 \otimes_R \bar{R} \simeq  m_1 \bigg(\bigoplus_{w \in \textup{Gal}_R(E(\rho_1),\bar{R})} w \rho_1 \bigg)$$
où $E_1$ est le centre de $D_1$, la représentation $\rho_1 \in \textup{Rep}_{\bar{R}}^{\textup{gen}}(\widehat{H}_{1,S})$ est irréductible et son corps de rationalité $E(\rho_1)$ dans $\bar{R}$ est isomorphe à $E_1$.

\begin{lem} \label{big_theta_ext_des_scal_lem} On a un isomorphisme de représentations dans $\textup{Rep}_{\bar{R}}(\widehat{H}_{1,S} \times \widehat{H}_{2,S})$ :
$$(\Theta(\pi_1) \otimes_{D_1} \pi_1) \otimes_R \bar{R} \simeq \bigoplus_{w \in \textup{Gal}_R(E(\rho_1),\bar{R})} \Theta(w \rho_1) \otimes_{\bar{R}} w \rho_1.$$
De plus, en considérant $\Theta(\pi_1)$ comme une représentation dans $\textup{Rep}_{E_1}^{\textup{gen}}(\widehat{H}_{2,S})$, il existe une bijection $\varphi : \textup{Gal}_R(E(\rho_1), \bar{R}) \to \textup{Gal}_R(E_1,\bar{R})$ telle que :
$$\Theta(\pi_1) \otimes_{E_1,\varphi(w)} \bar{R} \simeq m_1 \big( \Theta(w \rho_1) \big).$$ \end{lem}

\begin{proof} En notant $S_{\bar{R}} = S \otimes_R \bar{R}$, il vient $\omega_{\psi,S}^R \otimes_R \bar{R} = \omega_{\psi,S_{\bar{R}}}^{\bar{R}}$. Le Théorème \ref{extension_des_scalaires_coinvariants_last_thm}, appliqué pour $V = \omega_{\psi,S}^R$, assure alors que :
$$(\omega_{\psi,S}^R)_{\pi_1} \otimes_R \bar{R} \simeq \bigoplus_{w \in \textup{Gal}_R(E(\rho_1),\bar{R})} (\omega_{\psi,S_{\bar{R}}}^{\bar{R}})_{w \rho_1}.$$
Par définition des membres de gauche et de droite, cela donne le premier isomorphisme de l'énoncé.

Ensuite, toute représentation $V_{\pi_1} \otimes_{E_1,w'} \bar{R}$ pour $w' \in \textup{Gal}_R(E_1, \bar{R})$ est isomorphe à une, et une seule, représentation $(V \otimes_R \bar{R})_{w \rho_1}$ pour $w \in \textup{Gal}_R(E(\rho_1),\bar{R})$. Cela définit donc la bijection $\varphi$.

Pour terminer, l'isomorphisme $(\Theta(\pi_1) \otimes_{D_1} \pi_1) \otimes_{E_1,\varphi(w)} \bar{R} \simeq \Theta(w \rho_1) \otimes_{\bar{R}} w \rho_1$ induit l'isomorphisme recherché grâce au fait que $D_1' = D_1 \otimes_{E_1,\varphi(w)} \otimes \bar{R} \simeq M_{m_1}(\bar{R})$ et :
\begin{eqnarray*}
(\Theta(\pi_1) \otimes_{D_1} \pi_1) \otimes_{E_1,\varphi(w)} \bar{R}  & \simeq &  (\Theta(\pi_1) \otimes_{E_1,\varphi(w)} \bar{R}) \otimes_{D_1'} (\pi_1 \otimes_{E_1,\varphi(w)} \bar{R}) \\
  & \simeq & (\Theta(\pi_1) \otimes_{E_1,\varphi(w)} \bar{R}) \otimes_{D_1'} (m_1 (w\rho_1)). \end{eqnarray*} \end{proof}

\begin{prop} \label{extension_des_scal_THETA1_prop} Les assertions suivantes sont équivalentes :
\begin{enumerate}[label=\textup{\alph*)}]
\item le $R[\widehat{H}_{2,S}]-D_1$-bimodule $\Theta(\pi_1)$ est de longueur finie ;
\item toutes les représentations $\Theta(w \rho_1)$ sont de longueur finie ;
\item il existe $w \rho_1$ tel que $\Theta(w \rho_1)$ soit de longueur finie.
\end{enumerate}
Cela prouve que $(\Theta_1) \Leftrightarrow (\Theta_1')$. \end{prop}

\begin{proof} Il suffit de montrer que a) $\Rightarrow$ b) et c) $\Rightarrow$ a), l'implication b) $\Rightarrow$ c) étant évidente.

Pour la première implication a) $\Rightarrow$ b), on remarque d'abord que si $\Theta(\pi_1)$ est un bimodule de longueur finie, alors $\Theta(\pi_1) \otimes_{D_1} \pi_1$ est une représentation de longueur finie dans $\textup{Rep}_R( \widehat{H}_{1,S} \times \widehat{H}_{2,S})$. Ensuite, comme toute représentation irréductible de $\widehat{H}_{1,S} \times \widehat{H}_{2,S}$ est admissible d'après le Corollaire \ref{rep_irred_gp_red_sont_adm_cor}, on déduit de l'exactitude de l'extension des scalaires à $\bar{R}$ et du Théorème \ref{decomposition_extension_des_scalaires_thm} que la représentation $(\Theta(\pi_1)\otimes_{D_1} \pi_1 ) \otimes_R \bar{R}$ est de longueur finie dans $\textup{Rep}_{\bar{R}}( \widehat{H}_{1,S} \times \widehat{H}_{2,S})$. Le Lemme \ref{big_theta_ext_des_scal_lem} permet de conclure que $\Theta(w \rho_1)$ est de longueur finie pour tout $w \rho_1$.

Pour la dernière implication c) $\Rightarrow$ a), on utilise encore le Lemme \ref{big_theta_ext_des_scal_lem}. L'égalité $\Theta(\pi_1) \otimes_{E_1,\varphi(w)} \bar{R} \simeq m_1 \big(\Theta(w\rho_1)\big)$ est compatible à l'action de $D_1 \otimes_{E_1,\varphi(w)} \bar{R} \simeq M_{m_1}(\bar{R})$ à droite, par conséquent, le bimodule $\Theta(\pi_1) \otimes_{E_1,\varphi(w)} \bar{R}$ est un $\bar{R}[\widehat{H}_{2,S}]-M_{m_1}(\bar{R})$-bimodule de longueur finie à condition que $\Theta(w \rho_1)$ soit une représentation de longueur finie dans $\textup{Rep}_{\bar{R}}(\widehat{H}_{2,S})$. Cela implique que $\Theta(\pi_1)$ est un $R[\widehat{H}_{2,S}]-D_1$-bimodule de longueur finie.
\end{proof}

\begin{prop} \label{extension_des_scal_THETA2_prop} On suppose que $\Theta(\pi_1)$ est un $R[\widehat{H}_{2,S}]-D_1$-bimodule de longueur finie, dont on note $\theta(\pi_1)$ le co-socle. Quand $\Theta(\pi_1) \neq 0$, les assertions suivantes sont équivalentes :
\begin{enumerate}[label=\textup{\alph*)}]
\item le $R[\widehat{H}_{2,S}]-D_1$-bimodule $\theta(\pi_1)$ est irréductible ;
\item toutes les représentations $\theta(w \rho_1)$ sont irréductibles ;
\item il existe $w \rho_1$ tel que $\theta(w \rho_1)$ soit irréductible.
\end{enumerate}
Cela prouve que $(\Theta_2) \Leftrightarrow (\Theta_2')$ si $(\Theta_1')$ est vraie pour $\pi_1$.
\end{prop}

\begin{proof} Comme précédemment, l'implication b) $\Rightarrow$ c) est évidente. 

Pour ce qui est de a) $\Rightarrow$ b), on prouve la contraposée. On suppose donc qu'il existe $w \rho_1$ telle que $\theta(w \rho_1)$ ne soit pas irréductible. Alors il existe $\tau_2$ et $\tau_2'$ deux représentations irréductibles dans $\textup{Rep}_{\bar{R}}^{\textup{gen}}(\widehat{H}_{2,S})$ telles que $\Theta(w \rho_1) \otimes_{\bar{R}} (w \rho_1)$ admette la représentation $(\tau_2 \otimes_{\bar{R}} (w \rho_1)) \oplus (\tau_2' \otimes_{\bar{R}} (w \rho_1))$ comme quotient.

D'une part, la représentation $\textup{Res}^R(\tau_2 \otimes_{\bar{R}} (w \rho_1))$ est $\pi_1$-isotypique en tant que représentation de $\widetilde{H}_{1,S}$.  D'après le Lemme \ref{pi-coinvariant}, il existe un $R[\widehat{H}_{2,S}]-D_1$-bimodule $\sigma_2$ tel que $\textup{Res}^R(\tau_2 \otimes_{\bar{R}} (w \rho_1)) \simeq \sigma_2 \otimes_{D_1} \pi_2$. De plus, elle est semi-simple en tant que représentation dans $\textup{Rep}_R(\widehat{H}_{1,S} \times \widehat{H}_{1,S})$. Donc $\sigma_2$ est isotypique et on note $\pi_2$ un facteur irréductible quelconque de $\sigma_2$. Il en va de même pour $\tau_2'$ avec des notations similaires $\sigma_2'$ et $\pi_2'$.

D'autre part, le Lemme \ref{big_theta_ext_des_scal_lem} entraîne que $\Theta(w \rho_1) \otimes_{\bar{R}} (w \rho_1) \simeq (\Theta(\pi_1) \otimes_{D_1} \pi_1) \otimes_{E_1,\varphi(w)} \bar{R}$. Par suite, on déduit de l'inclusion évidente de $\Theta(\pi_1) \otimes_{D_1} \pi_1$ dans le membre de droite que $\Theta(\pi_1) \otimes_{D_1} \pi_1$ admet comme quotient $(\pi_2 \otimes_{D_1} \pi_1) \oplus (\pi_2' \otimes_{D_1} \pi_1)$. Le noyau de ce quotient est de la forme $\sigma_2'' \otimes_{D_1} \pi_1$ où $\sigma_2''$ est un sous-$R[\widehat{H}_{2,S}]-D_1$-bimodule de $\Theta(\pi_1)$ d'après le Lemme \ref{pi-coinvariant-sous-rep}. Par conséquent, ce quotient induit un quotient $\Theta(\pi_1) \to \pi_2 \oplus \pi_2'$ de $R[\widehat{H}_{2,S}]-D_1$-bimodules dont le noyau est précisément $\sigma_2''$. D'où $\theta(\pi_1)$ n'est pas irréductible.

Enfin, pour l'implication c) $\Rightarrow$ a), on raisonne à nouveau par contraposée. On suppose donc que $\theta(\pi_1)$ n'est pas irréductible. Il s'agit de montrer que pour tout $w \rho_1$, la représentation $\theta(w \rho_1)$ n'est pas irréductible. Or, $\Theta(\pi_1) \otimes_{D_1} \pi_1$ a pour quotient $\theta(\pi_1) \otimes_{D_1} \pi_1$. Par conséquent, $(\Theta(\pi_1) \otimes_{D_1} \pi_1)\otimes_R \bar{R}$ a pour quotient $(\theta(\pi_1) \otimes_{D_1} \pi_1) \otimes_R \bar{R}$. Or la représentation $(\Theta(\pi_1) \otimes_{D_1} \pi_1)\otimes_{E_1,\varphi(w)} \bar{R}$ admet comme quotient $(\theta(\pi_1) \otimes_{D_1} \pi_1) \otimes_{E_1,\varphi(w)} \pi_1$ et $\theta(\pi_1) \otimes_{E_1,\varphi(w)} \bar{R} = m_1 (\theta(w \rho_1))$. D'où $\theta(w \rho_1)$ n'est pas irréductible puisque $\theta(\pi_1)$ ne l'est pas et l'exactitude de l'extension des scalaires donne que $\theta(\pi_1) \otimes_{E_1,\varphi(w)} \bar{R}$ est au moins de longueur $2 m_1$. \end{proof}

\begin{prop} \label{extension_des_scal_THETA3_prop} On suppose que $\Theta(\pi_1)$ et $\Theta(\pi_1')$ sont deux $R[\widehat{H}_{2,S}]-D_1$-bimodules de longueur finie dont les co-socles respectifs $\theta(\pi_1)$ et $\theta(\pi_1')$ sont irréductibles. On note $\rho_1$ et $\rho_1'$ les représentations qui interviennent dans l'extension des scalaires à $\bar{R}$ de $\pi_1$ et $\pi_1'$. On suppose que $D_1 = \textup{End}_{\widehat{H}_{1,S}}(\pi_1) \simeq \textup{End}_{\widehat{H}_{1,S}}(\pi_1')$. Alors les assertions suivantes sont équivalentes :
\begin{enumerate}[label=\textup{\alph*)}]
\item $\theta(\pi_1) \simeq \theta(\pi_1')$ en tant que $R[\widehat{H}_{2,S}]-D_1$-bimodules ;
\item pour tout $w \rho_1$, il existe $w' \rho_1'$ tel que $\theta(w \rho_1) \simeq \theta(w' \rho_1')$ ;
\item il existe $w \rho_1$ et $w' \rho_1'$ tel que $\theta(w \rho_1) \simeq \theta(w \rho_1')$.
\end{enumerate}
Cela prouve que $(\Theta_3) \Leftrightarrow (\Theta_3')$ si $(\Theta_1')$ et $(\Theta_2')$ sont vraies pour $\pi_1$ et $\pi_1'$.
\end{prop}

\begin{proof} L'implication b) $\Rightarrow$ c) est toujours évidente.

En ce qui concerne a) $\Rightarrow$ b), l'isomorphisme $\theta(\pi_1) \otimes_{E_1,\varphi(w)} \bar{R} \simeq m_1 (\theta(w \rho_1))$ est un isomorphisme de $\bar{R}[\widehat{H}_{2,S}]-(D_1 \otimes_{E_1,\varphi(w)} \bar{R})$-bimodules. Un isomorphisme similaire est valable pour $\theta(\pi_1')$ et un certain $\theta(w'\rho_1')$. Si $\theta(\pi_1) \simeq \theta(\pi_1')$, alors $\theta(w \rho_1) \simeq \theta(w' \rho_1')$ dans $\textup{Rep}_{\bar{R}}(\widetilde{H}_{2,S})$.

Pour la dernière implication c) $\Rightarrow$ a), il existe d'après le paragraphe précédent un isomorphisme $\theta(\pi_1) \otimes_{E_1,\varphi(w)} \bar{R} \simeq \theta(\pi_1') \otimes_{E_1,\varphi'(w')} \bar{R}$ puisque $\theta(w\rho_1) \simeq \theta(w' \rho_1')$. Soit $w_0 \in \textup{Gal}_R(E(\rho_1),\bar{R})$. Alors $w_0(\theta(\pi_1) \otimes_{E_1,\varphi(w)} \bar{R}) \simeq \theta(\pi_1) \otimes_{E_1 , \varphi(w_0 w)} \bar{R}$. En particulier, cela implique que $\theta(w_0 w \rho_1) \simeq \theta(w_0 w' \rho_1)$. Il existe donc une bijection $\psi$ de $\textup{Gal}_R(E(\rho_1),\bar{R})$ telle que pour tout $w_0 \in \textup{Gal}_R(E(\rho_1),\bar{R})$, on ait $\theta(w_0 \rho_1) \simeq \theta(\psi(w_0) \rho_1')$. Ainsi, on a des isomorphismes $\bar{R}[\widehat{H}_{2,S}]-(D_1\otimes_R \bar{R})$-bimodules :
$$\theta(\pi_1) \otimes_R \bar{R} \simeq \oplus_{w_0} \theta(\pi_1) \otimes_{E_1 , \varphi(w_0)} \bar{R} \simeq \theta(\pi_1') \otimes_R \bar{R}.$$
Par restriction des scalaires, $\textup{Res}^R(\theta(\pi_1) \otimes_R \bar{R})$ est un $R[\widehat{H}_{2,S}]-D_1$-bimodule qui est à la fois $\theta(\pi_1)$-isotypique et $\theta(\pi_1')$-isotypique. Donc $\theta(\pi_1) \simeq \theta(\pi_1')$. \end{proof}

En examinant les trois propositions précédentes \ref{extension_des_scal_THETA1_prop}-\ref{extension_des_scal_THETA2_prop}-\ref{extension_des_scal_THETA3_prop}, on obtient facilement le résultat suivant.

\begin{theo} \label{correspondance_sur_un_corps_non_alg_clos_equiv_thm} On rappelle que $R$ est un corps parfait dont on fixe une clôture algébrique $\bar{R}$. Alors les assertions suivantes sont équivalentes :
\begin{enumerate}[label=\textup{\alph*)}]
\item \textup{($\Theta_1'$)-($\Theta_2'$)-($\Theta_3'$)} est vrai pour toute représentation irréductible dans $\textup{Rep}_R^{\textup{gen}}(\widehat{H}_{1,S})$ ;
\item \textup{($\Theta_1$)-($\Theta_2$)-($\Theta_3$)} est vrai pour toute représentation irréductible dans $\textup{Rep}_{\bar{R}}^{\textup{gen}}(\widehat{H}_{1,S})$.
\end{enumerate} \end{theo}

\begin{rem} En particulier, cela signifie qu'il est suffisant de vérifier la validité des énoncés de la correspondance thêta locale pour les représentations à coefficients dans un corps algébriquement clos. Cela simplifie considérablement la situation considérée car dans ce cas on a toujours $D_1 = \textup{End}_{\widehat{H}_{1,S}}(\pi_1)= \bar{R}$. \end{rem}

\paragraph{Un exemple.} Soit $\rho_1 \in \textup{Irr}_{\mathbb{C}}^{\textup{gen}}(\widehat{H}_{1,S})$. Soit $K$ le corps des fractions de $\mathcal{A}$, qui est le corps de décomposition d'une famille de polynômes de $\mathbb{C}[X]$, donc un sous-corps de $\mathbb{C}$. Alors d'après le Lemme \ref{unicite_restriction_des_scalaires_induction_lem}, il existe une unique représentation $\pi_1(\rho_1,K) \in \textup{Rep}_K^{\textup{gen}}(\widehat{H}_{1,S})$ irréductible telle que $\pi_1 \otimes_K \mathbb{C}$ contienne $\rho_1$. Pour tout $w \in \textup{Gal}_K(\mathbb{C})$ :
$$\theta(w \rho_1) = \theta(\pi_1) \otimes_{E_1,\varphi(ww_0)} \mathbb{C} \simeq w( \theta(\pi_1) \otimes_{E_1,\varphi(w_0)} \mathbb{C}) = w (\theta(\rho_1)).$$

\subsection{Compatibilités à la réduction pour les paires de type I} \label{compatibilite_a_la_reduction_type_I_section}

Soit $(H_1,H_2)$ une paire duale de type I dans $\textup{Sp}(W)$ avec $F$ local non archimédien. Soit $\ell$ un premier qui ne divise pas le pro-ordre de $H_1$. Pour une extension algébrique $k$ du corps fini $\mathbb{F}_\ell$, on note $W(k)$ l'anneau des vecteurs de Witt et $K$ le corps des fractions de $W(k)$. C'est un anneau local complet de caractéristique $0$ d'idéal maximal $(\ell)$ et pour lequel on note $\mathfrak{r}_\ell : W(k) \to k$ la réduction modulo $\ell$. On note $R$ quand on veut signifier indistinctement un anneau parmi $K$, $W(k)$ et $k$.

Soit $\mu$ une mesure de Haar normalisée de $H_1$ à valeurs dans $K$. En particulier une telle mesure $\mu$ est à valeurs dans $W(k)$ puisque $\ell$ est banal. On note $\mathfrak{r}_\ell(\mu)$ sa réduction à valeurs dans $k$, qui est encore une mesure de Haar normalisée. Les algèbres de Hecke $\mathcal{H}_K(H_1)$ et $\mathcal{H}_{W(k)}(H_1)$ sont alors associées à la même mesure $\mu$ ; et l'algèbre de Hecke $\mathcal{H}_k(H_1)$ associée à $\mathfrak{r}_\ell(\mu)$. On note génériquement $\mathcal{H}_R(H_1)$ ces algèbres. Pour les choix de mesures précédents, on a des morphismes d'algèbres évidents :
$$\mathcal{H}_{W(k)}(H_1) \hookrightarrow \mathcal{H}_K(H_1) \textup{ et } \mathcal{H}_{W(k)}(H_1) \twoheadrightarrow \mathcal{H}_k(H_1).$$

\paragraph{Centre de Bernstein.}  Le centre de la catégorie $\textup{Rep}_R(H_1)$ s'identifie à la limite projective des centres des $\mathcal{H}_R(H_1,K_1)$ où $K_1$ parcourt les sous-groupes compacts ouverts de $H_1$ et les applications de transitions sont pour $K_1' \subset K_1$ :
$$f \in \mathcal{H}_R(H_1,K_1') \mapsto e_{K_1} f e_{K_1} \in \mathcal{H}_R(H_1,K_1) =  e_{K_1} \mathcal{H}_R(H_1,K_1') e_{K_1}$$
où $e_{K_1}$ est l'idempotent associé à $K_1$. On écrira ainsi $e=(e(K_1))_{K_1}$ si $e$ est un élément du centre de $\textup{Rep}_R(H_1)$, qu'on appelle aussi centre de Bernstein. En particulier l'unité du centre de Bernstein est $1 = (e_{K_1})_{K_1}$. Ces définitions dépendent du choix de la mesure de Haar sur $H_1$.

\paragraph{Idempotents centraux et décomposition.} Soit $S$ un sous-ensemble de $\textup{Irr}_R(H_1)$. On note $\textup{Rep}_R^S(H_1)$ la sous-catégorie pleine de $\textup{Rep}_R(H_1)$ dont les objets ont tous leurs sous-quotients irréductibles dans $S$. On dit qu'un sous-ensemble $S$ de $\textup{Irr}_R(H_1)$ décompose $\textup{Rep}_R(H_1)$ si l'on a  un produit de catégories :
$$\textup{Rep}_R(H_1) = \textup{Rep}_R^S(H_1) \times \textup{Rep}_R^{{}^cS}(H_1).$$
Il existe alors un (unique) idempotent central $e_S$ dans le centre de $\textup{Rep}_R(H_1)$ qui donne la décomposition précédente \textit{i.e.} de sorte que :
$$e_S \textup{Rep}_R(H_1)= \textup{Rep}_R^S(H_1) \textup{ et } (1-e_S)\textup{Rep}_R(H_1) = \textup{Rep}_R^{{}^cS}(H_1).$$
Réciproquement, tout idempotent central $e$ du cetre de Bernstein induit une décomposition de la catégorie $\textup{Rep}_R(H_1)$. Par définition, une telle décomposition induit une partition à deux éléments de $\textup{Irr}_R(H_1)$. On dit qu'un idempotent central $e$ est primitif si la catégorie $e \textup{Rep}_R(H_1)$ est une catégorie indécomposable. Ceci est équivalent au fait que $e$ ne s'écrive pas comme somme de deux idempotents centraux. On dit qu'un sous-ensemble $S$ non vide de $\textup{Irr}_R(H_1)$ définit un bloc dans $\textup{Rep}_R(H_1)$ si $S$ décompose $\textup{Rep}_R(H_1)$ et s'il n'existe pas de sous-ensemble propre non vide de $S$ qui décompose $\textup{Rep}_R(H_1)$. Enfin, l'idempotent central $e_S$ est primitif si et seulement si $S$ définit un bloc.

\paragraph{Représentations cuspidales.} Soit $\Pi_1 \in \textup{Rep}_K(H_1)$ une représentation cuspidale absolument irréductible. Comme le centre de $H_1$ est compact, la représentation $\Pi_1$ est un objet projectif et injectif de la catégorie $\textup{Rep}_K(H_1)$. Le fait que $\Pi_1$ soit projective et injective se reformule en disant que $\{\Pi_1\}$ décompose $\textup{Rep}_K(H_1)$. De plus, la catégorie $\textup{Rep}_K^{\{\Pi_1\}}(H_1)$ est semi-simple puisque toutes les représentations sont $\Pi_1$-isotypiques et seule compte la multiplicité de $\Pi_1$. On note $e_{\Pi_1}$ l'idempotent central primitif associé. 

On définit le foncteur de réduction modulo $\ell$ pour les représentations à coefficients dans $W(k)$ à l'aide de $\mathfrak{r}_\ell : W(k) \to k$ et que l'on le note encore :
$$\mathfrak{r}_\ell : V \in \textup{Rep}_{W(k)}(H_1) \mapsto V/\ell V = V \otimes_{W(k)} k \in \textup{Rep}_k(H_1).$$

\begin{rem} Ce foncteur n'est \textit{pas} le foncteur usuel de réduction modulo $\ell$, plutôt noté $r_\ell$ en général \cite[II.5.11.b]{vig}. On peut néanmoins obtenir $r_\ell$ pour les $W(k)$-réseaux admissibles en composant $\mathfrak{r}_\ell$ avec la semi-simplification pour les représentations dans $\textup{Rep}_k(H_1)$. Cependant, le foncteur $r_\ell$ présente le désavantage de n'être défini que pour les représentations entières et de supprimer toute information concernant les choix de réseaux. Il faut faire attention car, si $\Pi_1 \in \textup{Rep}_K(H_1)$ est entière, on a $\mathfrak{r}_\ell(\Pi_1) = 0$ alors que $r_\ell(\Pi_1)$ est une représentation semi-simple non nulle. \end{rem}

On suppose maintenant que $H_1$ admet des sous-groupes discrets co-compacts. Il en existe quand, par exemple, $F$ est de caractéristique $0$. D'après \cite[Cor. 6.10]{dat_nu}, tout $W(k)$-réseau stable de $\Pi_1$ se réduit modulo $\ell$ en une représentation cuspidale absolument irréductible, qui est unique à isomorphisme près. On la note $\pi_1$. En d'autres termes, si l'on considère $\Pi_1$ comme une représentation à coefficients dans $W(k)$, cela signifie que tous ses sous-quotients irréductibles sont isomorphes à $\pi_1$. Comme le centre de $H_1$ est compact et $\ell$ est banal vis-à-vis de $H_1$, la représentation $\pi_1$ est un objet projectif et injectif de la catégorie $\textup{Rep}_k(H_1)$.   De même qu'au paragraphe précédent, le singleton $\{ \pi_1 \}$ décompose $\textup{Rep}_k(H_1)$. On note $e_{\pi_1}$ l'idempotent central primitif associé. 

\paragraph{Compatibilité des décompositions.} On souhaite montrer que $e_{\Pi_1}$ est dans le centre de $\textup{Rep}_{W(k)}(H_1)$ au sens où $e_{\Pi_1}(K_1)$ est dans le centre de $\mathcal{H}_{W(k)}(H_1,K_1)$ pour tout sous-groupe ouvert $K_1$ de $H_1$. Le morphisme d'algèbres donné par la réduction modulo $\ell$ :
$$\mathfrak{r}_\ell : \mathcal{H}_{W(k)}(H_1) \twoheadrightarrow \mathcal{H}_k(H_1)$$
induit un morphisme évident entre centres de Bersntein. On aimerait alors montrer que, si $H_1$ admet des sous-groupes discrets co-compacts,  $r_\ell$ envoie $e_{\Pi_1}(K_1)$ sur $e_{\pi_1}(K_1)$ pour tout sous-groupe ouvert $K_1$ de $H_1$. Par conséquent, en notant $\mathfrak{r}_\ell(e_{\Pi_1})=(e_{\pi_1}(K_1))_{K_1}$, cela entraînerait l'égalité :
$$\mathfrak{r}_\ell(e_{\Pi_1}) = e_{\pi_1}.$$

\begin{lem} \label{idempotent_central_Pi1_coeff_W(k)_lem} On suppose que $F$ est de caractéristique résiduelle impaire. Alors pour tout sous-groupe ouvert compact $K_1$ de $H_1$ :
$$e_{\Pi_1}(K_1) \in \mathcal{H}_{W(k)}(H_1,K_1).$$ \end{lem}

\begin{proof} Soit $\Pi_1^{K_1}$ l'ensemble des vecteurs $K_1$-invariants de $\Pi_1$. Alors $\Pi_1^{K_1}$ est un $\mathcal{H}_K(H_1,K_1)$-module qui est soit irréductible, soit nul. Comme $\Pi_1$ est admissible, on a toujours que $\Pi_1^{K_1}$ est de dimension finie sur $K$. Soit $\mathcal{B}=(v_i)_{i \in I}$ une base de $\Pi_1^{K_1}$ sur $K$. On note $(v_i^\vee)_{i \in I}$ la base duale de $\mathcal{B}$ dans $(\Pi_1^\vee)^{K_1}$. La fonction :
$$e(K_1) : g \in H_1 \mapsto \sum_{i \in I} v_i^\vee(g^{-1}v_i) \in K$$ 
est $K_1$-bi-invariante. Elle est à support compact puisque le centre de $H_1$ est compact et que $\Pi_1$ est cuspidale, donc ses coefficients sont à support compact. Cette fonction appartient donc à $\mathcal{H}_K(H_1,K_1)$. Elle ne dépend pas du choix de la base $\mathcal{B}$ et appartient au centre de $\mathcal{H}_K(H_1,K_1)$ d'après \cite[I.7.8 e)]{vig}. De plus, cette définition est compatible aux applications de transitions. On en déduit que $e =(e(K_1))_{K_1}$ est un élément du centre de Bernstein.

La représentation $\Pi_1$ est cuspidale. Comme le centre de $H_1$ est compact, son caractère central se factorise par le caractère d'un groupe fini. Ce caractère central est donc à valeurs dans $W(k)^\times$, ce qui entraîne que la représentation $\Pi_1$ est entière. Ainsi la fonction $e(K_1)$ définie précédemment est à valeurs dans $W(k)$. Donc $e$ appartient au centre de $\textup{Rep}_{W(k)}(H_1)$.

L'élément $e$ du centre de Bernstein n'est pas forcément un idempotent, mais on décrit maintenant comme le normaliser pour qu'il le soit. La représentation $\Pi_1$ étant projective, elle admet un degré formel. Par définition, le degré formel est un élément $d_{\Pi_1} \in K^\times$ qui normalise $e=(e(K_1))_{K_1}$ de sorte que $d_{\Pi_1} e = (d_{\Pi_1} e(K_1))_{K_1}$ soit un idempotent central. Tel que défini ici, le degré formel dépend du choix de la mesure normalisée $\mu$ qui définit le produit de convolution dans $\mathcal{H}_K(H_1)$. Cependant, il définit une unique classe dans $K^\times/W(k)^\times$. Par conséquent, s'il existe une mesure normalisée $\mu$ telle que $d_{\Pi_1} \in W(k)^\times$, alors le degré formel appartient à $W(k)^\times$ pour toute mesure normalisée. En particulier on a $e_{\Pi_1} = d_{\Pi_1} e$.

Le but est de prouver que $e_{\Pi_1}$ appartient au centre de Bernstein de $\textup{Rep}_{W(k)}(H_1)$. Pour ce faire, il suffit de montrer que $d_{\Pi_1} \in W(k)^\times$ puisque $e$ appartient au centre de $\textup{Rep}_{W(k)}(H_1)$. Comme la caractéristique résiduelle de $F$ n'est pas $2$, toute représentation irréductible cuspidale a un type cuspidal \cite{kurinczuk_stevens} sur une clôture algébrique de $K$. Comme $\Pi_1$ est absolument irréductible, cela entraîne l'existence d'un sous-groupe compact ouvert $J_1$ de $H_1$ et d'une représentation absolument irréductible $\lambda_1 \in \textup{Rep}_K(J_1)$ tels que :
$$\Pi_1 \simeq \textup{ind}_{J_1}^{H_1}(\lambda_1).$$ 

Comme $\Pi_1$ est de plus projective, les degrés formels de $\Pi_1$ et de $\lambda_1$ sont égaux d'après \cite[I.8.4]{vig}. Or, le degré formel de $\lambda_1$ est celui d'une représentation d'un groupe fini $J_1'$ quotient de $J_1$. Donc le cardinal $J_1'$ divise le pro-ordre de $J_1$. On a alors \cite[I.7.8.b)]{vig} en prenant pour $\mu$ la mesure normalisée sur $J_1$ :
$$d_{\lambda_1}=\textup{dim}_K(\lambda_1)$$
où $\textup{dim}_K(\lambda_1)$ divise le cardinal de $J_1'$ puisque $\ell$ est banal et $\lambda_1$ est absolument irréductible. Donc $d_{\Pi_1} = \textup{dim}_K(\lambda_1) \in W(k)^\times$ puisque le cardinal de $J_1'$ est inversible dans $W(k)$. \end{proof}

\begin{lem} On suppose que $H_1$ admet des sous-groupes discrets co-compacts et que la caractéristique résiduelle de $F$ est impaire. On a une décomposition induite par $e_{\Pi_1}$ :
$$\textup{Rep}_{W(k)}(H_1) = e_{\Pi_1} \textup{Rep}_{W(k)}(H_1) \times (1- e_{\Pi_1}) \textup{Rep}_{W(k)}(H_1).$$
De plus $e_{\Pi_1}$ est un idempotent central primitif du centre de Bernstein de $\textup{Rep}_{W(k)}(H_1)$ dont l'image par $r_\ell$ est $e_{\pi_1}$. On a enfin que $\{\pi_1\}$ décompose $\textup{Rep}_{W(k)}(H_1)$ et :
$$e_{\Pi_1} \textup{Rep}_{W(k)}(H_1)= \textup{Rep}_{W(k)}^{\{\pi_1\}}(H_1)$$
est un bloc de $\textup{Rep}_{W(k)}(H_1)$. \end{lem}

\begin{proof} On sait déjà que l'idempotent central $e_{\Pi_1}$ appartient au centre de la catégorie $\textup{Rep}_{W(k)}(H_1)$, ce qui induit la première décomposition de l'énoncé. Ensuite, le morphisme d'algèbres :
$$\mathfrak{r}_\ell : \mathcal{H}_{W(k)}(H_1) \twoheadrightarrow \mathcal{H}_k(H_1)$$
induit un morphisme entre centres de Bernstein via $e=(e(K))_K \mapsto \mathfrak{r}_\ell(e) = (\mathfrak{r}_\ell \circ e(K))_K$ où la fonction $\mathfrak{r}_\ell \circ e(K)\in \mathcal{H}_k(H_1,K_1)$ est l'image de $e(K) \in \mathcal{H}_{W(k)}(H_1,K_1)$ par la réduction modulo $\ell$, que l'on note encore $\mathfrak{r}_\ell : W(k) \twoheadrightarrow k$.

Ensuite, soit $L$ un réseau entier de $\Pi_1$. Alors $L^\vee$ est un réseau naturel de $\Pi_1^\vee$ et en choisissant une base $\mathcal{B}=(v_i)_{i \in I}$ du $W(k)$-module $L^{K_1}$, la base duale de $\mathcal{B}$ est une base du $W(k)$-module $L^\vee$. Un rappel de la preuve précédente donne que la fonction :
$$e(K_1) : g \in H_1 \mapsto \sum_{i \in I} v_i^\vee(g^{-1}v_i) \in K$$
induit, à un facteur $d_{\Pi_1} \in W(k)^\times$ près, un idempotent central dans $\mathcal{H}_{W(k)}(H_1,K_1)$. Or $\mathfrak{r}_\ell(L)=\pi_1$ est une représentation cuspidale irréductible d'après \cite[Cor. 6.10]{dat_nu}. De plus $L^{K_1} \otimes_{W(k)} k \simeq (L \otimes_{W(k)} k)^{K_1} = (\pi_1)^{K_1}$ car $\{1_{K_1}\}$ décompose $\textup{Rep}_{K}(H_1)$ et l'idempotent central associée $e_{1_{K_1}}$, de réduction $\mathfrak{r}_\ell(e_{1_{K_1}})= e_{\mathfrak{r}_\ell(1_{K_1})}$, est dans le centre de $\textup{Rep}_{W(k)}(H_1)$. 

Comme $\mathfrak{r}_\ell ( L^\vee \otimes_{W(k)} L) \simeq \pi_1^\vee \otimes_k \pi_1$, on en déduit que :
$$\mathfrak{r}_\ell \circ e(K_1) : g \in H_1 \mapsto \mathfrak{r}_\ell \bigg( \sum_{i \in I} v_i^\vee(g^{-1}v_i) \bigg) = \sum_{i \in I} \mathfrak{r}_\ell(v_i^\vee)(g^{-1} \mathfrak{r}_\ell(v_i))\in k$$
où $\mathfrak{r}_\ell(\mathcal{B})=(\mathfrak{r}_\ell(v_i))_{i \in I}$ est une base de $(\pi_1)^{K_1}$, de base duale $(\mathfrak{r}_\ell(v_i^\vee))_{i\in I}$ dans $(\pi_1)^\vee$. Donc c'est à un facteur près l'idempotent central $e_{\pi_1}$ de $\mathcal{H}_k(H_1,K_1)$ associé à la représentation cuspidale projective $\pi_1$. Or, si $d_{\Pi_1} \in W(k)^\times$ normalise $e$ pour une certaine mesure normalisée $\mu$, alors le facteur de normalisation de $\mathfrak{r}_\ell \circ e$ est $d_{\pi_1} = \mathfrak{r}_\ell(d_{\Pi_1}) \in k^\times$ pour la mesure normalisée $\mathfrak{r}_\ell(\mu)$. Ainsi $\mathfrak{r}_\ell(e_{\Pi_1})= \mathfrak{r}_\ell(d_{\Pi_1} e) = d_{\pi_1} \mathfrak{r}_\ell(e) = e_{\pi_1}$.

Il est clair que $e_{\Pi_1}$ est un idempotent central primitif du centre de $\textup{Rep}_K(H_1)$. Comme l'inclusion d'algèbres :
$$\mathcal{H}_{W(k)}(H_1) \to \mathcal{H}_K(H_1)$$
induit une inclusion des centres de Bernstein, on en déduit que $e_{\Pi_1}$, en tant qu'élément du centre de $\textup{Rep}_{W(k)}(H_1)$, est primitif puisqu'il l'est dans le centre de $\textup{Rep}_K(H_1)$.\end{proof}

\begin{prop} On suppose que $H_1$ admet des sous-groupes discrets co-compacts et que la caractéristique résiduelle de $F$ est impaire. Soit $V \in \textup{Rep}_{W(k)}(H_1)$. Alors pour tout $V \in \textup{Rep}_{W(k)}(H_1)$ :
$$\mathfrak{r}_\ell(e_{\Pi_1} V) = e_{\pi_1} \mathfrak{r}_\ell(V).$$ \end{prop}

\begin{rem} On insiste sur le fait que le foncteur $r_\ell$ n'est pas exact à gauche. En effet, pour tout $W(k)$-réseau stable $L_1$ de $\Pi_1$, on a $\mathfrak{r}_\ell(L_1)=\pi_1$ ; alors que $\mathfrak{r}_\ell(\Pi_1)=0$. \end{rem}

\begin{proof} On a $\ell e_{\Pi_1} V = e_{\Pi_1} (\ell V)$ et le foncteur $V \mapsto e_{\pi_1} V$ est exact. Donc :
$$\mathfrak{r}_\ell(e_{\Pi_1} V) = e_{\Pi_1} V / \ell e_{\Pi_1} V = e_{\Pi_1} \mathfrak{r}_\ell(V).$$
Or, l'action naturelle de $W(k)$ sur $\mathfrak{r}_\ell(V)$ se factorise par $k$ puisque $\ell (\mathfrak{r}_\ell (V)) = 0$. Ainsi :
$$e_{\Pi_1} \mathfrak{r}_\ell(V) = \mathfrak{r}_\ell(e_{\Pi_1})( \mathfrak{r}_\ell(V)) = e_{\pi_1} \mathfrak{r}_\ell(V).$$ \end{proof}

\paragraph{Réduction modulo $\ell$ de la représentation de Weil.} On rappelle que $(H_1,H_2)$ est une paire duale de type I dans un groupe symplectique $\textup{Sp}(W)$ sur un corps local non archimédien $F$. Soit $k$ un corps parfait de caractéristique $\ell$ tel qu'il existe un caractère additif non trivial de $F$. Soit $\psi$ un caractère lisse non trivial de $F$ à valeurs dans $W(k)$. Alors $\psi$ est en particulier à valeurs dans $K$, et il se réduit en un caractère non trivial à valeurs dans $k$, encore noté $\psi$. On se permet cet abus de notation car le contexte est toujours clair.

Soit $X$ un lagrangien de $W$. On note $V_X^R$ le modèle de la représentation métaplectique à coefficients dans $R$ associé à $\psi$ et $X$, où $V_X^{W(k)}$ est défini dans la Remarque \ref{extension_des_scalaires_reduction_des_scalaires_rem}. On a des morphismes équivariants pour l'action du groupe d'Heisenberg :
$$V_X^{W(k)} \hookrightarrow V_X^K \textup{ et } V_X^{W(k)} \twoheadrightarrow V_X^k$$
qui sont donnés par l'inclusion naturelle $W(k) \to K$ et le morphisme de réduction $\mathfrak{r}_\ell : W(k) \to k$. Le modèle de la représentation de Weil sur $R$ associé à $X$ est :
$$(\omega_{\psi,V_X^R}^R,V_X^R) \in \textup{Rep}_R(\widehat{\textup{Sp}}_{\psi,V_X}^K(W)),$$
où l'on considère ces représentations comme celles d'un même groupe $\widehat{\textup{Sp}}_{\psi,V_X}^K(W)$, conformément aux compatibilités définies dans la Remarque \ref{extension_des_scalaires_reduction_des_scalaires_rem}. En particulier, les morphismes précédents sont $\widehat{\textup{Sp}}_{\psi,V_X}^K(W)$-équivariants :
$$\omega_{\psi,V_X^{W(k)}}^{W(k)} \hookrightarrow \omega_{\psi,V_X^K}^K \textup{ et } \omega_{\psi,V_X^{W(k)}}^{W(k)} \twoheadrightarrow \omega_{\psi,V_X^k}^k.$$
Enfin, on note $\widehat{H}_1$ et $\widehat{H}_2$ les images réciproques de $H_1$ et $H_2$ dans $\widehat{\textup{Sp}}_{\psi,V_X}^K(W)$. Ainsi :
$$\omega_{\psi,V_X^R}^R \in \textup{Rep}_R( \widehat{H}_1 \times \widehat{H}_2).$$

\paragraph{Réduction de la correspondance thêta classique dans le cas banal.} On suppose maintenant de plus que $\ell$ ne divise pas le pro-ordre de $H_1$. Quand $\widehat{H}_1$ et $\widehat{H}_2$ sont scindés sur $H_1$ et $H_2$, on a des équivalences de catégories :
$$\textup{Rep}_R^{\textup{gen}}(\widehat{H}_1) \simeq \textup{Rep}_R(H_1)  \textup{ et } \textup{Rep}_R^{\textup{gen}}(\widehat{H}_2) \simeq \textup{Rep}_R(H_2).$$
Ces catégories partagent donc les mêmes propriétés au sens suivant : une représentation dans $\textup{Rep}_R(\widehat{H}_1)$ est projective, respectivement injective ou cuspidale ou entière, si et seulement si son image dans $\textup{Rep}_R(H_1)$ l'est. En outre, les centres de ces catégories sont isomorphes.  Quand les paires ne sont pas scindés, il faut remplacer partout $H_1$ par le relevé $\widehat{H}_1$, puis $H_2$ par $\widehat{H}_2$, mais la théorie reste la même car la théorie des types, du centre de Bernstein et le principe de Brauer Nesbitt se traitent identiquement.

Soit $\Pi_1$ une représentation cuspidale absolument irréductible dans $\textup{Rep}_K^{\textup{gen}}(\widehat{H}_1)$. En considérant $\Pi_1$ comme une représentation à coefficients dans $W(k)$, on note comme précédemment $\pi_1$ l'unique sous-quotient absolument irréductible de $\Pi_1$. En particulier $\pi_1$ est aussi une représentation cuspidale dans $\textup{Rep}_k^{\textup{gen}}(\widehat{H}_1)$. Un résultat bien connu du cas complexe \cite[Chap. 3, IV.4 Th. 1)a)]{mvw} assure alors que la représentation $\Theta(\Pi_1)$ est irréductible quand elle est non nulle. Si de plus $\Theta(\Pi_1)$ admet un $W(k)$-réseau $L$ stable par $\widehat{H}_2$, cette représentation est entière. Le principe de Brauer-Nesbitt impose alors que la semi-simplifiée de $\mathfrak{r}_\ell(L)$ est de longueur finie et ne dépend pas du choix de $L$.

\begin{prop} \label{reduction_big_theta_ss_brauer_nesbitt_prop} On suppose que $H_1$ admet des sous-groupes discrets co-compacts et que la caractéristique résiduelle de $F$ est impaire. On rappelle que $\ell$ est supposé banal vis-à-vis de $H_1$. Alors la représentation $\Theta(\Pi_1)$ est entière et les semi-simplifiées de $\mathfrak{r}_\ell (\Theta(\Pi_1))$ et $\Theta(\pi_1)$ sont isomorphes. En particulier, la représentation $\Theta(\pi_1)$ est de longueur finie. \end{prop}

\begin{proof} Comme la catégorie $\textup{Rep}_K^{\{\Pi_1\}}(\widehat{H}_1)$ est semi-simple, il vient :
$${\big(\omega_{\psi,V_X^K}^K \big)}_{\Pi_1} \simeq e_{\Pi_1} \omega_{\psi,V_X^K}^K.$$
Cette construction est compatible à l'action de $\widehat{H}_2$, donc cet isomorphisme est entre représentations de $\textup{Rep}_K(\widehat{H}_1 \times \widehat{H}_2)$. Comme $\Pi_1$ est absolument irréductible, en fixant un plongement de $K$ dans $\mathbb{C}$, on a que $\Pi_1 \otimes_K \mathbb{C}$ est irréductible et :
$$\Theta(\Pi_1 \otimes_K \mathbb{C}) \simeq \Theta(\Pi_1) \otimes_K \mathbb{C}.$$
Les énoncés $(\Theta_1)$-$(\Theta_2)$-$(\Theta_3)$ étant valides sur $\mathbb{C}$, il vient que $\Theta(\Pi_1 \otimes_K \mathbb{C})$ est soit nulle soit irréductible. Donc cela implique que $\Theta(\Pi_1)$ est absolument irréductible quand elle est non nulle. Cela montre que $e_{\Pi_1} \omega_{\psi,V_X^K}^K = \Theta(\Pi_1) \otimes_K \Pi_1$ est irréductible.

En voyant $e_{\Pi_1} \omega_{\psi,V_X^K}^K$ comme une représentation à coefficients dans $W(k)$, le sous-module:
$$e_{\Pi_1} \omega_{\psi,V_X^{W(k)}}^{W(k)} \in \textup{Rep}_{W(k)}(\widehat{H}_1 \times \widehat{H}_2)$$
est l'image de $\omega_{\psi,V_X^{W(k)}}^{W(k)}$ par l'idempotent central $e_{\Pi_1}$, vu comme un élément du centre de Bernstein de la catégorie $\textup{Rep}_{W(k)}(\widehat{H}_1 \times \widehat{H}_2)$. Or :
$$\mathfrak{r}_\ell(e_{\Pi_1} \omega_{\psi,V_X^{W(k)}}^{W(k)}) = e_{\pi_1} \mathfrak{r}_\ell(\omega_{\psi,V_X^{W(k)}}^{W(k)}) = e_{\pi_1} \omega_{\psi,V_X^k}^k = \Theta(\pi_1) \otimes_k \pi_1.$$

Pour conclure, il s'agit maintenant de montrer que $e_{\Pi_1} \omega_{\psi,V_X^{W(k)}}^{W(k)} \in \textup{Rep}_{W(k)}(\widehat{H}_1 \times \widehat{H}_2)$ est un $W(k)$-réseau stable de la représentation irréductible :
$$e_{\Pi_1} \omega_{\psi,V_X^K}^K \simeq \Theta(\Pi_1) \otimes_K \Pi_1 \in \textup{Rep}_K(\widehat{H}_1 \times \widehat{H}_2).$$
L'anneau $W(k)$ est local, principal et complet ; la dimension sur $K$ de la représentation $e_{\Pi_1} \omega_{\psi,V_X^K}^K$ est dénombrable car elle est de longueur finie. D'après \cite[I.9.2]{vig}, tout sous-$W(k)$-module stable de $e_{\Pi_1} \omega_{\psi,V_X^K}^K$ qui ne contient pas de droite est $W(k)$-libre. Or $\omega_{\psi,V_X^{W(k)}}^{W(k)}$ s'identifie à un espace de fonctions qui ne contient pas de droite. Comme $e_{\Pi_1}$ est un idempotent central, on a :
$$\omega_{\psi,V_X^{W(k)}}^{W(k)}=e_{\Pi_1} \omega_{\psi,V_X^{W(k)}}^{W(k)} \oplus (1-e_{\Pi_1})\omega_{\psi,V_X^{W(k)}}^{W(k)}.$$
Donc $e_{\Pi_1} \omega_{\psi,V_X^{W(k)}}^{W(k)}$ ne contient pas de droites. C'est donc bien un $W(k)$-réseau stable de la représentation $e_{\Pi_1} \omega_{\psi,V_X^K}^K$, ce qui entraîne que la représentation irréductible $\Theta(\Pi_1)$ est entière.

Par le principe de Brauer-Nesbitt les semi-simplifiées des réductions modulo $\ell$ de tout $W(k)$-réseau stable de $e_{\Pi_1} \omega_{\psi,V_X^K}^K$ sont égales. En d'autres termes, pour tous $W(k)$-réseaux stables $L_1$ de $\Pi_1$ et $L_{\Theta(\Pi_1)}$ de $\Theta(\Pi_1)$, les semi-simplifiées de :
$$\mathfrak{r}_\ell(e_{\Pi_1} \omega_{\psi,V_X^{W(k)}}^{W(k)})=\Theta(\pi_1) \otimes_k \pi_1 \textup{ et }  \mathfrak{r}_\ell(L_{\Theta(\Pi_1)} \otimes_{W(k)}L_1) = \mathfrak{r}_{\ell}(L_{\Theta(\Pi_1)}) \otimes_k \pi_1$$
sont isomorphes. On en déduit que les semi-simplifiées de $\mathfrak{r}_\ell(L_{\Theta(\Pi_1)})$ et $\Theta(\pi_1)$ sont isomorphes. Comme la première est de longueur finie \cite[II.5.11.a]{vig}, cela est aussi vrai pour $\Theta(\pi_1)$. \end{proof}

On peut améliorer ce résultat dans le cas banal avec une condition sur $\Theta(\Pi_1)$.

\begin{theo} On suppose que $H_1$ admet des sous-groupes discrets co-compacts et que la caractéristique résiduelle de $F$ est impaire. On suppose que $\ell$ est banal vis-à-vis de $H_1$ et que $\Theta(\Pi_1)$ est une représentation cuspidale absolument irréductible. Alors $\Theta(\pi_1)$ est une représentation cuspidale absolument irréductible. \end{theo}

\begin{proof} D'après \cite[Cor. 6.10]{dat_nu}, la réduction modulo $\ell$ de tout $W(k)$-réseau stable de $\Theta(\Pi_1)$ est absolument irréductible et cuspidale. Donc la semi-simplifiée de $\mathfrak{r}_\ell(\Theta(\Pi_1))$ est une représentation cuspidale absolument irréductible, qui n'est autre que $\mathfrak{r}_\ell(\Theta(\pi_1))$ elle-même par irréductibilité. Par conséquent la Proposition \ref{reduction_big_theta_ss_brauer_nesbitt_prop} précédente entraîne que $\mathfrak{r}_\ell(\Theta(\Pi_1)) \simeq \Theta(\pi_1)$, d'où le résultat pour $\Theta(\pi_1)$. \end{proof}

\begin{rem} Dans les résultats précédents, on peut remplacer $K$ par n'importe quelle extension algébrique finie totalement ramifiée $L$ de $K$ et $W(k)$ par l'anneau des entiers $\mathcal{O}_L$ de $K$. En effet, en choisissant une uniformisante $\varpi_L$ de $\mathcal{O}_L$, l'anneau $\mathcal{O}_L$ de $L$ est encore principal, local et complet. De plus, on peut vérifier que \cite[Cor. 6.10]{dat_nu} est toujours valable dans ce cadre. \end{rem}

\begin{rem} \label{reduction_mod_ell_longue_rem} La difficulté principale quand on veut appliquer des arguments de réduction semble provenir du fait que l'on a besoin d'une description précise du plus grand quotient $\Pi_1$-isotypique :
$$ \omega_{\psi,V_X^K}^K \twoheadrightarrow {(\omega_{\psi,V_X^K}^K)}_{\Pi_1}.$$
Dans la présente section, cela a été traduit en termes d'action d'un élément $e_{\Pi_1}$ du centre de Bernstein qui donne de bonnes compatibilités à la réduction. En toute généralité, pour espérer appliquer des arguments de réduction -- \textit{i.e.} pour passer de $K$ à $k$ -- à la correspondance thêta, il paraît inévitable de devoir décrire plus explicitement la formation de ces $\Pi_1$-coinvariants. Plusieurs problèmes ouverts en ce sens sont les suivants. Comment traiter les cas où l'on relâche les conditions et où : 
\begin{itemize}[label=$\bullet$]
\item $k$ est un corps parfait de caractéristique $\ell$ ? Les preuves précédentes restent valables à condition de savoir que la correspondance thêta est valide sur tout corps de caractéristique $0$. Grâce à la Section \ref{compatibilité_extension_des_scalaires_section} néanmoins, ce résultat était déjà connu sur tout corps isomorphe à un sous-corps de $\mathbb{C}$. Bien évidemment, quand $k$ est une extension algébrique de $\mathbb{F}_\ell$, le corps $K$ est isomorphe à un sous-corps de $\mathbb{C}$.
\item $H_1$ n'admet pas de sous-groupes discrets co-compacts ? Le résultat que l'on utilise de manière clé \cite[Cor. 6.10]{dat_nu} repose sur ce point. La construction que l'on a présentée semble alors échouer. En effet, on ne sait même pas en général si toute représentation irréductible à coefficients dans $k$ intervient comme sous-quotient de la réduction  modulo $\ell$ d'un réseau stable d'une représentation irréductible à coefficients dans $K$, ce qui est précisément l'énoncé \cite[Lem. 6.8 i)]{dat_nu}.
\item $\ell$ n'est pas banal vis-à-vis de $H_1$ ? Toujours dans \cite[Cor. 6.10]{dat_nu}, l'hypothèse de banalité est capitale. Néanmoins, on peut considérer une classe plus fine de représentations même si $\ell$ n'est pas banal, à savoir celles dont la réduction modulo $\ell$ de tout réseau reste irréductible et projective. On peut également étudier les idempotents centraux primitifs $e_{\Pi_1}$ qui ont \og bonne réduction \fg{}  modulo $\ell$ \textit{i.e.} ceux à coefficients entiers dont la réduction est un idempotent central $\mathfrak{r}_\ell(e_{\Pi_1})$ non nul.
\item $F$ est de caractéristique résiduelle $2$ ? Cette hypothèse est nécessaire pour calculer le degré formel à l'aide d'un argument de théorie des types \cite{kurinczuk_stevens} et montrer que les formules qui donnent l'idempotent central $e_{\Pi_1}$ sont à coefficients entiers et se réduisent bien modulo $\ell$. Une étude plus poussée de la caractéristique résiduelle $2$ pourrait fonctionner puisque \cite[Cor. 6.10]{dat_nu} est encore valable.
\item $\Pi_1$ n'est pas cuspidale ? La stratégie présentée paraît échouer car elle repose de manière cruciale sur le fait que $H_1$ est à centre compact, et que les représentations irréductibles cuspidales sont alors projectives. Or cela est manifestement faux si l'on considère un Levi de $H_1$ puisque des facteurs de type $\textup{GL}$ interviennent. \end{itemize} \end{rem}

\begin{rem} Dans la preuve de la Proposition \ref{reduction_big_theta_ss_brauer_nesbitt_prop}, on invoque le principe de Brauer-Nesbitt pour montrer l'égalité entre semi-simplifiées de la réduction de deux réseaux. Peut-on décrire plus explicitement le réseau $e_{\Pi_1} \omega_{\psi,V_X^{W(k)}}^{W(k)}$ ? Par exemple, si l'on prouvait que le $W(k)$-réseau stable :
$$e_{\Pi_1} \omega_{\psi,V_X^{W(k)}}^{W(k)} \in \textup{Rep}_{W(k)}(\widehat{H}_1 \times \widehat{H}_2)$$
était de la forme $L_{\Theta(\Pi_1)} \otimes_{W(k)} L_1$, avec $L_1$ un $W(k)$-réseau stable de $\Pi_1$ et $L_{\Theta(\Pi_1)}$ un $W(k)$-réseau stable de $\Theta(\Pi_1)$, on aurait une égalité plus précise, à savoir :
$$\mathfrak{r}_\ell(L_{\Theta(\Pi_1)}) = \Theta(\pi_1).$$  \end{rem}

\appendix

\section{Lien avec \cite{ct}} \label{lien_avec_ct_section}

On va montrer que les objets construits dans \cite{weil}, et généralisés dans \cite{ct}, s'identifient au groupe métaplectique et à la représentation de Weil tels que nous les avons construits, c'est-à-dire en suivant le développement de \cite{mvw}.

\subsection{Construction de \cite{weil} et \cite{ct}}

Soient $X$ un espace vectoriel de dimension finie sur $F$ et $W = X \oplus X^*$. Par hypothèse sur $R$, il existe un caractère lisse non trivial $\psi$ de $F$ à valeurs dans $R^\times$.

\subsubsection{Rappels utiles}

Les faits présentés proviennent de \cite[1 \& 2]{ct}.

\paragraph{Formes quadratiques et bicaractères.} On rappelle qu'une \textit{forme quadratique} sur $X$ est une application continue $X \to F$ telle que pour tout $x \in X$ et tout $u \in F$, $f(ux)=u^2f(x)$, et $(x,y)\mapsto f(x+y)-f(x)-f(y)$ est $F$-bilinéaire. Un \textit{caractère de degré 2} de $X$ est une application $\varphi:X \to R^\times$ telle que $(x,y) \mapsto \varphi(x+y) \varphi(x)^{-1} \varphi(y)^{-1}$ est un bicaractère de $X \times X$ \textit{i.e.} un caractère lisse en chaque variable. À toute forme quadratique $f$ peut être associé un bicaractère $\psi \circ f$.
 
On note $[.,.]$ le crochet de dualité. En identifiant $(X^*)^* \simeq X$, on écrit $[x,x^*]=[x^*,x]$ indifféremment. On rappelle que $\hat{X}_R \simeq X^*$ via $x^* \mapsto \psi([.,x^*])$.

\paragraph{Le groupe symplectique.} Soit $\mathcal{B}$ la forme bilinéaire de $( X \oplus X^*) \times ( X \oplus X^*)$ dans $F$ définie par $\mathcal{B}((x_1,x_1^*),(x_2,x_2^*))=[x_1,x_2^*]$. 

\begin{defi} On note $\text{Sp}(W)$ le groupe des automorphismes $\sigma$ de $W$ vérifiant $$\mathcal{B}(\sigma(w_1),\sigma(w_2)) - \mathcal{B}(\sigma(w_2),\sigma(w_1)) = \mathcal{B}(w_1,w_2) - \mathcal{B}(w_2,w_1).$$ \end{defi}

\begin{prop} Muni de la forme antisymétrique :
$$\langle w_1,w_2 \rangle = \mathcal{B}(w_1,w_2) - \mathcal{B}(w_2,w_1)$$
le groupe des isométries de l'espace symplectique $(W,\langle .,. \rangle )$ est $\emph{Sp}(W)$. \end{prop}

On associe à tout élément $\sigma \in \text{Sp}(W)$ une forme quadratique définie par $$f_{\sigma}(w)=\frac{1}{2} \big( \mathcal{B}(\sigma(w),\sigma(w)) - \mathcal{B}(w,w) \big)$$
On vérifie facilement la relation de cocycle pour tout $\sigma_1, \sigma_2 \in \text{Sp}(W)$ $$f_{\sigma_1,\sigma_2} = f_{\sigma_1} + f_{\sigma_2} \circ \sigma_1$$ et, pour tout $\sigma \in \text{Sp}(W)$ et tout $w_1, w_2 \in W$ $$f_\sigma(w_1+w_2)-f_\sigma(w_1)-f_\sigma(w_2)= \mathcal{B}(\sigma(w_1),\sigma(w_2)) - \mathcal{B}(w_1,w_2).$$

\begin{rem} On pose $\mathcal{F}= \psi \circ \mathcal{B}$. En composant par $\psi$ les relations précédentes, on obtient essentiellement des conditions analogues pour des bicaractères $\psi \circ f_\sigma$.
\end{rem}

\paragraph{Définitions de $A(W)$ et $B_0(W)$.} Soit $A(W)$, noté $A$ dans la suite, l'ensemble $W \times R^\times$ muni de la loi de groupe :
$$(w,r) \cdot (w',r')=(w+w',r r' \mathcal{F}(w,w')).$$
Le centre $Z$ de $A$ s'identifie à $R^\times$ via $r \mapsto (0,r)$.

On note $B_0=\text{Aut}_Z(A)$ le sous-groupe des automorphismes de $A$ triviaux sur le centre de $A$  \textit{i.e.} $B_0=\{\sigma \in \text{Aut}(A) \ | \ \forall z \in Z, \ \sigma(z)=z\}$.

\begin{prop} On a l'isomorphisme de groupe suivant : $$\begin{array}{ccc}
\emph{Sp}(W) \rtimes W^* & \longrightarrow & B_0 \\
(\sigma,\tau) &\longmapsto& \big((w,r) \mapsto (\sigma(w),r \varphi_{\sigma,\tau}) \big)
\end{array}$$
où $\varphi$ est le bicaractère défini par $\varphi_{\sigma,\tau} = (\psi \circ f_\sigma) \ (\psi \circ \tau)$.
\end{prop}

\paragraph{Définitions de $U_0$, $\mathbb{A}(W)$ et $\mathbb{B}_0(W)$.} Pour $w=(u,u^*) \in X \oplus X^*=W$ et $r \in R^\times$, on définit un opérateur $U(w,r) \in \textup{GL}_R(C_c^\infty(X))$ de la manière suivante :
$$\begin{array}{cccc}
U_0 : & A & \longrightarrow & \textup{GL}_R(C_c^\infty(X)) \\
 & (w,r) & \longmapsto & \big( \phi \mapsto r \ \psi([.,u^*]) \ \Phi(.+u) \big)
\end{array}.$$
C'est un morphisme de groupes injectif, on note $\mathbb{A}$ son image. On a donc $A \simeq \mathbb{A}$ via $U_0$. On en déduit que leurs groupes d'automorphismes sont isomorphes, et même mieux : $\text{Aut}_Z (\mathbb{A}) \simeq \text{Aut}_Z (A)=B_0$.

De plus, le groupe $B_0$ agit sur $\mathbb{A}$. En effet, il agit sur $A$, on peut donc faire agir $B_0 \simeq \text{Sp}(W) \rtimes W^*$ via $U_0$. Plus précisément $$\begin{array}{ccc}
(B_0,\mathbb{A}) & \longrightarrow & \mathbb{A} \\
((\sigma,\tau),U(w,r)) & \longmapsto & U(\sigma(w), r \varphi_{\sigma,\tau}(w))
\end{array}.$$
Remarquons qu'à travers cette action, on a explicitement l'isomorphisme en question entre $B_0$ et $\text{Aut}_Z(\mathbb{A})$, les automorphismes de $\mathbb{A}$ triviaux sur le centre de $\mathbb{A}$. Ce centre n'est ni plus ni moins que $\{ r \text{Id}_{\mathcal{S}(X)} \ | \ r \in R^\times \} \simeq R^\times$.

On note $\mathbb{B}_0$ le normalisateur de $\mathbb{A}$ dans $\textup{GL}_R(C_c^\infty(X))$, \textit{i.e.} $$\mathbb{B}_0 = \{ s \in \textup{GL}_R(C_c^\infty(X)) \ | \ s \mathbb{A} s^{-1} = \mathbb{A} \}$$
Par conséquent, si $s \in \mathbb{B}_0$, alors la conjugaison par $s$, notée $\text{conj}(s)$, est un élément de $\text{Aut}_Z(\mathbb{A}) \simeq B_0$.

\begin{lem} L'application :
$$\begin{array}{cccc}
\pi_0 & \mathbb{B}_0 & \longrightarrow & B_0 \\
 & s & \longmapsto & \emph{conj}(s)
\end{array}$$ est un morphisme de groupes.
\end{lem}

\begin{theo} On a une suite exacte :
$$1 \to R^\times \to \mathbb{B}_0 \overset{\pi_0}{\to} B_0 \to 1$$
où la première flèche est simplement l'inclusion $R^\times \to \{ r \emph{Id}_{C_c^\infty(X)} \ | \ r \in R^\times \}$. \end{theo}

\subsubsection{Le groupe métaplectique défini dans \cite[2.3.]{ct}}

L'isomorphisme entre $B_0$ et $\text{Sp}(W) \rtimes W^*$ donne une flèche $\mu : \textup{Sp} (W) \to B_0$ en associant simplement à $\sigma \in \textup{Sp} (W)$, l'image de $(\sigma,1) \in \textup{Sp} (W) \rtimes W^*$ dans $B_0$. La seconde flèche est celle obtenue dans la sous-section précédente $\pi_0 : \mathbb{B}_0 \to \mathbb{B}_0$. Il faut se rappeler que $\mathbb{B}_0 \subset \textup{GL}_R(C_c^\infty(X))$, et que $\pi_0$ a pour noyau le centre de $\textup{GL}_R(C_c^\infty(X)))$, on peut donc penser à $B_0$ comme un sous-groupe de $\textup{PGL}_R(C_c^\infty(X))$.

\begin{defi} Le groupe métaplectique de $W$, associé à $R$ et $\psi$, est le sous-groupe $\text{Mp}_{R,\psi}(W)=\textup{Sp}(W) \times_{B_0} \mathbb{B}_0$ de $\textup{Sp}(W) \times \mathbb{B}_0$ qui est composé des paires $(\sigma,s)$ telles que $\mu(\sigma)=\pi_0(s)$.
\end{defi}

On résume la situation dans le diagramme suivant, en se rappelant que $B_0$ est par définition $\textup{Aut}_Z(A)$ :

$$\xymatrix{
	\textup{Mp}_{\psi,R}(W) \ar[r] \ar[d] & \mathbb{B}_0 \ar[rd] & \\
	\textup{Sp}(W) \ar[rd] 		&		&	\textup{Aut}_Z(\mathbb{A}) \ar[d]^\wr \\
						& \textup{Sp}(W) \rtimes W^* \ar[r]^\sim & \textup{Aut}_Z(A)
	}$$
qui peut être réécrit comme 

$$ \xymatrix{
     \textup{Mp}_{\psi,R}(W) \ar[r] \ar[d] & \mathbb{B}_0 \ar[d]^{\pi_0} \\
     \textup{Sp}(W) \ar[r]^\mu & B_0
  }$$

\subsection{Lien avec la construction de \cite{mvw} et la nôtre} \label{construction_parallèle}

Soit $F$ un corps dont la caractéristique est différente de 2, et qui est soit fini, soit local non archimédien. Soit $(W,\langle , \rangle )$ un espace symplectique sur $F$. 

\paragraph{Lien entre $H$ et $A$.} On rappelle que le groupe d'Heisenberg $H = W \times F$ est donné par la loi de groupe :
$$(w,t).(w',t')=(w+w',t+t'+\frac{1}{2} \langle w,w' \rangle ).$$
Le centre du groupe $F$ est $\{(0,t) \ | \ t \in F\}$. Il est isomorphe à $F$ et on se permet de dire que $F$ est le centre de $H$. On se donne une polarisation complète $W =X+Y$ et on écrit $w=x+y$ pour signifier la décomposition de $W$ vis-à-vis de $X$ et $Y$. On peut établir un premier lien entre le groupe d'Heisenberg $H$ défini dans \cite[Chap. 2]{mvw} et le groupe $A$ qui provient de \cite[\S 2.1]{ct}, et dont on a rappelé la construction plus haut. Pour cela, on identifie $X^* \simeq Y$ et on pose $\mathcal{B}(w,w')=\langle x,y' \rangle $. Cette identification dépend donc du choix de la polarisation complète dont on se dote. On prend bien soin de garder la même pour la suite.

\begin{prop} \label{isom_gamma_prop} Le morphisme :
$$\begin{array}{cccc}
\Gamma : & H & \to & A \\
& (w,t) & \mapsto & (w, \psi(t) \psi( \frac{1}{2} \langle x,y \rangle ))
\end{array}$$
induit un isomorphisme $\overline{\Gamma}$ entre $H/\emph{Ker}(\psi)$ et le sous-groupe $W \rtimes \emph{Im}(\psi)$ de $A$. Les centres de ces deux groupes sont respectivement $F/\emph{Ker}(\psi)$ et $\emph{Im}(\psi)$, qui sont bien isomorphes.

De plus, la conjugaison par $\overline{\Gamma}$ induit un isomorphisme :
$$ s \in B_0 = \emph{Aut}_Z(A) \overset{\sim}{\longrightarrow} \overline{\Gamma}^{-1} \circ s \circ \overline{\Gamma} \in \emph{Aut}_Z(H/\emph{Ker}(\psi)).$$
\end{prop}

\begin{proof} Le fait que $\Gamma$ soit un un morphisme de groupes provient de :
$$\mathcal{B}(w,w') + \frac{1}{2} \langle x,y \rangle  + \frac{1}{2} \langle x',y' \rangle  = \frac{1}{2}  \langle w,w' \rangle  + \frac{1}{2} \langle x+x',y+y' \rangle .$$
L'image de $\Gamma$ est $W \times \text{Im}(\psi)$ et son noyau s'identifie à $\textup{Ker} (\psi)$. Trouver le centre de chacun des groupes en jeu est immédiat.

Ensuite, il suffit de remarquer que $\textup{Aut}_Z(A) = \textup{Aut}_Z(W \times \text{Im}(\psi)) \simeq \textup{Aut}_Z(H/\textup{Ker}(\psi))$. Seule la première égalité est à prouver, la deuxième résultant de la conjugaison par $ \overline{\Gamma}$. Tout élément de $B_0$ est de la forme $s=(\sigma, (\psi \circ f_\sigma ) \ (\psi \circ \gamma))$ avec $\gamma \in W^*$, et tel que $s(w,r)=(\sigma(w),(\psi \circ f(w)) \ r)$ où $f= f_\sigma + \gamma$. Donc $s$ induit par restriction un automorphisme de $W \times \text{Im}(\psi)$, qui est bien l'identité sur le centre $\text{Im}(\psi) \subset R^\times$. Réciproquement, tout élément $s'$ de $\textup{Aut}_Z(W \times \text{Im}(\psi))$ s'écrit sous la même forme qu'un élément de $B_0$. En effet, soit $s'=(\sigma',\phi')$ où $\sigma' : W \times \text{Im}(\psi) \to W$ et $\phi' :  W \times \text{Im}(\psi) \to \text{Im}(\psi)$. De la loi de groupe sur $A = W \times R$ induite sur $W \times \text{Im}(\psi)$, on tire comme dans le cas de $B_0$ \cite[Prop. 2.1]{ct} des relations qui entraînent que : $\sigma'(w,r)=\sigma'(w)$ et $\sigma' \in \text{Sp}(W)$ ; il existe $\gamma \in W^*$ tel que $\phi'(w,r) = (\psi \circ f_\sigma(w))(\psi \circ \gamma(w)) r$. Alors l'élément $s=(\sigma',(\psi \circ f_\sigma) (\psi \circ \gamma)) \in B_0$ relève $s'$ à $A$. \end{proof}

\paragraph{Le groupe $B_1$.} On note $\textup{Aut}(H)$ le groupe des automorphismes du groupe $H$, $\textup{Aut}_Z(H)$ le sous-groupe des automorphismes triviaux sur le centre de $H$. On décrit explicitement les éléments de $\textup{Aut}_Z(H)$ :

\begin{prop} Soit $s \in \emph{Aut}(H)$, on note $\varphi_s$ et $\tau_s$ les applications associées telles que $s((w,t))= (\varphi_s ((w,t)),\tau_s(w,t))$. On identifie alors $\emph{Aut}_Z(H)$ à $ \emph{Sp}(W) \rtimes \emph{Hom}(W,F)$ via l'isomorphisme :
$$s \mapsto (\sigma,\gamma)=\varphi_s|_{(W,0)}, \tau_s|_{(W,0)}).$$
Le produit étant défini par $(\sigma,\gamma).(\sigma',\gamma')=(\sigma \circ \sigma', \gamma \circ \sigma' + \gamma')$.
\end{prop}

\begin{proof}  Plusieurs choses sont à vérifier. Dans l'ordre : on donne une description des relations portant sur $\varphi_s$ et $\tau_s$ ; on montre que l'application du théorème est bien un morphisme ; et enfin, que $\sigma \in \textup{Sp}(W)$ car elle est bien $F$-linéaire.

On sait que les éléments de $\textup{Aut}_Z(H)$ sont triviaux sur le centre. On a tout d'abord $s((w,t))=s((0,t).(w,0))=(0,t).s((w,0))$. Il suffit donc de connaître les relations portant sur $s((w,0))$. Rappelons que pour tout $t \in F$, $s((0,t))=(0,t)$ \textit{i.e.} $\varphi_s((0,t))=0$ et $\tau_s((0,t))=t$. À l'aide du calcul explicite :
$$s((w,0)).s((w',0))=s((w,0).(w',0))=s((w+w',\frac{1}{2} \langle w,w' \rangle ))$$
on tire :
$$\varphi_s((w,0))+\varphi_s((w',0))=\varphi_s((w+w',0)) \textup{ et } \langle \varphi_s((w,0)),\varphi_s((w',0)) \rangle =\langle w,w' \rangle.$$

On obtient finalement $\varphi_s((w,t))=\varphi_s((w,0))=\sigma(w)$ et $\tau_s((w,t))=t$. On note $\tau = \tau_s$ qui est indépendant de $s \in \textup{Aut}_Z(H)$. La composée $s \circ s'$ de morphismes de $s, s' \in \textup{Aut}_Z(H)$ donne $\varphi_{s \circ s'}((w,t))= \varphi_s((\varphi_{s'}(w),0))= \sigma \circ \sigma' (w)$. Et $\tau$ est inchangé. De plus, comme $s$ est inversible, il existe un morphisme de groupe $\sigma' : W \to W$ tel que $\sigma \circ \sigma' = \textup{Id}_W$, donc $\sigma$ est un automorphisme.

Il ne reste plus qu'à prouver que $\sigma$ est $F$-linéaire. On déduit des équations précédentes que pour tout $\lambda \in F$, on a $\langle \sigma(\lambda w ) - \lambda \sigma(w),\sigma(w') \rangle =0$. Comme $\sigma$ est bijective, on conclut par le fait que la forme est non dégénérée sur $W$. Donc $\sigma \in \textup{Sp}(W)$.

En résumé, pour tout $s \in \textup{Aut}_Z(H)$, on a $s((w,t))=(\sigma(w),t)$ avec $\sigma \in \textup{Sp}(W)$. Réciproquement, toute application de cette forme est bien dans $\textup{Aut}_Z(H)$. \end{proof}

On s'intéresse ici à  un sous-groupe particulier de $\textup{Aut}_Z(H)$. On note $W^*$ le dual algébrique de $W$.

\begin{defi} On définit donc $B_1$ comme le sous-groupe $\textup{Sp}(W)\rtimes W^*$ dans $\textup{Aut}_Z(H)$. \end{defi}

\begin{prop} Le morphisme de groupes :
$$\begin{array}{ccc}
B_1 & \longrightarrow & B_0 \\
(\sigma,\gamma) & \longmapsto & (\sigma, (\psi \circ \gamma) \ (\psi \circ f_\sigma))
\end{array}$$
est un isomorphisme. \end{prop}

\begin{proof} C'est bien un morphisme de groupes étant donné la formule de cocyle $f_{\sigma \circ \sigma'}=f_\sigma \circ \sigma' + f_{\sigma'}$. L'injectivité et la surjectivité se vérifient aisément en se rappelant que le dual algébrique est isomorphe aux caractères de $W$.
\end{proof}

\paragraph{Isomorphisme entre produits fibrés.} Le théorème principal que l'on se propose de montrer concerne les groupes métaplectiques définis dans \cite{weil} et \cite{mvw}. Donnons tout d'abord une première version dans laquelle le lien avec le dernier de ces groupes est absent. Ici, $\mathbb{B}_1$ est un sous-groupe $\textup{GL}_R(C_c^\infty(X))$ analogue au $\mathbb{B}_0$ dans la construction de \cite{weil}, \cite{ct}.

\begin{theo} Le produit fibré $\emph{Sp}(W) \times_{B_1} \mathbb{B}_1$ est isomorphe à $\emph{Sp}(W) \times_{B_0} \mathbb{B}_0$ , l'isomorphisme entre $B_1$ et $B_0$ étant compatible avec les flèches respectives.
\end{theo}

\begin{proof} La preuve va se dérouler en deux parties. Les deux diagrammes ci-dessous seront une bonne aide à la lecture, même si les objets auxquels ils se rapportent ne seront définis que par la suite.

La première est immédiate et consiste à prouver que le diagramme suivant commute 
$$\xymatrix{
		\textup{Sp}(W) \ar[r]^\sim \ar[d] & \textup{Sp}(W) \ar[d] \\
		B_1 \ar[r]^\sim & B_0
		}$$
Les flèches respectives étant : l'identité de $\textup{Sp}(W)$ ; les plongements canoniques de $\textup{Sp}(W)$ dans les produits semi-directs $B_1$ et $B_0$ ; l'isomorphisme de la proposition précédente.

Ensuite, dans un deuxième temps, il s'agit de s'occuper de la commutativité et des isomorphismes du diagramme
$$\xymatrix{
		\mathbb{B}_1 \ar[r]^\sim \ar[d] & \mathbb{B}_0 \ar[d] \\
		\textup{Aut}_Z(\mathbb{H}) \ar[r]^\sim \ar[d]^\wr & \textup{Aut}_Z(\mathbb{A}) \ar[d]^\wr \\
		B_1 	\ar[r]^\sim	& B_0
		}$$

\begin{lem} Soit $(w,t) \in H$. On définit une application $U_1 : H \to \textup{GL}_R(C_c^\infty(X))$ telle que, pour tout $\phi \in C_c^\infty(X), \ U_1(w,t) \phi (u)= \psi(t) \psi(\frac{1}{2} \langle x,y \rangle )\psi(\langle u,y \rangle )\phi(u+x)$. Alors, $U_1$ est un morphisme. Il induit un isomorphisme entre $H/\emph{Ker}(\psi)$ est son image, notée $\mathbb{H}$. Le diagramme qui suit est commutatif et les flèches horizontales sont des isomorphismes 
$$\xymatrix{
	H/\emph{Ker}(\psi) \ar[r]^{\overline{U_1}} \ar@{^{(}->}[d]_{\overline{\Gamma}} & \mathbb{H} \ar@{^{(}->}[d] \\
	A \ar[r]^{U_0} & \mathbb{A}
	 }$$
De plus, la flèche verticale de droite est l'inclusion de $\mathbb{H}$ dans $\mathbb{A}$ vus comme sous-groupes de $\textup{GL}_R(C_c^\infty(X))$. L'action de $B_0$ sur $A$ définie une action de $B_0$ sur $H/\emph{Ker}(\psi)$ et sur $\mathbb{H}$. \end{lem}

\begin{proof} Il suffit de noter que $U_1=U_0 \circ \Gamma$, où $\Gamma$ et $\overline{\Gamma}$ sont les morphisme de la Proposition \ref{isom_gamma_prop}. En particulier $\mathbb{H} \simeq H / \textup{Ker}(\psi)$. \end{proof}

Ceci va permettre d'obtenir le carré bas du diagramme. Tout d'abord, les isomorphismes $\overline{U_1}$ et $U_0$ induisent des isomorphismes $\textup{Aut}_Z(\mathbb{H})\simeq \textup{Aut}_Z(H/\textup{Ker}(\psi))$ ainsi que $\textup{Aut}_Z(\mathbb{A}) \simeq \textup{Aut}_Z(A)$. De plus, grâce à la proposition \ref{isom_gamma_prop}, $\textup{Aut}_Z(H/\textup{Ker}(\psi)) \simeq \textup{Aut}_Z(A)$. Finalement, on obtient 
$$\xymatrix{
	\textup{Aut}_Z(\mathbb{H}) \ar[r]^\sim \ar[d]_\sim & \textup{Aut}_Z(\mathbb{A}) \ar[d]^\sim \\
	B_1 \ar[r]^\sim & B_0
	}$$
	
	La flèche $\textup{Aut}_Z(\mathbb{H}) \to \textup{Aut}_Z(\mathbb{A})$ est l'inverse de la restriction à $\mathbb{H}$, \textit{i.e.} l'inverse de $f \in \textup{Aut}_Z(\mathbb{A})  \to f|_{\mathbb{H}} \in \textup{Aut}_Z(\mathbb{H})$.

Pour rappel, les isomorphismes entre ces groupes sont les suivants :
$$U_0 \circ \textup{Aut}_Z(A) \circ U_0^{-1} = \textup{Aut}_Z(\mathbb{A}) ;$$
$$\overline{U_1} \circ \textup{Aut}_Z(h) \circ \overline{U_0}^{-1} = \textup{Aut}_Z(\mathbb{H}) ;$$
$$\overline{\gamma}^{-1} \circ \textup{Aut}_Z(h) \circ \overline{\Gamma} = \textup{Aut}_Z(H/\textup{Ker}(\psi)).$$

Pour finir, on s'occupe du carré du haut. Par construction, $\mathbb{B}_0$ et $\mathbb{B}_1$ sont les normalisateurs de $\mathbb{A}$ et $\mathbb{H}$ dans $\textup{GL}_R(C_c^\infty(X))$. Ces normalisateurs sont les mêmes car $\mathbb{A}= R^\times \mathbb{H}$ et $R^\times$ est dans le centre de $\textup{GL}_R(C_c^\infty(X))$.
\end{proof}

\section{Représentations d'un produit de groupes} \label{representations_un_produit_groupes_section}

Dans toute cette partie, $R$ sera un corps. Soient $G_1$ et $G_2$ deux groupes localement profinis. Pour des représentations $(\pi_1,V_1) \in \textup{Rep}_R(G_1)$ et $(\pi_2,V_2) \in \textup{Rep}_R(G_1)$, on pose $D_1 = \textup{End}_{G_1}(\pi_1)$ et $D_2 = \textup{End}_{G_2}(\pi_2)$. La représentation $(\pi_1,V_1)$ est un $D_1$-module à gauche. Cette structure de module commute à la structure de $R[H_1]$-module à gauche, c'est-à-dire $(\pi_1,V_1)$ est un $(D_1,R[H_1])$-bimodule à gauche. Il en va de même pour $(\pi_2,V_2)$ qui est un $(D_2,R[H_2])$-bimodule à gauche.

\subsection{Produit tensoriel et extension des scalaires}

On suppose que les représentations $(\pi_1,V_1)$ et $(\pi_2,V_2)$ sont irréductibles. En particulier, $D_1$ et $D_2$ sont des corps par le lemme de Schur.

\paragraph{Produit tensoriel.} On applique \cite[A VIII.210, Th. 2]{bou} pour obtenir :

\begin{lem} \label{produit_tensoriel_de_rep_irred_lem} Soient $V=V_1 \otimes_R V_2 \in \textup{Rep}_R(G_1 \times G_2)$ et $D =\textup{End}_{G_1 \times G_2}(V_1 \otimes_R V_2)$. Alors on a : 
$$D  \simeq \textup{End}_{G_1}(\pi_1) \otimes_R \textup{End}_{G_2}(\pi_2).$$ De plus, l'ensemble $\mathbb{V}$ des sous-représentations de $V$ est en bijection avec l'ensemble $\mathbb{D}$ des sous-$D$-modules à droite de $D$, grâce à l'application suivante qui est compatible à la relation d'inclusion :
$$\begin{array}{ccc}
\mathbb{D} & \to & \mathbb{V} \\
D' & \mapsto & D' V
\end{array}.$$ \end{lem}

\paragraph{Extension des scalaires dans les corps parfaits.} On rappelle la définition de l'action d'un élément de Galois sur une représentation donnée. Soient $R'$ une extension algébrique de $R$ et $w : R' \to R'$ un automorphisme de $R$-algèbres. En particulier, la restriction de $w$ à $R$ est l'identité. Pour toute représentation $(\rho_1,V_1') \in \textup{Rep}_{\bar{R}} (G_1)$, en choisissant une $R'$-base $(e_i)_I$ de $V_1'$, on définit $(w \rho_1,V_1') \in \textup{Rep}_R'( G_1 )$ par :
$$[w \rho_1(g_1)]_{i,j} = w([\rho_1(g_1)]_{i,j}) \textup{ où } i, j \in I \textup{ et } g_1 \in G_1.$$
Cette construction ne dépend pas du choix de la base $(e_i)_I$ au sens où la classe d'isomorphisme de telles représentations $(w \rho_1, V_1')$ est bien définie \cite[Chap. II, 4.1.a]{vig}. Par conséquent, cela a du sens de parler de $w \rho_1$ sans mentionner de choix de base.

Grâce au Lemme \ref{produit_tensoriel_de_rep_irred_lem} précédent, on va généraliser \cite[Chap. II, 4.4]{vig}. On fixe désormais une clôture algébrique $\bar{R}$ de $R$ et on rappelle quelques notions développées dans \cite[Chap. II, 4]{vig}.

Le corps de rationalité $E(\rho_1)$ d'une représentation $\rho_1$ dans $\textup{Rep}_{\bar{R}}(G_1)$ est défini comme le corps des invariants de $\bar{R}$ sous l'action de $H(\rho_1)=\{ w \in \textup{Gal}_R(\bar{R}) \ | \ w \rho_1 \simeq \rho_1 \}$ :
$$E(\rho_1) = (\bar{R})^{H(\rho_1)}.$$
En particulier, l'action du groupe de Galois $\textup{Gal}_R(\bar{R})$ sur les classes d'isomorphismes définies par $\rho_1$ se factorise par l'action du groupe $\textup{Gal}_R(E(\rho_1),\bar{R})$, qui est défini comme l'ensemble des plongements $R$-linéaires de $E(\rho_1)$ dans $\bar{R}$.

Un sous-corps $E$ de $\bar{R}$ est un corps de définition $E$ pour $\rho_1$ s'il existe une représentation $\tau_1 \in \textup{Rep}_E(G_1)$ telle que $\tau_1 \otimes_E \bar{R} \simeq \rho_1$. On dit alors que $\rho_1$ est réalisable sur $E$. Quand le corps $R$ est parfait, l'extension algébrique $\bar{R}/R$ est galoisienne, et on sait \cite[Chap. II, 4.1.c]{vig} qu'un corps de définition de $\rho_1$ contient toujours son corps de rationalité $E(\rho_1)$. La réciproque n'est pas vraie en toute généralité et échoue déjà pour les groupes finis.

\begin{theo} \label{decomposition_extension_des_scalaires_thm} On suppose que $R$ est un corps parfait et que $\pi_1$ est admissible. On adopte les notations suivantes :
\begin{itemize}[label=$\bullet$]
\item $E_1$ est le centre de $D_1$ avec $n_1 = [E_1 : R]$ ;
\item $m_1$ est le degré de cette algèbre à division $D_1$ sur son centre $E_1$ ;
\item $\textup{Gal}_R(E_1,\bar{R})$ est l'ensemble des plongements $R$-linéaires de $E_1$ dans $\bar{R}$.
\end{itemize}
Alors pour tout facteur irréductible $\rho_1 \in \textup{Rep}_{\bar{R}}(G_1)$ de $\pi_1 \otimes_R \bar{R}$, il existe $w_1 \in \textup{Gal}_R(E_1,\bar{R})$ tel que $E(\rho_1) = w_1(E_1)$. Le choix d'un tel $\rho_1$ et $w_1$ induit un isomorphisme de représentations dans $\textup{Rep}_{\bar{R}}(G_1)$ :
$$\pi_1 \otimes_R \bar{R} \simeq m_1 \ \bigg( \bigoplus_{w \in \textup{Gal}_R(E(\rho_1),\bar{R})} w \rho_1 \bigg).$$
Pour terminer, en considérant $\pi_1$ comme une représentation dans $\textup{Rep}_{E_1}(G_1)$ via l'action du centre $E_1$ de son anneau des endomorphismes $D_1$, le plongement $w w_1 : E_1 \to \bar{R}$ pour $w \in  \textup{Gal}_R(E(\rho_1),\bar{R})$ induit un isomorphisme de représentations dans $\textup{Rep}_{\bar{R}}(G_1)$ :
$$\pi_1 \otimes_{E_1} \bar{R} \simeq m_1( w \rho_1).$$ \end{theo}

\begin{proof} On va montrer qu'il existe une extension finie $R'$ de $R$ dans $\bar{R}$ pour laquelle ce résultat est déjà vrai. Le but est de prouver que la représentation $\pi_1 \otimes_R R'$ est somme de représentations absolument irréductibles qui ont même multiplicité et s'obtiennent comme des conjuguées l'une de l'autre. Tout d'abord, $D_1$ est une algèbre à division de centre $E_1$, qui est de dimension finie sur $E_1$ car $\pi_1$ est admissible. Il existe donc une extension (séparable) $E_1'$ de $E_1$ de degré $m_1$ telle que $D_1 \otimes_{E_1} E_1' \simeq M_{m}(E_1')$.

En prenant la clôture galoisienne $R'$ de toute image de $E_1'$ dans $\bar{R}$, on peut voir $R'$ comme une représentation irréductible du groupe $G_2 = \mathbb{Z}$ muni de la topologie discrète. En effet, grâce au théorème de l'élément primitif, il existe $\beta \in R'$ non nul de polynôme minimal $P$ tel que $R' = R[\beta] \simeq R[X]/(P(X))$. La représentation $(\pi_2,R') \in \textup{Rep}_R(\mathbb{Z})$ définie par $\pi_2( 1 ) = \beta \in \textup{GL}_R(R')$ est bien irréductible.

D'après le Lemme \ref{produit_tensoriel_de_rep_irred_lem}, les sous-représentations de $\pi_1 \otimes_R \pi_2 = \pi_1 \otimes_R R'$ correspondent aux sous-$D$-modules à droite sur $D = \textup{End}_{G_1}(\pi_1) \otimes_R R'$. Or :
$$D \simeq \prod_{1 \leq k \leq n_1} M_{m_1}(R').$$
D'après \cite[Chap. II, 4.2]{vig}, la représentation $\pi_1 \otimes_R R'$ est semi-simple. Étant donné que $D = \textup{End}_{G_1}(\pi_1 \otimes_R R')$ d'après le Lemme \ref{produit_tensoriel_de_rep_irred_lem}, il existe des représentations irréductibles $(\tau_k)$ dans $\textup{Rep}_{R'}(G_1)$ qui sont deux à deux non isomorphes et telles que :
$$\pi_1 \otimes_R R' \simeq \bigoplus_{1 \leq k \leq n_1} (m_1 \tau_k).$$
De plus $\textup{End}_{G_1}(\tau_k) = R'$, donc chaque représentation $\rho_k$ est absolument irréductible. Par conséquent $\rho_k = \tau_k \otimes_{R'} \bar{R}$ est irréductible. Ainsi $\pi_1 \otimes_R \bar{R} = m_1  (\bigoplus_{1 \leq k \leq n_1} \rho_k)$.

Il reste à montrer que $\textup{Gal}_R(E_1,\bar{R})$ agit simplement transitivement sur les $n_1$ classes d'isomorphismes définies par la famille $(\rho_k)$.  Tout d'abord $\textup{Gal}_R(E_1,\bar{R})=\textup{Gal}_R(E_1,R')$ est de cardinal $n_1$ puisque $R$ est parfait. Donc on peut considérer de manière équivalente, au lieu de $\bar{R}$ et $(\rho_k)$, le corps $R'$ et la famille $(\tau_k)$. Pour chaque $w \in \textup{Gal}_R(E_1,R')$, on va expliciter le facteur $M_{m_1}(R')$ correspondant dans la décomposition de $D$. Le centre de $D$ est $E_1 \otimes_R R'$. En utilisant le théorème de l'élément primitif, il existe un polynôme unitaire $Q \in R[X]$ de degré $n_1$ et une racine $\alpha \in E_1$ de celui-ci tels qu'on ait :
$$E_1 =R[\alpha] \simeq R[X] / (Q(X))$$
De plus, le polynôme $Q$ est scindé à racines simples dans $R'[X]$ car l'extension $R'$ de $R$ est galoisienne et contient une racine de $Q$.

Ainsi on a un isomorphisme d'anneaux :
$$E_1 \otimes_R R' \simeq \prod_{w \in \textup{Gal}_R(E_1,R')} R'[X] / (X-w(\alpha))$$
auquel est associé une unique décomposition $\sum_{w \in \textup{Gal}_R(E_1,R')} e_w = 1_{E_1 \otimes_R R'}$ de l'unité en somme d'idempotents, ici notée $(e_w)_w$. Pour $w \in \textup{Gal}_R(E_1,R')$, l'idempotent $e_w$ vérifie :
$$e_w (\alpha \otimes_R 1_{R'} - 1_R \otimes_R w(\alpha)) = 0 .$$
Avec ces notations $e_w$ est un idempotent central de $D$, et on obtient explicitement l'isomorphisme :
$$D = \prod_{w \in \textup{Gal}_R(E_1,R')} e_w D \simeq \prod_{1 \leq k \leq n_1} M_{m_1}(R')$$
où le dernier isomorphisme dépend du choix d'une base.

Enfin, l'action de $\textup{Gal}_R(R')$ laisse invariante la représentation $\pi_1 \otimes_R \bar{R}$ au sens où pour tout $w \in \textup{Gal}_R(R')$, on a $w(\pi_1 \otimes_R R') \simeq \pi_1 \otimes_R R'$. Par conséquent, il existe un indice $k$ tel que $w \tau_1 \simeq \tau_k$. Par définition du corps de rationalité,  $w \tau_1 \simeq \tau_1$ si et seulement si $w|_{E(\tau_1)}=\textup{Id}_{E(\tau_1)}$. Et l'orbite de $\tau_1$ est de cardinal $| \textup{Gal}_R(E(\tau_1),R')|$.

On va montrer qu'il existe $w_1 \in \textup{Gal}_R(E_1,R')$ tel que $E(\tau_1)=w_1(E_1)$, de sorte qu'on aura que le cardinal de l'orbite de $\tau_1$ est $n_1$. Or :
$$e_w(\pi_1 \otimes_R R') \simeq \pi_1 \otimes_{E_1} R'$$
où $\pi_1$ est considérée à droite comme une représentation dans $\textup{Rep}_{E_1}(G_1)$ et le produit tensoriel est pris pour $w : E_1 \to R'$ avec $w \in \textup{Gal}_R(E_1,R')$. Les sous-représentations de $e_w(\pi_1 \otimes_R R')$ étant en bijection avec les sous-$e_w D$-modules à droite de $e_w D \simeq M_{m_1}(R')$, on en déduit que $\pi_1 \otimes_{E_1} R'$ est isotypique. Donc il existe $w_1 \in \textup{Gal}_R(E_1,R')$ tel que $e_{w_1} (\pi_1 \otimes_R R') \simeq m_1 \tau_1 \simeq \pi_1 \otimes_{E_1} R'$. Ainsi $E(m_1 \tau_1) = E(\tau_1) = E(\pi_1 \otimes_{E_1} R')$. Comme $E(\pi_1 \otimes_{E_1} R') \subset w_1(E_1)$, on a une première inclusion. Pour l'inclusion réciproque, l'image par $w \in \textup{Gal}_R(R')$ de $e_{w_1}(\pi_1 \otimes_R R')$ est isomorphe à $e_{w w_1}(\pi_1 \otimes_R R')$. De plus, $\{ w w_1 \ | \ \textup{Gal}_R(R') \} = \textup{Gal}_R(E_1,R')$, donc l'action est transitive puisque tout $m_1 \tau_k$ est obtenu comme un $e_{w w_1}(\pi_1 \otimes_R R')$. Ainsi pour tout $w \in \textup{Gal}_R(R')$ tel que $w|_{w_1(E_1)} \neq \textup{Id}_{w_1(E_1)}$, on a $w \tau_1 \not \simeq \tau_1$ \textit{i.e.} $w_1(E_1) \subset E(\tau_1)$. Donc $w_1(E_1) = E(\tau_1)$ et l'orbite de $\tau_1$ sous l'action de $\textup{Gal}_R(E(\tau_1),R')$ est de cardinal $n_1$ \textit{i.e.} est la famille $(\tau_k)$. \end{proof}

\begin{rem} \label{decomposition_extension_des_scalaires_rem} Bien que cela ne soit pas rendu explicite ici, il est possible de décrire le comportement de l'extension des scalaires quand $R$ n'est pas parfait. Cela fait apparaître des extensions non triviales entre d'une représentation irréductible par elle-même. On a toujours une décomposition du type $\pi_1 \otimes_R \bar{R} = \sum_w \pi_1 \otimes_{E_1} R'$ où la somme porte sur les plongements $\textup{Gal}_R(E_1,\bar{R})$. En revanche, l'anneau $E_1 \otimes_R \bar{R}$ n'est plus nécessairement un produit de corps car il n'est pas réduit si $E_1/R$ n'est pas séparable. Le Lemme \ref{produit_tensoriel_de_rep_irred_lem} étant encore valable, la décomposition sera donnée par les sous-$D_1 \otimes_R \bar{R}$-modules à droite de $D_1 \otimes_R \bar{R}$ où :
$$D_1 \otimes_R \bar{R} \simeq M_{m_1} (E_1 \otimes \bar{R}).$$
Par exemple si $R= \mathbb{F}_2(t)$, la représentation $(\pi_1,R^2) \in \textup{Rep}_R(\mathbb{Z})$ avec $\pi_1(1) =  \left[ \begin{array}{cc} 0 & t \\
1 & 0 \end{array} \right]$ est irréductible. On a $D_1 \simeq \mathbb{F}_2(\sqrt{t})$. Mais pour l'extension algébrique $R'=\mathbb{F}_2(\sqrt{t})$, on obtient $D_1 \otimes_R R' \simeq R'[\varepsilon]$ où $\varepsilon^2=0$. Or $R'[\varepsilon]$ possède trois idéaux $(0) \subset (\varepsilon) \subset (1)$, donc la représentation $\pi_1 \otimes_R R'$ est de longueur $2$ et n'est pas semi-simple. Sa semi-simplifiée est $2 \chi$ où $\chi : 1 \in \mathbb{Z} \mapsto \sqrt{t} \in R'$. \end{rem}

\begin{lem} \label{unicite_restriction_des_scalaires_induction_lem} On suppose que $R$ est un corps parfait. Soit $\rho_1 \in \textup{Rep}_{\bar{R}}(G_1)$ une représentation irréductible admissible réalisable sur une extension finie de $R$. Il existe alors une représentation irréductible admissible $\pi_1 \in \textup{Rep}_R(G_1)$ unique à isomorphisme près telle que $\pi_1 \otimes_R \bar{R}$ contienne $\rho_1$ comme sous-quotient. On la note $\pi_1(\rho_1,R)$. De plus, si $E$ un corps de définition de $\rho_1$ de degré minimal tel que $\tau_E \otimes_E \bar{R} = \rho_1$, on a en notant $\textup{Res}^R$ la restriction des scalaires à $R$ :
$$\pi_1(\rho_1,R) \simeq \textup{Res}^R(\tau_E).$$ \end{lem} 

\begin{proof} On commence tout d'abord par l'unicité. Pour tout $\pi_1 \in \textup{Rep}_R(G_1)$ irréductible, la représentation $\textup{Res}^R(\pi_1 \otimes_R \bar{R})$ est $\pi_1$-isotypique. Ainsi, pour tout sous-quotient $W$ de $\pi_1 \otimes_R \bar{R}$, la représentation $\textup{Res}^R(W)$ est $\pi_1$-isotypiques. C'est en particulier vrai si le sous-quotient en question est $\rho_1$. Comme $\textup{Res}^R(\rho_1)$ est isotypique, cela entraîne l'unicité d'une telle représentation $\pi_1$ quand elle existe.

On réalise ensuite $\rho_1$ sur une extension finie $E$ de $R$ avec $\tau_E \otimes_E \bar{R} = \rho_1$. Il est clair que $\tau_E$ est admissible puisque $\rho_1$ l'est. Cela entraîne que $\textup{Res}^R(\tau_E)$ est admissible. Comme $\textup{Res}^R(\tau_E)$ est de type finie, elle admet un quotient irréductible $\pi_1$ \textit{i.e.} un morphisme non nul $f : \textup{Res}^R(\tau_E) \to \pi_1$. Soit $\phi : R[G_1] \to E[G_1]$ l'inclusion évidente. Le foncteur d'oubli $\textup{Rep}_E(G_1) \to \textup{Rep}_R(G_1)$ admet un adjoint à droite qui est $\textup{Hom}_{R[G_1]}(E[G_1],-)$. Donc il correspond $f$ par adjonction un morphisme non nul $f' : \tau_E \to \textup{Hom}_{R[G_1]}(E[G_1],\pi_1)$ donc $f'$ est injective. Ainsi, $\textup{Res}^R(\tau_E)$ est une sous-représentation de :
$$\textup{Res}^R(\textup{Hom}_{R[G_1]}(E[G_1],\pi_1))= \bigoplus_{[E:R]} \pi_1.$$
Donc $\textup{Res}^R(\tau_E)$ est $\pi_1$-isotypique de multiplicité finie. On en déduit que $\pi_1$ est admissible puisque $\textup{Res}^R(\tau_E)$ l'est et que le foncteur des invariants pour les sous-groupes compacts est exact à gauche. Donc $\pi_1 = \pi_1(\rho_1,R)$ existe et est admissible.

Il reste enfin à montrer l'assertion $\pi_1(\rho_1,R) \simeq \textup{Res}^R(\tau_E)$ si l'extension $E/R$ est de degré minimal. On note simplement $\pi_1$ pour abréger $\pi_1(\rho_1,R)$. On a déjà montré au paragraphe précédent que $\textup{Res}^R(\tau_E)$ est $\pi_1$-isotypique de longueur finie. Il suffit par conséquent de montrer qu'elle est irréductible. On peut appliquer le Théorème \ref{decomposition_extension_des_scalaires_thm} précédent qui assure qu'il existe un sous-corps $E_1'$ de $E$ isomorphe au centre $E_1$ de $D_1 = \textup{End}_{G_1}(\pi_1(\rho_1,R))$ tel que $E/E_1'$ soit de degré $m_1$. On tire enfin des égalités :
$$\pi_1 \otimes_{E_1} E \simeq m_1 \tau_E \textup{ et } \textup{Res}^R(\pi_1 \otimes_{E_1} E) = \oplus_{[E:E_1']} \pi_1 = m_1 \pi_1$$
que $m_1 \textup{Res}^R(\tau_E) \simeq m_1 \pi_1$ \textit{i.e.} $\textup{Res}^R(\tau_E) \simeq \pi_1$. \end{proof}

\begin{rem} Conformément à la Remarque \ref{decomposition_extension_des_scalaires_rem}, on peut développer ce résultat quand $R$ n'est pas parfait. Il n'y a besoin d'aucune modification \textit{a priori}, excepté qu'au lieu des facteurs $m_1 \tau_E$ il faut considérer un facteur $A_{\tau_E}$ de longueur $m_1$, qui peut contenir des extensions non triviales de $\tau_E$ par elle-même. En revanche, il est tout à fait possible que $A_{\tau_E} \simeq \pi_1 \otimes_{E_1} E$ ne soit pas semi-simple alors que $\textup{Res}^R(A_{\tau_E})$ l'est toujours. \end{rem}

\begin{cor} \label{rep_irred_gp_red_sont_adm_cor} Soit $G$ un groupe réductif (connexe) sur un corps local non archimédien. On suppose que $R$ est un corps parfait et que $G$ contient un sous-groupe compact ouvert de pro-ordre inversible dans $R$. Alors toute représentation irréductible dans $\textup{Rep}_R(G)$ est admissible. \end{cor}

\begin{proof} On applique le Lemme \ref{unicite_restriction_des_scalaires_induction_lem} précédent. D'une part, toute représentation irréductible $\rho_1$ dans $\textup{Rep}_{\bar{R}}(G)$ est admissible. Ce fait bien connu \cite[Chap. II, 2.8]{vig} provient de ce que cela est vrai pour les représentations cuspidales. D'autre part, toute représentation irréductible $\rho_1$ dans $\textup{Rep}_{\bar{R}}(G)$ peut se réaliser sur une extension finie de $R$ d'après \cite[Chap. II, 4.7]{vig}.

Si le groupe réductif en question n'est pas connexe, on sait qu'il admet un quotient fini par sa partie connexe. On peut résonner de manière similaire sur les représentations cuspidales de ses sous-groupes de Levi. \end{proof}

\begin{rem} Toujours dans le même esprit que les remarques précédentes, ce résultat semble aussi valable quand $R$ n'est pas parfait. \end{rem}

\subsection{Plus grand quotient $\pi_1$-isotypique} \label{decomposition_rep_prod_groupes_section}

On suppose dans cette section qu'il existe un sous-groupe compact ouvert de $G_1 \times G_2$ de pro-ordre inversible dans $R$, ce qui assure l'existence d'une mesure de Haar sur $G_1 \times G_2$ à valeurs dans $R$. En particulier, il existe des mesures de Haar sur $G_1$ et $G_2$ à valeurs dans $R$. Le théorème suivant généralise les résultats de \cite{flath} pour les représentations à coefficient complexes $R = \mathbb{C}$, dont \cite[Th. A.4]{vig_invent} est déjà une généralisation au cas $R$ algébriquement clos.

\begin{theo} \label{flath_theorem} On suppose que $R$ est un corps parfait.
\begin{enumerate}[label=\textup{\alph*)}]
\item Si $\pi_1 \in \emph{Rep}_R(G_1)$ et $\pi_2 \in \emph{Rep}_R(G_2)$ sont deux représentations irréductibles admissibles, alors $\pi_1 \otimes_R \pi_2 \in \textup{Rep}_R(G_1 \times G_2)$ est une représentation semi-simple admissible.
\item Si $\pi$ est une représentation irréductible admissible dans $\textup{Rep}_R(G_1 \times G_2)$, alors  il existe des représentations irréductibles admissibles $\pi_1$ de $G_1$ et $\pi_2$ de $G_2$ telles que $\pi$ soit un quotient de $\pi_1 \otimes_R \pi_2$. De plus, $\pi_1$ et $\pi_2$ sont uniques à isomorphisme près c'est-à-dire $\pi$ détermine ces classes d'isomorphisme. \end{enumerate} \end{theo}

\begin{proof} a) Grâce au Lemme \ref{produit_tensoriel_de_rep_irred_lem}, il suffit de prouver que la $R$-algèbre suivante $D= \textup{End}_{G_1}(\pi_1) \otimes_R \textup{End}_{G_2}(\pi_2)$ est semi-simple. Comme elle est de dimension finie sur $R$ par admissibilité, il suffit de prouver que son centre est réduit. Or, son centre est $E_1 \otimes_R E_2$ en notant $E_1$ et $E_2$ respectivement les centres de ces algèbres d'endomorphismes. Comme $R$ est parfait, ces extensions finies $E_1$ et $E_2$ sont séparables, ce qui assure que le centre est réduit. 

b) Pour le deuxième point, soit $W \in \text{Rep}_R(G_1 \times G_2)$ une représentation irréductible admissible. Soit $K=K_1 \times K_2$, où $K_i$ est un sous-groupe compact ouvert de $G_i$ de pro-ordre inversible dans $R$ et tel que $W^K \neq 0$. L'espace $W^K$ est de dimension finie, et irréductible en tant que $\mathcal{H}_R(G_1 \times G_2,K)$-module, ce qui permet d'appliquer les résultats de \cite[A VIII.208]{bou}. Il existe donc un couple, unique à isomorphisme près, de $\mathcal{H}_R(G_i,K_i)$-modules simples $W_i^{K_i}$ et un morphisme non nul (donc surjectif \dots)  $a_K : W_1^{K_1} \otimes_R W_2^{K_2} \to W^K$ de $\mathcal{H}_R(G_1 \times G_2,K)$-modules de sorte que $W^K$ soit un quotient $W_1^{K_1} \otimes W_2^{K_2}$. De même pour tout sous-groupe ouvert $K'=K_1' \times K_2'$ inclus dans $K$, on peut construire un morphisme $a_{K'} : W_1^{K_1'} \otimes_R W_2^{K_2} \to W^{K'}$ pour deux modules simples $W_1^{K_1'}$ et $W_2^{K_2'}$. Il existe alors un morphisme :
$$b_i(K,K') \in \textup{End}_{\mathcal{H}_R(G_1,K_1) \otimes \mathcal{H}_R(G_2,K_2)} (W_1^{K_1} \otimes W_2^{K_2} , W_1^{K_1'} \otimes W_2^{K_2'})$$
tel que le diagramme suivant commute : 
$$\xymatrix{
	W^K  \ar@{^{(}->}[d]^{\textup{incl.}} & \ar[l]_{a_K \hspace{0.5cm}}  W_1^{K_1} \otimes W_2^{K_2} \ar[d]^{b_{K,K'}} \\
	W^{K'}  		& \ar[l]_{a_{K'} \hspace{0.5cm}}	W_1^{K_1'} \otimes W_2^{K_2'}	
	}$$
	
\noindent De plus, les morphismes $b_{K,K'}$ peuvent être choisis de sorte que les sous-groupes compacts ouverts précédents $K$, $K'$ forment un système inductif. Alors $W$ est un quotient de la représentation $W_1 \otimes_R W_2$ où :
$$ W_i = \underset{\underset{K_i}{\to}}{\lim} W_i^{K_i} \in \textup{Rep}_R(G_i) \textup{ est irréductible admissible}.$$ La classe de $W_i$ est déterminée par celle de $W$. En effet, deux représentations $W$ et $W'$ sont isomorphes si et seulement si pour tout compact ouvert $K$ de pro-ordre inversible $W^K \simeq (W')^K$. La restriction de $W$ à $G_i$ est donc $W_i$-isotypique car $W_i^{K_i}$ est unique à isomorphisme d'après \cite[A VIII.208]{bou}. \end{proof}

On peut donner une condition un peu plus fine pour qu'un produit tensoriel de représentations irréductibles soit irréductible.

\begin{lem} \label{produit_de_rep_irred_End_lem} On suppose que les représentations $\pi_1$ et $\pi_2$ sont irréductibles et qu'il existe un morphisme non nul de $R$-algèbres $D_2 \to D_1^\textup{op}$. Alors $\pi_1 \otimes_{D_2} \pi_2 \in \textup{Rep}_R(G_1 \times G_2)$ est une représentation irréductible. \end{lem}

\begin{proof} Soit $v = \sum v_1^i \otimes_{D_2} v_2^i \in \pi_1 \otimes_{D_2} \pi_2$. Comme $D_2= \textup{End}_{G_2}(\pi_2)$ est un corps, on peut choisir une famille finie $v_2^i$ de sorte que celle-ci soit libre sur $D_2$ et qu'on ait $v = \sum v_1^i \otimes_{D_2} v_2^i$ avec $v_1^i \neq 0$. On va montrer que $V_1 \otimes_{D_2} v_2^1$ est contenu dans la sous-représentation engendrée par $v$, ce qui suffit pour prouver que $v$ engendre $\pi_1 \otimes_{D_2} \pi_2$. Pour ce faire, on choisit un élément $f_2 \in R[G_2]$ de sorte que :
$$\pi_2(f_2) v_2^i = \left\{ \begin{array}{cc} 
0 & \text{ si } i\neq 1 \\
v_2^1 & \text{ si } i=1 \end{array} \right. .$$
Un tel élément existe toujours. En effet, on note $A$ l'image de $R[G_2] \to \textup{End}_R(\pi_2)$.  Alors $A$ est contenue dans $\textup{End}_{\textup{End}_{G_2}(\pi_2)}(\pi_2) = \textup{End}_{\textup{comm}(A)}(\pi_2) = \textup{comm}(\textup{comm}(A))$ où \og $\textup{comm}$ \fg{} désigne le commutant d'une algèbre dans $\textup{End}_R(\pi_2)$. On considère le sous-espace vectoriel $W_2$ de $\pi_2$ engendré par les $v_2^i$. L'application que l'on veut interpoler est la projection sur $v_2^1$ parallèlement aux autres $v_2^i$. En particulier, il existe un élément $b \in \textup{End}_{D_2}(\pi_2) = \textup{comm}(\textup{comm}(A))$ dont la restriction $b|_{W_2}$ à $W_2$ réalise cette projection. Comme la famille $v_2^i$ est finie, le Théorème de Densité \cite[Chap. I, B.6]{vig} assure qu'il existe $a \in A$ tel que $a|_{W_2} = b|_{W_2}$. D'où l'existence de $f_2$. On en déduit donc que $V_1 \otimes_{D_2} v_2^1$ est contenu dans la sous-représentation engendrée par $v$. \end{proof}

Les deux lemmes suivants généralisent \cite[Chap. 2, Lem. III.3 \& Lem. III.4]{mvw}. Les modifications apportées aux preuves sont mineures, mais on les reprend en détail pour insister sur ces légers changements.

\begin{lem} \label{pi-coinvariant-sous-rep} Soit $(\pi_1,V_1) \in \textup{Rep}_R(G_1)$ une représentation irréductible et admissible. Soit $(\pi_2,V_2) \in \textup{Rep}_R(G_2)$. On suppose que $V_2$ est munie d'une structure de module à droite sur $D_1$ compatible à l'action de $G_2$, cela signifie qu'on se donne un morphisme non nul :
$$D_1 \to D_2^{\textup{op}}.$$
Alors, pour toute sous-représentation $V \in \textup{Rep}_R(G_1 \times G_2)$ dans $V_2 \otimes_{D_1} V_1$, il existe une sous-représentation $V_2'$ de $V_2$, munie d'une structure de $D_1$-module à droite, telle que $V = V_2' \otimes_{D_1} V_1$ dans $\textup{Rep}_R(G_1 \times G_2)$. \end{lem}

\begin{proof} Posons $V_2'=\{ v_2 \in V_2 \ | \ \forall v_1 \in V_1, v_2 \otimes v_1 \in V \}$. Cet espace est stable par $G_2$ et par l'action à droite de $D_1$. La représentation $V_2' \otimes V_1$ est une sous-représentation de $V$. On va montrer que $V$ est contenu dans $V_2' \otimes V_1$. Soit $v \in V$. On emploie le même argument que celui de la preuve du Lemme \ref{produit_de_rep_irred_End_lem}. On le rappelle en raccourci. Par définition, il existe une famille finie $(v_1^i)$ libre sur le corps $D_1$ telle que $v=\sum_i v_2^i \otimes v_1^i$ et $v_2^i \neq 0$. On peut alors trouver un élément dans $f_1^i \in R[G_1]$ tel que :
$$\pi_1(f_1^i) v_1^j = \left\{ \begin{array}{cc} 
0 & \text{ si } j\neq i \\
v_1^i & \text{ si } j=i \end{array} \right. .$$ 
Ainsi, $v_2^i \in V_2'$ pour tout $i$, donc $v = \sum v_2^i \otimes v_1^i \in V_2'\otimes V_1$. \end{proof}

Quand le corps $R$ est algébriquement, les résultats se simplifient considérablement puisque les anneaux d'endomorphismes $D_1$ et $D_2$ sont égaux à $R$. L'expression de $V_2$ dans le résultat suivant se simplifie également en $(V\otimes V_1^\vee)_{1_{G_1}}$ qui correspond bien à une généralisation du plus grand quotient $\pi_1$-isotypique dans le cas complexe.

\begin{lem} \label{pi-coinvariant} Soit $(\pi,V) \in \emph{Rep}_R(G_1 \times G_2)$. Soit $(\pi_1,V_1) \in \textup{Rep}_R(G_1)$ irréductible admissible.
\begin{itemize}[label=$\bullet$]
\item On définit une sous-représentation de $V$ par :
$$V[\pi_1]= \bigcap_{f \in \emph{Hom}_{G_1}(V,V_1)} \emph{Ker} (f) \in \textup{Rep}_R(G_1 \times G_2).$$
On appelle \emph{\og plus grand quotient $\pi_1$-isotypique de $V$ \fg{}} la représentation :
$$V_{\pi_1} = V / V[\pi_1] \in \textup{Rep}_R(G_1 \times G_2).$$ \item Il existe alors un $\mathcal{H}_R(G_2)-D_1$-bimodule $(\pi_2,V_2)$, unique à isomorphisme de bimodule près, tel que :
$$V_{\pi_1} \simeq \pi_2 \otimes_{D_1} \pi_1.$$
De plus, on a un isomorphisme de $\mathcal{H}_R(G_2) - D_1$-bimodule :
$$V_2 \simeq (V \otimes_R \textup{Hom}_{D_1}(V_1,D_1)^{\infty})_{1_{G_1}}.$$ \end{itemize} \end{lem}

\begin{proof} Le groupe $G_2$ agit de manière évidente sur $\textup{Hom}_{G_1}(V,V_1)$, donc sur $V[\pi_1]$ \textit{a fortiori}. On a bien $V[\pi_1]$ et $V_{\pi_1}$ qui sont des représentations de $G_1 \times G_2$. Pour l'unicité de $V_2$, on va montrer que si $V_2$ existe, alors $V_2 \simeq (V \otimes_R \pi_1' )_{1_{G_1}}$ où l'on a pris les coinvariants vis-à-vis de $G_1$ et la représentation $\pi_1' =\textup{Hom}_{D_1}(V_1,D_1)^{\infty}$ est non réduite à $0$. La structure de $D_1$-module à droite sur $(V \otimes_R \pi_1')_{1_{G_1}}$ est héritée de celle sur $\pi_1'$. On suppose donc qu'on a un isomorphisme $V_{\pi_1} \simeq V_2 \otimes_{D_1} V_1$ dans $\textup{Rep}_R(G_1 \times G_2)$. Le morphisme quotient $V \to V_{\pi_1}$ induit un isomorphisme $(V \otimes_R \pi_1')_{1_{G_1}} \simeq (V_{\pi_1} \otimes_R \pi_1')_{1_{G_1}}$ puisque $V[\pi_1] \otimes_R \pi_1'$ est inclus dans $(V_{\pi_1} \otimes_R \pi_1') [1_{G_1}]$. D'une part, comme $\pi_1$ est admissible, on peut montrer que $(\pi_1')' \simeq \pi_1$ en tant que ($\mathcal{H}_R(G_1) \otimes_R D_1$)-module. D'autre part, pour tout $f \in \textup{Hom}_{G_1} (V \otimes_R \pi_1' , 1_{G_1})$, on a  $v \in V \mapsto (v_1' \mapsto f(v \otimes_R v_1')) \in V_1$ qui appartient à $\textup{Hom}_{G_1}(V,\pi_1)$. Donc si $v \in V[\pi_1]$, pour tout $v_1' \in V_1'$, on obtient bien que $f(v \otimes_R v_1')=0$. Maintenant, si $V_{\pi_1} \simeq \pi_2 \otimes_{D_1} \pi_1$, alors :
$$(V \otimes_R \pi_1')_{1_{G_1}} = (V_{\pi_1} \otimes_R \pi_1')_{1_{G_1}} = V_2 \otimes_{D_1} (\pi_1 \otimes_R \pi_1')_{1_{G_1}} = V_2$$
la dernière égalité résulte d'un argument classique car $(\pi_1 \otimes_R \pi_1')_{1_{G_1}}$ est un $D_1$-module à gauche de dimension $1$. En effet, en procédant comme dans \cite[III.1.9]{ren}, on un isomorphisme de $R$-espaces vectoriels  $\textup{Hom}_{\mathcal{H}_R(G_1) \otimes_R D_1}(\pi_1 \otimes_R \pi_1',1_{G_1,D_1}) \simeq \textup{End}_{G_1}(\pi_1')$ où $1_{G_1,D_1}$ désigne la \og représentation triviale \fg{} \textit{i.e.} un $D_1$-espace vectoriel de dimension $1$ muni de l'action triviale de $G_1$. Comme $\textup{End}_{G_1}(\pi_1')$ est un $D_1$-espace vectoriel de dimension 1, le dual de $(\pi_1 \otimes_R \pi_1')_{1_{G_1}}$ est de dimension $1$ sur $D_1$. D'où le fait que $(\pi_1 \otimes \pi_1')_{1_{G_1}}$ soit de dimension $1$.

En ce qui concerne l'existence de $V_2$, il suffit de montrer que $V_{\pi_1}$ est une sous-représentation d'une représentation de la forme $V_2' \otimes_{D_1} \pi_1$ et utiliser le lemme précédent. On pose $V_2 = (V \otimes_R \pi_1')_{1_{G_1}}$ et on désigne par $p$ le morphisme quotient $V \otimes_R \pi_1' \to V_2$ associé aux coinvariants vis-à-vis de $G_1$. Soit $\phi$ l'application $R$-linéaire :
$$\phi : v \in V \mapsto (\phi(v) : v_1' \mapsto p(v \otimes_R v_1')) \in \textup{Hom}_{D_1}(V_1', V_2).$$
Elle est un morphisme de représentations de $G_1 \times G_2$ et se factorise par $V_{\pi_1}$. On va montrer que l'image de $\phi(v)$ est de dimension finie. Soient $v \in V$ et $K$ un sous-groupe ouvert compact de $G_1$ de pro-ordre inversible dans $R$ fixant $v$. Soit $e_K$ l'idempotent associé dans l'algèbre de Hecke $\mathcal{H}_R(G_1)$. Pour $v_1' \in V_1'$, on a :
$$\phi(v)(v_1')=p(v \otimes_R v_1') = p(\pi(e_K) v \otimes_R v_1' ) = p(v \otimes_R \pi_1'(\check{e_K}) v_1 )$$
où $\check{e_K}$ est l'image de $e_K$ par l'automorphisme $g \mapsto g^{-1}$. Mais $\check{e_K}=e_K$, ce qui entraîne que $\phi(v)(v_1') = \phi(v)(\pi_1'(\check{e_K})v_1')$. Autrement dit $\phi(v)$ se factorise par $\pi_1'(e_K)$. On a un plongement naturel $V_2 \otimes_{D_1} V_1 \rightarrow \textup{Hom}_{D_1}(V_1',V_2)$. L'admissibilité de $\pi_1$ implique que son image est le sous-espace des $f \in \textup{Hom}_{D_1}(V_1',V_2)$ tel qu'il existe un sous-groupe ouvert compact $K$ de $G_1$ de pro-ordre inversible pour lequel $f$ se factorise par $\pi_1'(e_K)$. Alors $\phi$ se factorise par $\phi' : V_{\pi_1} \rightarrow V_2 \otimes_{D_1} V_1$.

Il reste à montrer que $\phi'$ est injective pour conclure. Soit $ v \in V_{\pi_1}$, $v \neq 0$. Il existe par hypothèse $f \in \textup{Hom}_{G_1}(V_{\pi_1},\pi_1)$ tel que $f(v) \neq 0$. Fixons un tel $f$ et $v_1' \in V_1'$ tels que $v_1' \circ f(v) \neq 0$. Par fonctorialité, $f$ définit un morphisme de $D_1$-modules à droite :
$$f' : (V \otimes_R V_1')_{G_1} \rightarrow (V_1 \otimes_R V_1')_{1_{G_1}}.$$
En composant avec $h : v_1 \otimes_R v_1' \in V_1 \otimes_R V_1' \mapsto v_1'(v_1) \in 1_{G_1,D_1}$, qui se factorise par $(V_1 \otimes V_1')_{1_{G_1}}$, on obtient que $h \circ f' \circ p(v \otimes_R v_1') = v_1' \circ f(v) \neq 0$ et donc $p(v \otimes_R v_1') \neq 0$. Par conséquent $\phi(v  + V[\pi_1]) \neq 0$ et $\phi'$ est injective. Comme énoncé plus haut, $V_{\pi_1}$ est une sous-représentation de $V_2 \otimes_{D_1} V_1$ et le lemme précédent donne l'existence de $(\pi_2,V_2)$. \end{proof}

\begin{cor} \label{pi1_coinvariants_famille_pi_1i_cor} Soit $(\pi_{1,i})_{i \in I}$ une famille finie de représentations irréductibles deux à deux non isomorphes. Pour $V \in \textup{Rep}_R(G_1 \times G_2)$, on note $p_{\pi_{1,i}}$ la projection $V \mapsto V_{\pi_{1,i}}$. Alors l'application :
$$p_I : v \in V \mapsto (p_{\pi_{1,i}}(v) )_{i\in I} \in \oplus_{i\in I} V_{\pi_1^i}$$
est surjective de noyau :
$$\bigcap_{i \in I} V[\pi_{1,i}] = \bigcap_{f \in \textup{Hom}_{G_1}(V, \oplus_{i \in I} \pi_{1,i})} \textup{Ker} (f).$$ \end{cor}

\begin{proof} Tout d'abord, il est clair que $\cap_{i \in I} V[ \pi_{1,i}]$ est le noyau de l'application de l'énoncé puisque le noyau de $p_{\pi_{1,i}}$ est $V[\pi_{1,i}]$ par définition. Ensuite, on voit facilement que la dernière égalité est vraie étant donné que $\textup{Hom}_{G_1}(V,\oplus_i \pi_{1,i}) = \prod_i \textup{Hom}_{G_1}(V,\pi_{1,i})$.

Il reste donc à prouver que l'application ainsi définie est surjective. On le montre à l'aide d'une récurrence finie. Soit $i_0 \in I$. Par définition $V_{\pi_{1,i_0}}  = V / V[\pi_{1,i_0}]$ est une représentation $\pi_{1,i_0}$-isotypique dans $\textup{Rep}_R(G_1)$. Pour tout $i \in I$ avec $i \neq i_0$, on a :
$$V_{\pi_{1,i}} = (V[\pi_{1,i_0}])_{\pi_{1,i}}.$$
En effet on a des inclusions $V[\pi_{1,i}]\cap V[\pi_{1,i_0}] \subset  V[\pi_{1,i_0}]  \subset V$, où le dernier quotient $V/V[\pi_{1,i_0}]$ est $\pi_{1,i_0}$-isotypique et le premier est $\pi_{1,i}$-isotypique. Par conséquent $V$ se surjecte sur $V_{\pi_{1,i}} \oplus V_{\pi_{1,i_0}}$ avec $V_{\pi_{1,i_0}} \simeq V[\pi_{1,i_0}] / V[\pi_{1,i}]\cap V[\pi_{1,i_0}]$. Une récurrence finie montre qu'il en de même pour $|I|$ représentations deux à deux non isomorphes : si l'on considère par exemple $V[\pi_{1,i_0}] \to \oplus_{i \in I \backslash {i_0}} V_{\pi_{1,i}}$ au lieu de l'application de l'énoncé, les images de $V$ et de $V[\pi_{1,i_0}]$ dans $\oplus_{i \in I \backslash {i_0}} V_{\pi_{1,i}}$ sont égales, et on peut recommencer pour un autre indice $i_1$ de $I \backslash \{ i_0 \}$. Comme $I$ est fini, on en déduit que l'application de l'énoncé est surjective. \end{proof}

\paragraph{Comportement vis-à-vis de l'extension des scalaires.}  On fixe une clôture algébrique $\bar{R}$ de $R$. Pour toute représentation irréductible admissible $\pi_1$ dans $\textup{Rep}_R(G_1)$, et quand le corps $R$ est parfait, le Théorème \ref{decomposition_extension_des_scalaires_thm} assure qu'il existe $\rho_1 \in \textup{Rep}_{\bar{R}}(G_1)$ irréductible (admissible) telle que :
$$\pi_1 \otimes_R \bar{R} \simeq m_1 \bigg( \bigoplus_{w \in \textup{Gal}_R(E(\rho_1),\bar{R})} w \rho_1 \bigg).$$
Soit $V \in \textup{Rep}_R(G_1 \times G_2)$. Peut-on relier l'espace des $\pi_1$-coinvariants $V_{\pi_1}$ avec l'ensemble des $w \rho_1$-coninvariants $V_{w \rho_1}$ où $w \in \textup{Gal}_R(E(\rho_1),\bar{R})$ ? La réponse est apportée par le théorème suivant.

\begin{theo} \label{extension_des_scalaires_coinvariants_last_thm} On suppose que $R$ est un corps parfait. Soit $(\pi_1,V_1)$ une représentation irréductible admissible dans $\textup{Rep}_R(G_1)$. On considère la décomposition précédente de $\pi_1 \otimes_R \bar{R}$ donnée par le Théorème \ref{decomposition_extension_des_scalaires_thm}. Alors pour tout $V \in \textup{Rep}_R(G_1 \times G_2)$, on a :
$$V_{\pi_1} \otimes_R \bar{R} \simeq \bigoplus_{w \in \textup{Gal}_R(E(\rho_1),\bar{R})} (V \otimes_R \bar{R})_{w\rho_1}.$$ \end{theo}

\begin{proof} D'après le Corollaire \ref{pi1_coinvariants_famille_pi_1i_cor}, l'application :
$$V \otimes_R \bar{R} \to \bigoplus_{w \in \textup{Gal}_R(E(\rho_1),\bar{R})} (V\otimes_R \bar{R})_{w \rho_1}$$
est surjective et a pour noyau $\bigcap_w (V \otimes_R \bar{R}) [w \rho_1]$. On va montrer que :
$$V[\pi_1] \otimes_R \bar{R} = \bigcap_{w \in \textup{Gal}_R(E(\rho_1),\bar{R})} (V \otimes_R \bar{R})[w \rho_1].$$
Cette dernière égalité entraînera alors que $(V \otimes_R \bar{R})/ (V[\pi_1] \otimes_R \bar{R}) \simeq V_{\pi_1} \otimes_R \bar{R}$ est le plus grand quotient isotypique de $V$ associé, au sens du Corollaire \ref{pi1_coinvariants_famille_pi_1i_cor}, à la famille finie $(w \rho_1)_{w \in \textup{Gal}_R(E(\rho_1),\bar{R})}$.

L'inclusion directe demande le moins d'effort. Soit $v \in V[\pi_1]$. Il s'agit de montrer que pour tout $w \in \textup{Gal}_R(E(\rho_1),\bar{R})$ et tout $f \in \textup{Hom}_{G_1}(V \otimes_R \bar{R},w \rho_1)$, on a $v \otimes_R 1 \in \textup{Ker}(f)$. En particulier, un tel $f$ définit un morphisme de $R[G_1]$-modules $\textup{Res}^R(V \otimes_R \bar{R}) \to \textup{Res}^R (w \rho_1)$ par oubli des scalaires. De plus, le morphisme $f$ est non nul si et seulement si sa restriction à $V\otimes_R 1 = \{ v \otimes_R 1 \ | \ v \in V\}$ est non nulle. Or, comme dans les premières lignes de la preuve du Lemme \ref{unicite_restriction_des_scalaires_induction_lem}, la représentation $\textup{Res}^R(w \rho_1)$ est $\pi_1$-isotypique. Par conséquent $f|_{V \otimes_R 1} : V \simeq V \otimes_R 1 \to \textup{Res}^R(w \rho_1) \simeq \oplus \pi_1$. Donc $f(v \otimes_R 1) =0$ par définition de $V[\pi_1]$.

En ce qui concerne le sens indirect, on sait d'après le Lemme \ref{pi-coinvariant} qu'il existe un $R[G_2]-D_1$-bimodule lisse $V_2$ tel que  $V_{\pi_1} \simeq V_2 \otimes_{D_1} V_1$ où $D_1 = \textup{End}_{G_1}(\pi_1)$ est une algèbre à division de degré $m_1$ sur son centre. De plus, on a un isomorphisme de représentations :
$$V_{\pi_1} \otimes_R \bar{R} \simeq (V_2 \otimes_R \bar{R}) \otimes_{D_1 \otimes_R \bar{R}} (V_1 \otimes_R \bar{R}).$$
D'après le Théorème \ref{decomposition_extension_des_scalaires_thm}, l'anneau $D_1 \otimes_R \bar{R}$ est isomorphe à $\prod_w e_w D \simeq \prod_w M_{m_1}(\bar{R})$ où $(e_w)_{w \in \textup{Gal}_R(E_1,\bar{R})}$ est un système d'idempotents centraux primitifs dans $D_1 \otimes_R \bar{R}$. On en déduit que :
$$V_{\pi_1} \otimes_R \bar{R} \simeq \bigoplus V_w \otimes_{e_w D} (m_1 (w \rho_1))$$
où $m_1 w \rho_1 = e_w (\pi_1 \otimes_R \bar{R})$ et $V_w = V_2 e_w$.

On montre ensuite qu'il existe une représentation $V_{2,w} \in \textup{Rep}_{\bar{R}}(G_2)$ telle que :
$$V_w \otimes_{e_w D} (m_1 (w \rho_1)) \simeq V_{2,w} \otimes_{\bar{R}} (w \rho_1).$$
En effet $e_w D \simeq M_{m_1}(\bar{R})$, et en notant $e_{i,j}$ la matrice élémentaire dans $M_{m_1}(\bar{R})$, on a $e_{1,1} + \dots + e_{m_1,m_1} = \textup{Id}_{m_1}$ qui est une décomposition de l'identité en somme d'idempotents. Ils ne sont cependant pas centraux dans $M_{m_1}(\bar{R})$. Néanmoins, chaque $e_{i,i}$ définit une application $v \in V_w \mapsto v e_{i,i} \in V_w$ qui est un morphisme de $\bar{R}[G_2]$-modules. De plus $e_{i,i} V_w \simeq e_{1,1} V_w$ pour tout $i$. En notant $V_{2,w} = V_w e_{1,1} \in \textup{Rep}_{\bar{R}}(G_2)$, on a $V_w \simeq m_1 V_{2,w}$ et $(m_1 V_{2,w}) \otimes_{M_{m_1}(\bar{R})} (m_1 (w \rho_1)) \simeq V_{2,w} \otimes_{\bar{R}} (w \rho_1)$. Ainsi le quotient $V \otimes_R \bar{R} \to V_{2,w} \otimes_{\bar{R}} (w \rho_1)$ se factorise par $(V \otimes_R \bar{R})_{w \rho_1}$ par définition du plus grand quotient $w \rho_1$-isotypique. Donc le quotient $V \otimes_R \bar{R} \to V_{\pi_1} \otimes_R \bar{R} \simeq \oplus_w V_{2,w} \otimes_{\bar{R}} (w \rho_1)$ se factorise par $\oplus_w (V \otimes_R \bar{R})_{w \rho_1}$. En d'autres termes $\cap_w (V\otimes_R \bar{R})_{w \rho_1} \subset V[\pi_1] \otimes_R \bar{R}$. \end{proof}

\bibliographystyle{alpha}
\bibliography{lesrefer}

\end{document}